\documentclass{article}
\pdfoutput=1
\usepackage{amsthm}
\usepackage{amsmath}
\usepackage[dvipsnames]{xcolor}
\usepackage[utf8]{inputenc}
\usepackage[colorlinks = true]{hyperref}
\hypersetup{citecolor=bluish,
linkcolor = redish,
urlcolor=redish}
\usepackage{etex}
\usepackage[shortlabels]{enumitem}
\usepackage{subfiles}
\usepackage{lmodern}
\usepackage{titlesec}
\usepackage{wrapfig}
\usepackage{csquotes}
\usepackage[T1]{fontenc}
\setenumerate[1]{label={(\arabic*)}}
\usepackage{rotating}

\DeclareFontFamily{U}{min}{}

\DeclareFontShape{U}{min}{m}{n}{<-> udmj30}{}

\usepackage{upgreek}
\usepackage{tikz-cd}
\usepackage{mathrsfs}
\usepackage{sseq}
\usepackage{mathdots}
\usepackage{thm-restate}
\usepackage{amsmath}
\usepackage{amsxtra}
\usepackage{amscd}
\usepackage{amsthm}
\usepackage{amsfonts}
\usepackage{amssymb}
\usepackage{mathtools}
\usepackage[scr=rsfs,bb=stix]{mathalpha}
\usepackage{eucal}
\usepackage{bbm}
\usepackage[all,cmtip]{xy}

\usepackage{graphicx}
\usepackage{tikz}
\usepackage{morefloats}
\usepackage{pdfpages}
\usetikzlibrary{matrix,arrows,decorations.pathmorphing,quotes,angles,decorations.markings,shapes.misc}
\RequirePackage{color}
\definecolor{lavender}{rgb}{0.59, 0.44, 0.84}
\definecolor{redish}{rgb}{0.50, 0, 0}
\definecolor{bluish}{rgb}{0, 0, 0.50}
\usepackage{microtype}
\usepackage{epigraph}
\usepackage{tocloft}
\usepackage[colorinlistoftodos, textsize=tiny]{todonotes}

\usepackage[headsep=1cm,headheight=2cm,margin=1in]{geometry}
\usepackage[framemethod=TikZ]{mdframed}
\usepackage{setspace}

\renewcommand{\sf}[1]{\mathsf{#1}}
\renewcommand{\c}[1]{\mathcal{#1}}
\renewcommand{\b}[1]{\mathbb{#1}}
\newcommand{\un}[1]{\underline{#1}}

\newcommand{\ec}[1]{\underline{\c{#1}}}

\newcommand{\Fun}{\sf{Fun}}

\newcommand{\exact}[1]{(\c{#1},\c{#1}^{\ingr},\c{#1}^{\egr})}

\newcommand{\cC}{\c{C}}
\newcommand{\cD}{\c{D}}

\newcommand{\cS}{\c{S}}

\newcommand{\sD}{\sf{D}}

\newcommand{\bK}{\b{K}}

\newcommand{\eC}{\ec{C}}

\DeclareMathOperator{\Hom}{Hom}

\renewcommand{\lim}{\sf{lim}}

\newcommand{\Cat}{\sf{Cat}}

\DeclareMathOperator{\Set}{\sf{Set}}
\DeclareMathOperator{\lex}{\sf{lex}}

\DeclareMathOperator{\LMod}{LMod}

\DeclareMathOperator{\Span}{\sf{Span}}

\newcommand{\del}{\partial}

\newcommand{\Map}{\sf{Map}}
\newcommand{\Alg}{\sf{Alg}}

\newcommand{\Seg}{\sf{Seg}}

\newcommand{\RMod}{\sf{RMod}}

\newcommand{\BMod}{\sf{BMod}}

\newcommand{\Tw}{\sf{Tw}}
\newcommand{\TSegSpan}{\operatorname{2Seg_{\Delta}^{\leftrightarrow}}}

\newcommand{\rel}{\sf{rel}}
\newcommand{\rev}{\operatorname{rev}}

\def\triangle[#1]#2[#3]#4[#5]#6{
        \begin{tikzpicture}[auto]
                \node (A) at (0,0) {$#1$};
                \node (B) at (4,0) {$#2$};
                \node (C) at (2,-1) {$#3$};
               
                \draw[->] (A) to node[scale=0.7] {$#4$} (B);
                \draw[->] (A) to node[scale=0.7, anchor=north east] {$#5$} (C);
                \draw[->] (B) to node[scale=0.7] {$#6$} (C);

        \end{tikzpicture}}

\def\comspan[#1]#2[#3]#4[#5]#6[#7]#8[#9]{
\begin{tikzpicture}
        \node (A) at (0,0) {$#1$};
        \node (B) at (2,-1) {$#5$};
        \node (C) at (-2,-1) {$#3$};
        \node (D) at (4,-2) {$#7$};
        \node (E) at (0,-2) {$#8$};
        \node (F) at (-4,-2) {$#9$};
       
        \draw[-to] (A) --   (B);
        \draw[-to] (A) --  (C);
        \draw[-to] (B) --  (D);
        \draw[-to] (B) -- node[anchor = south east, scale = 0.7] {$#2$} (E);
        \draw[-to] (C) -- node[anchor = south west, scale = 0.7] {$#4$} (E);
        \draw[-to] (C) -- node[anchor = south east, scale = 0.7] {$#6$} (F);
       
        \draw[draw=none] (-.5*\textwidth,0) -- (.5*\textwidth,0); 
    \end{tikzpicture}}

\def\spantikz[#1]#2[#3]#4[#5]{
        \begin{tikzpicture}
                \node (A) at (0,0) {$#1$};
                \node (B) at (2,-2) {$#5$};
                \node (C) at (-2,-2) {$#3$};
               
                \draw[-to] (A) -- node[anchor = south west, scale = 0.7] {$#2$} (B);
                \draw[-to] (A) -- node[anchor = south east, scale = 0.7] {$#4$} (C);
               
                \draw[draw=none] (-.5*\textwidth,0) -- (.5*\textwidth,0); 
out this line to stop centering around the middle node
        \end{tikzpicture}}

\def\widespantikz[#1]#2[#3]#4[#5]{
        \begin{tikzpicture}
                \node (A) at (0,0) {$#1$};
                \node (B) at (3,-2) {$#5$};
                \node (C) at (-3,-2) {$#3$};
               
                \draw[-to] (A) -- node[anchor = south west, scale = 0.7] {$#2$} (B);
                \draw[-to] (A) -- node[anchor = south east, scale = 0.7] {$#4$} (C);
               
                \draw[draw=none] (-.5*\textwidth,0) -- (.5*\textwidth,0); 
out this line to stop centering around the middle node
        \end{tikzpicture}}

\def\doublewedge[#1]#2[#3]#4[#5]#6[#7]#8[#9]{
 \begin{tikzpicture}
                \node (A) at (0,0) {$#1$};
                \node (B) at (3,-2) {$#5$};
                \node (C) at (-3,-2) {$#3$};
                \node (D) at (-6,0) {$#7$};
                \node (E) at (6,0) {$#9$};

                \draw[-to] (A) -- node[anchor = south west, scale = 0.7] {$#2$} (B);
                \draw[-to] (A) -- node[anchor = south east, scale = 0.7] {$#4$} (C);
                \draw[-to] (D) -- node[anchor = south west, scale = 0.7] {$#6$} (C);
                \draw[-to] (E) -- node[anchor = south east, scale = 0.7] {$#8$} (B);
               
                \draw[draw=none] (-.5*\textwidth,0) -- (.5*\textwidth,0); 
out this line to stop centering around the middle node
\end{tikzpicture}
}

\def\squarevert[#1]#2[#3]#4{
        \begin{tikzpicture}
                \node (A) at (0,0) {$#1$};
                \node (B) at (4,-1.5) {$#2$};
                \node (C) at (-4,-1.5) {$#3$};
                \node (D) at (0,-3) {$#4$};
               
                \draw[-to] (A) -- (B);
                \draw[-to] (A) -- (C);
                \draw[-to] (B) --  (D);
                \draw[-to](C) --  (D);
               
                \draw[draw=none] (-.5*\textwidth,0) -- (.5*\textwidth,0); 
out this line to stop centering around the middle node
        \end{tikzpicture}}

\DeclareMathOperator{\op}{op}

\DeclareMathOperator{\id}{id}

\DeclareMathOperator{\tri}{\triangleright}
\DeclareMathOperator{\tle}{\triangleleft}

\DeclareMathOperator{\ev}{ev}

\DeclareMathOperator{\ingr}{in}
\DeclareMathOperator{\egr}{eg}

\DeclareMathOperator{\act}{act}
\DeclareMathOperator{\inert}{int}
\newcommand{\dspace}{birelative decomposition space}

\DeclareMathOperator{\bim}{bim}

\newtheorem{theorem}{Theorem}[section]
\newtheorem*{theorem*}{Theorem}
\newtheorem*{question*}{Question}
\newtheorem{prop}[theorem]{Proposition}
\newtheorem{lemma}[theorem]{Lemma}
\newtheorem{cor}[theorem]{Corollary}

\theoremstyle{definition}
\newtheorem{definition}[theorem]{Definition}
\newtheorem{remark}[theorem]{Remark}
\newtheorem{example}[theorem]{Example}

\newtheorem*{claim*}{Claim}

\newtheorem{construction}[theorem]{Construction}
\newtheorem{notation}[theorem]{Notation}

\newcommand{\tsegspan}[1]{\operatorname{2-Seg_{\Delta}^{\leftrightarrow}(#1)}}

\newcommand{\bisegspan}[1]{\operatorname{BiSeg^{\leftrightarrow}_{\Delta}(#1)}}

\usepackage{subfiles}

\begin{document}

\title{An $\infty$-Category of 2-Segal Spaces}
\author{Jonte Gödicke\footnote{Max-Planck Institut f\"{u}r Mathematik, Vivatgasse 7, 53111 Bonn, Germany, email:math@jonte-goedicke.com or godicke@mpim-bonn.mpg.de}}

\maketitle

\begin{abstract}
   \noindent Algebra objects in $\infty$-categories of spans admit a description in terms of $2$-Segal objects. We introduce a notion of span between $2$-Segal objects and extend this correspondence to an equivalence of $\infty$-categories. Additionally, for every $\infty$-category with finite limits $\c{C}$, we introduce a notion of a birelative $2$-Segal object in $\c{C}$ and establish a similar equivalence with the $\infty$-category of bimodule objects in spans. Examples of these concepts arise from algebraic and hermitian K-theory through the corresponding Waldhausen $S_{\bullet}$-construction. Apart from their categorical relevance, these concepts can be used to construct homotopy coherent representations of Hall algebras.
\end{abstract}

\tableofcontents

\section{Introduction}

A $2$-Segal space, also known as a decomposition space \cite{galvez2018decomposition}, is a simplicial space $X_{\bullet}:\Delta^{\op}\rightarrow \cS$ that satisfies a $2$-dimensional analog of the famous Segal conditions introduced in \cite{segal}. These concepts were initially introduced independently by Dyckerhoff and Kapranov \cite{dyckerhoff2019higher} to describe the underlying structure of different Hall algebra constructions, and by Gálvez-Carrillo, Kock and Tonks to describe various coalgebra structures appearing in combinatorics \cite{galvez2018decomposition}. Since then, the subject has undergone substantial developments, including the introduction of even higher Segal conditions \cite{poguntke2017higher} to applications in different areas of mathematics like the theory of TFTs \cite{dyckerhoff2018triangulated}, Auslander-Reiten theory \cite{dyckerhoffausreiten} and Donaldson-Thomas theory \cite{porta}. \\
The connection between $2$-Segal spaces and Hall algebras arises from the algebraic interpretation of $2$-Segal spaces. In \cite{stern20212}, it is shown that the space of $2$-Segal spaces is equivalent to the space of algebra objects in the $\infty$-category of spans of spaces. However, the appropriate notion of morphism between $2$-Segal spaces, which under this equivalence corresponds to a morphism of algebras, has not been determined so far. \\
An initial proposal might be to define the corresponding notion of morphism between $2$-Segal spaces as a morphism of simplicial spaces and thus to define the $\infty$-category of $2$-Segal spaces as a full subcategory of $\Fun(\Delta^{\op},\cS)$. Although this yields an interesting notion of morphism between $2$-Segal spaces, it does not correspond to an algebra morphism. This fails for different reasons. On the one hand, a general morphism of simplicial spaces is only expected to correspond to a lax algebra morphism \cite[Sect.4.2]{walde2016hall}. On the other hand, since the algebraic structure itself lives in an $\infty$-category of spans an algebra morphism should also be formalized in this context. These considerations lead to the main result of this text:

\begin{theorem}\label{thm:one}
Let $\c{C}$ be an $\infty$-category with finite limits. There exists an equivalence of $\infty$-categories
    \[ 
    \tsegspan{\c{C}} \xrightarrow{\simeq} \Alg(\Span(\c{C}^{\times}))
    \]
    between the subcategory of the $\infty$-category $\Span(\Fun(\Delta^{\op},\c{C}))$ of spans of simplicial objects with objects $2$-Segal object and morphisms $2$-Segal spans and the $\infty$-category of algebra objects in $\Span(\c{C}^{\times})$.
\end{theorem}
\noindent A span between $2$-Segal objects is called a $2$-Segal span if the legs of the span satisfy conditions that were previously introduced under the names CULF and IKEO in \cite{galvez2018decomposition}. There, the authors show that such morphisms induce algebra morphisms. The novel aspect of Theorem~\ref{thm:one} is that every algebra morphism is in fact induced by spans of CULF and IKEO morphisms. \\
Analogous to the case of algebras, we introduce a multicolored version of the $2$-Segal conditions on birelative simplicial spaces $X_{\bullet}:\Delta^{op}_{/[1]} \rightarrow \cS$. Using this, we extend the previous result to the case of bimodules:
\begin{theorem}
    Let $\c{C}$ be an $\infty$-category with finite limits. There exists an equivalence of $\infty$-categories
    \[ 
    \bisegspan{\c{C}} \xrightarrow{\simeq} \BMod(\Span(\c{C}^{\times}))
    \]
    between the subcategory of $\Span(\Fun(\Delta^{\op}_{/[1]},\c{C}))$ with objects birelative $2$-Segal objects and with morphism birelative $2$-Segal spans and the $\infty$-category of bimodule objects in $\Span(\c{C}^{\times})$.
\end{theorem}
\noindent  Similar conditions have been previously studied for left modules in the context of $1$-categories in both \cite{walde2016hall} and \cite{young2018relative}, and for bimodules in terms of double Segal spaces in \cite{carlier2020incidence}. Our result generalizes these works into the context of $\infty$-categories.  \\
The prototypical example of a $2$-Segal space is the Waldhausen $S_{\bullet}$-construction of an exact $\infty$-category $\c{C}$. It is thus reasonable to expect exact $\infty$-functors $F:\c{C} \rightarrow \c{D}$ to induce algebra morphism between those. Indeed, we show that certain classes of exact $\infty$-functors induce algebra morphisms between Waldhausen $S_{\bullet}$-constructions. Other examples arise from the relative Waldhausen $S_{\bullet}$-construction of an exact $\infty$-functor $F$. This simplicial space fits into a sequence of simplicial spaces: 
\begin{equation}\label{eq:relativesequ}
\c{D}_{\bullet}^{\simeq} \xrightarrow{\iota_{\bullet}} S_{\bullet}^{rel}(F) \xrightarrow{\pi_{\bullet}} S_{\bullet}(\c{C})
\end{equation}
which has first appeared in \cite{waldhausen} to construct the long exact sequence of relative K-theory. We show that the morphisms $\iota_{\bullet}$ and $\pi_{\bullet}$ are indeed algebra morphisms. These morphisms underlie the construction of derived Hecke actions on cohomological Hall algebras \cite{kapranov2019cohomological}.   \\
For the construction of examples of modules, we consider $2$-Segal spaces $X_{\bullet}$ equipped with a twisted $\sf{C}_{2}$-action. Examples of these arise in hermitian K-theory as the Waldhausen $S_{\bullet}$-construction of an exact $\infty$-category with duality. For every such $2$-Segal space with a twisted $\sf{C}_{2}$ action, the canonical morphism $\Tw(X_{\bullet})\rightarrow X_{\bullet}\times X_{\bullet}^{rev}$ is a relative $2$-Segal object with $\sf{C}_{2}$-action. Taking homotopy fixed points yields:

\begin{prop}
There exists an $\infty$-functor from the $\infty$-category of $2$-Segal spaces with duality to the $\infty$-category of relative $2$-Segal spaces
\[
\operatorname{Tw^{\leftrightarrow}(-)^{\sf{C}_{2}}}: \operatorname{2-Seg_{\Delta}^{\leftrightarrow,\sf{C}_{2}}(\cS)} \rightarrow \operatorname{Rel2Seg^{\leftrightarrow}_{\Delta}(\cS)}
\]
that maps a simplicial space with duality $X_{\bullet}$ to the relative $2$-Segal space $\Tw(X_{\bullet})^{\sf{C}_{2}}\rightarrow X_{\bullet}$.
\end{prop}

\noindent In the example of the Waldhausen $S_{\bullet}$-construction of an exact $\infty$-category with duality $\c{C}$ the simplicial space $\Tw(S_{\bullet}(\c{C}))^{\sf{C}_{2}}$ is known as the $\c{R}_{\bullet}$-construction \cite{hornbostel2004localization} of $\c{C}$. Consequently, the above Proposition expresses $\c{R}_{\bullet}(\c{C})$ as a module over $S_{\bullet}(\c{C})$. This generalizes results of \cite{young2018relative} for proto-exact categories with duality into the $\infty$-categorical context. Moreover, using the functoriality of this construction, we obtain examples of module morphisms from exact duality preserving functors between exact $\infty$-categories. These results have applications to the construction of cohomological Hall algebra representations \cite{diaconescu2022cohomological} and orientifold Donaldson-Thomas-theory \cite{young2020representations}.

\subsection{Outline}

We conclude this introduction with an outline of the structure of this paper. In Section~\ref{sect:2-seg}, we define the $\infty$-categories $\bisegspan{\cS}$ and $\tsegspan{\cS}$. Therefore, we first recall in Subsection~\ref{subsect:2-Segal} the definition of $2$-Segal objects and introduce the notion of $2$-Segal spans. Then, in Subsection~\ref{subsect:birel2-Seg} we extend the definitions to the case of birelative simplicial objects and introduce the birelative $2$-Segal conditions. Similarly to the $2$-Segal case, we introduce a version of the decomposition space conditions in Subsection~\ref{subsect:decomposition}. We show that these are equivalent to the birelative $2$-Segal conditions. Finally, we relate in Subsection~\ref{subsect:carlier} our notion of birelative $2$-Segal spaces to Carlier's notion of bimodule configurations \cite{carlier2020incidence}.\\
In Section~\ref{sect:examples}, we discuss examples of the notions introduced in Section~\ref{sect:2-seg} that naturally appear in the context of algebraic and hermitian K-theory. To construct examples of modules, we construct an $\infty$-functor from the $\infty$-category of $2$-Segal spaces with duality to the $\infty$-category of relative $2$-Segal objects. As a main application of this, we express the $\c{R}_{\bullet}$-construction as a module over the $S_{\bullet}$-construction.\\
In Section~\ref{sect:bimod} we prove the first half of our main theorem and characterize the space bimodule objects in span-categories in terms of birelative $2$-Segal objects. To increase the level of readability, we have included some combinatorial lemmas into the Appendix~\ref{sect:Appendix}. \\
In Section~\ref{sect:classificatin of module morphisms} we finally prove our main theorem and extend the equivalence of spaces constructed in Section~\ref{sect:bimod} to an equivalence of $\infty$-categories. As a corollary, we obtain similar characterizations for the $\infty$-category of algebra, left and right module objects.

\subsection{Acknowledgments}

I wish to thank my supervisor Tobias Dyckerhoff for his support and encouragement. I further thank Angus Rush for many discussions about spans and Joachim Kock for interesting discussions about the content of this text. Further, I thank Joachim Kock for pointing out the relation to bicomodule configurations. The author is funded by the Deutsche Forschungsgemeinschaft (DFG, German Research Foundation) under Germany's Excellence Strategy - EXC 2121 Quantum Universe'' - 390833306 and the Collaborative Research Center - SFB 1624 Higher structures, moduli spaces and integrability'' - 506632645.

\section{2-Segal Conditions}\label{sect:2-seg}

Throughout this section, $\c{C}$ denotes an $\infty$-category with finite limits. In \cite{dyckerhoff2019higher}, Dyckerhoff and Kapranov introduced a combinatorial description of associative algebra objects in $\infty$-categories of spans in terms of so-called \emph{$2$-Segal objects}. The goal of this section, is to introduce a variety of \textit{$2$-Segal type} conditions generalizing the original $2$-Segal conditions of Dyckerhoff and Kapranov, and to introduce notions of morphisms between them. In Sections~\ref{sect:bimod} adn \ref{sect:classificatin of module morphisms}, we relate these conditions to homotopy coherent algebra in $\infty$-categories of spans. Before stating the formal definitions, let us motivate them for the example of the $2$-Segal condition of \cite{dyckerhoff2019higher}.

There exists an $\infty$-category of spans $\Span(\c{C})$, whose objects are the objects of $\c{C}$ and whose $1$-morphisms are spans
\[
\begin{tikzcd}
    & W\arrow[dr]\arrow[dl] & \\
    X & & Y
\end{tikzcd}
\]
of $1$-morphisms in $\c{C}$. Since it is an $\infty$-category, it also has invertible $n$-morphisms for all $n\geq 2$. For example, an invertible $2$-morphism is given by a diagram in $\cC$ of the form
\[
\begin{tikzcd}
    & W \arrow[dr] \arrow[dl] & \\
    X & Z \arrow[u,"\simeq"] \arrow[d,swap,"\simeq"] & Y \\
    & V\arrow[ur]\arrow[ul] & 
\end{tikzcd}
\]
where the middle arrows are equivalences. We read this diagram as a morphism 
from the upper to the lower span. The $\infty$-category $\Span(\cC)$ further inherits a monoidal structure from the Cartesian product $\times$ on $\cC$. We can, therefore, study homotopy coherent associative algebras in the monoidal $\infty$-category $\Span(\cC)^{\otimes}$. Informally, the datum of such an algebra consists of
\begin{itemize}
    \item an underlying object $X_{1}\in \c{C}$,
    \item a multiplication span
    \[
\begin{tikzcd}
    & X_{2} \arrow[rd,"\del_{1}"] \arrow[ld,swap,"(\del_{0}{,}\del_{2})"] & \\
    X_{1}\times X_{1} & & X_{1} 
\end{tikzcd}
\]
\item a unit span
\[
\begin{tikzcd}
    & X_{0} \arrow[dr,"s_{0}"]\arrow[dl,swap, "p_{X_{0}}"] & \\
    \ast & & X_{1}
\end{tikzcd}
\]
\item and higher morphisms that describe homotopy coherent associativity and unitality.
\end{itemize}
As the notations suggest, this data organizes into a simplicial object $X_{\bullet}:\Delta^{\op}\rightarrow \cC$, whose $n$-simplices, for $n>2$ are encoded in the associativity data. For example, the lowest instance of higher associativity is the existence of an associativity $2$-isomorphism 
\[
\alpha:X_{2}\times_{X_{1}} X_{2} \xRightarrow{\simeq} X_{2} \times_{X_{1}}X_{2}
\]
in $\Span(\cC)$. This data is given by an object $X_{3}\in \cC$, together with $2$ invertible morphisms
\begin{equation}\label{eq:2-Seg}
\begin{tikzcd}
    & X_{3}  \arrow[dr, "(\del_{3}{,}\del_{1})"] \arrow[dl,"(\del_{0}{,}\del_{2})",swap] & \\
    X_{2}\times_{X_{1}}X_{2} & &  X_{2} \times_{X_{1}}X_{2}
\end{tikzcd}
\end{equation}
that fit into a commutative diagram:
\[
\begin{tikzpicture}[auto]
    \node (A) at (0,0) {$X_{1}\times X_{1} \times X_{1}$};
    \node (B) at (6,0) {$X_{1} \times X_{1}$};
    \node (C) at (12,0) {$X_{1}$};
    \node (D) at (3.5,1.6) {$X_{2} \times X_{1}$};
    \node (E) at (10,1.4) {$X_{2}$};
    \node (F) at (8,2) {$X_{2} \times_{X_{1}} X_{2}$};
    \node (G) at (3,1) {$X_{1}\times X_{2}$};
    \node (H) at (9,1) {$X_{2}$};
    \node (I) at (7,1.5) {$X_{2} \times_{X_{1}} X_{2}$};
    \node (J) at (7.3,2.8) {$X_{3}$};
    \node[scale=0.7,red] (L) at (7.3,2.4) {$\alpha$};
    \draw[->] (D) to [bend right=10] node[swap,scale=0.7] {$(\del_{2},\del_{0})\times \id_{X_{1}}$} (A);
    \draw[->,dotted] (D) to [bend left= 15] node[scale=0.7] {$\del_{1}\times \id_{X_{1}}$} (B);
    \draw[->,dotted] (E) to [bend right= 20]  (B);
    \draw[->] (E) [bend left= 10] to node[scale=0.7] {$\del_{1}$} (C);
    \draw[dotted, ->] (F) to [bend right= 10] (D);
    \draw[->] (F) to [bend left= 10] node[scale=0.7] {$\pi_{2}$}(E);
    \draw[->] (G) to [bend right= 10] node[scale=0.7] {$\id_{X_{1}}\times (\del_{2},\del_{0})$} (A);
    \draw[->] (G) to [bend left =10] node[swap,scale=0.7] {$\id_{X_{1}}\times \del_{1}$} (B);
    \draw[->] (H) to [bend right= 10] node[ scale=0.7] {$(\del_{2},\del_{0})$} (B);
    \draw[->] (H) to [bend left = 10] node[swap,scale=0.7] {$\del_{1}$} (C);
    \draw[->] (I) [bend right = 10] to (G);
    \draw[->] (I) to [bend left =10] (H);
    \draw[->] (J) to [bend left = 10] node[scale= 0.7] {$(\del_{3},\del_{1})$}  (F);
    \draw[->] (J) to [bend right= 30] node[swap, scale= 0.7] {$(\del_{0},\del_{2})$}  (I);
    \draw[->,red] (7.7,2.2) ..  controls (7.3,2.8)  and  (7.1,2.8) .. (7.1,1.8);
\end{tikzpicture}
\]
Since the diagram commutes, the components $\del_{i}$ satisfy the simplicial identities
\[
\del_{i}\circ \del_{j} = \del_{j-1} \del_{i} \text{ for } i<j.
\]
The lowest dimensional $2$-Segal condition is the requirement that the morphisms in Equation~\eqref{eq:2-Seg} are invertible (see Definition~\ref{def:2-Seg}) and therefore encodes the lowest dimensional instance of associativity. Analogously, the higher-dimensional $2$-Segal conditions encode the higher associativity of the algebra. 

We now summarize the content of this section. In Section~\ref{subsect:2-Segal}, we introduce the definitions of $2$-Segal objects and $2$-Segal spans. They assemble into an $\infty$-category $\TSegSpan(\c{C})$ that we will show in Corollary~\ref{cor:2seg} to be equivalent to the $\infty$-category $\Alg(\Span(\cC))$ of algebra objects and algebra morphisms in the $\infty$-category $\Span(\cC)^{\otimes}$. Afterward, we extend the discussion in \ref{subsect:birel2-Seg} and introduce $2$-Segal and Segal span conditions for the indexing categories $\Delta^{\op}_{/[1]}$,  $\Delta_{\geq}^{\op}$ and $\Delta^{\op}_{\leq}$ generalizing the ones on $\Delta$. We show in Section~\ref{sect:bimod} that these admit a similar interpretation in terms of bi-, left, and right module objects in the symmetric monoidal $\infty$-category $\Span(\c{C})^{\otimes}$ respectively. Finally, we provide in Subsection~\ref{subsect:decomposition} a different characterization of these new $2$-Segal conditions in terms of active-inert pullbacks.  This generalizes the equivalence between $2$-Segal and decomposition spaces from \cite{galvez2018decomposition} to this more general class of indexing categories.

\subsection{2-Segal Objects}\label{subsect:2-Segal}
To simplify the exposition, we frequently abuse notation and identify a simplicial object $X_{\bullet}:\Delta^{\op} \rightarrow \c{C}$ with its extension to $\mathsf{Fin}^{\op}_{\geq}$ the category of all finite non-empty linearly ordered sets.  
Let us recall, as a start, the definition of a $2$-Segal object.

\begin{definition}\cite[Def.2.3.1]{dyckerhoff2019higher}\label{def:2-Seg}
    Let $X_{\bullet}:\Delta^{\op} \xrightarrow{} \c{C}$ be a simplicial object. $X_{\bullet}$ is called \textit{$2$-Segal} if for every $n\geq 3$ and $0 \leq i < j\leq n$ the map
    \[
        X_{n} \xrightarrow{} X_{\{0,1,\cdots,i,j,j+1,\cdots,n\}}\times_{X_{\{i,j\}}} X_{\{i,i+1,...,j\}}
    \]
    is an equivalence. We denote by $\operatorname{2-\Seg}_{\Delta}(\c{C})$ the full subcategory of $\Fun(\Delta^{\op},\c{C})$ generated by $2$-Segal objects.
\end{definition}
  The morphisms in the $\infty$-category $\operatorname{2-\Seg}_{\Delta}(\c{C})$ are morphisms of simplicial objects. Under the equivalence between $2$-Segal objects and algebra objects in Spans \cite{stern20212}, a general simplicial map between $2$-Segal objects induces a lax-algebra morphism between the corresponding algebra objects \cite[Sect.4.2]{walde2016hall}. Let us sort out what a strong algebra morphism looks like in terms of $2$-Segal objects: 

\begin{definition}\label{def:2-segspan}
    Let $\sigma:\Delta^{1}\xrightarrow{} \Span(\operatorname{2-\Seg}_{\Delta}(\c{C}))$ be a $1$-morphism in the $\infty$-category of spans of $2$-Segal objects.\footnote{See Construction~\ref{constr:span}} The morphism $\sigma$ can be represented by a span
    \[
        \begin{xy}
            \xymatrix{
            & M_{\bullet} \ar[dr]^{t} \ar[dl]_{s} & \\
            X_{\bullet} & & Y_{\bullet}\\
            }
        \end{xy}
    \]
    with $X_{\bullet},Y_{\bullet}$ and $M_{\bullet}:\Delta^{\op}\rightarrow \c{C}$ $2$-Segal.
    The morphism $\sigma$ is called a \textit{$2$-Segal span} if for every $n\geq 0$ the induced diagrams
   \[
        \begin{xy}
            \xymatrix{ X_{n} \ar[d] & M_{n} \ar[l] \ar[d] \\
            X_{\{0,n\}} & M_{\{0,n\}} \ar[l] \\}
        \end{xy}
    \]
    and
    \[
        \begin{xy}
            \xymatrix{Y_{n} \ar[d] & M_{n} \ar[d] \ar[l] \\
             \prod_{i=0}^{n-1}Y_{\{i,i+1\}} &  \prod_{i=0}^{n-1}M_{\{i,i+1\}} \ar[l] \\
            }
        \end{xy}
    \]
    are pullback diagrams. We call $s$ an \textit{active equifibered $\Delta^{\op}$-morphism} and $t$ a \textit{relative Segal $\Delta^{\op}$-morphism}. 
\end{definition}

\begin{notation}
    We introduce in Subsection~\ref{subsect:birel2-Seg} similar notions for other indexing categories than $\Delta$.
    If the indexing category is clear from the context, we frequently abuse notation and drop the indexing category in the notation. E.g, we will call an active equifibered $\Delta^{\op}$-morphism simply an active equifibered morphism when the indexing category is clear from the context.
    
\end{notation}

\begin{remark}
    The conditions imposed on the individual legs have previously been studied under the name CULF \cite[Sect.4]{galvez2018decomposition} and IKEO in \cite[Sect.8.5]{galvez2018decomposition}. The authors further demonstrate that these types of morphisms induce algebra morphisms in the $\infty$-category of spans. Our terminology is motivated by the theory of algebraic patterns \cite{barkan2022envelopes}, where similar notions appear.
\end{remark}

 It is worth noting that in the definition of a $2$-Segal span, we explicitly require the tip of the span to be $2$-Segal as well. This is not an additional requirement. Indeed, as demonstrated in \cite[Lem.4.6]{galvez2018decomposition}, the source of an active equifibered morphism $f: X_{\bullet}\rightarrow Y_{\bullet}$ with target $2$-Segal is itself $2$-Segal.

In the following, we interpret $2$-Segal spans as the morphisms of a category and therefore have to study their behavior under composition of spans:

\begin{lemma}\label{lem:2segcomp}
    Let $\sigma:\Lambda_{1}^{2} \xrightarrow{} \Span(\operatorname{2-\Seg}_{\Delta}(\c{C}))$ be given by a composable pair of spans
    \[
        \begin{xy}
            \xymatrix{& M_{\bullet} \ar[dr] \ar[dl] & & N_{\bullet} \ar[dr] \ar[dl] & \\
            X_{\bullet} & & Y_{\bullet} & & Z_{\bullet}}
        \end{xy}
    \]
    s.t. each individual span is a $2$-Segal span. Then also the composite span
   \[
            \begin{xy}
                \xymatrix{ & M_{\bullet} \times_{Y_{\bullet}} N_{\bullet}\ar[dr]^{t} \ar[dl]^{s}&\\ 
                X_{\bullet} & & Z_{\bullet}}
            \end{xy}
      \]
    is a $2$-Segal span.
\end{lemma}
\begin{proof}
    First, we show that $s$ is active equifibered. For $n\neq 0$, consider the diagram:
    \begin{equation*}
    \begin{gathered}
            \begin{xy}
            \xymatrix{N_{\{0,n\}}\times_{Y_{\{0,n\}}} M_{\{0,n\}} \ar[d] & N_{n} \times_{Y_{n}} M_{n} \ar[d] \ar[l] \\
            M_{\{0,n\}} \ar[d] & M_{n} \ar[d] \ar[l] \\
            X_{\{0,n\}} & X_{n} \ar[l]}
            \end{xy}
    \end{gathered}
     \end{equation*}
     We need to prove that the exterior diagram is a pullback diagram. By assumption, the lower square is a pullback diagram. Additionally, the upper square is a pullback square if and only if the outer rectangle in the diagram
   \[
            \begin{xy}
                \xymatrix{N_{\{0,n\}} \ar[d] & N_{n} \ar[d] \ar[l] & N_{n} \times_{Y_{n}} M_{n} \ar[d] \ar[l] \\
                Y_{\{0,n\}} & Y_{n} \ar[l] & M_{n} \ar[l]}
            \end{xy}
     \]
    is a pullback square. But this follows from the pasting law.

   On the other hand, we show that $t$ is relative Segal. We need to show that for every $n\geq 0$, the exterior rectangle in the diagram:
    \[
            \begin{xy}
                \xymatrix{Z_{n} \ar[d] &N_{n} \ar[d] \ar[l]  & M_{n}\times_{Y_{n}} N_{n} \ar[l] \ar[d] \\
                \prod_{i=0}^{n-1} Z_{\{i,i+1\}} & \prod_{i=0}^{n-1} N_{\{i,i+1\}} \ar[l] & \prod_{i=0}^{n-1}M_{\{i,i+1\}}\times_{Y_{\{i,i+1\}}} N_{\{i,i+1\}}\ar[l]}
            \end{xy}
     \]
    is a pullback. By assumption, the left square is a pullback square. Further, the right square is a pullback square if and only if the outer rectangle in the diagram:
     \[
            \begin{xy}
                \xymatrix{ \prod_{i=0}^{n-1} M_{\{i,i+1\}} \ar[d] & M_{n} \ar[d] \ar[l] & N_{n} \times_{Y_{n}} M_{n} \ar[d] \ar[l] \\
                 \prod_{i=0}^{n-1}Y_{\{i,i+1\}} & Y_{n} \ar[l] & N_{n} \ar[l]}
            \end{xy}
    \]
    is a pullback square. But this is again a consequence of the pasting law.
\end{proof}

\begin{definition}
    We define the \emph{$\infty$-category of $2$-Segal objects} $\TSegSpan(\c{C})$ as the wide subcategory of $\Span(\operatorname{2-\Seg_{\Delta}}(\c{C}))$ with morphisms given by $2$-Segal spans.
\end{definition}

\subsection{Birelative 2-Segal Objects}\label{subsect:birel2-Seg}

 To describe more general algebraic structures in categories of spans, we need to consider more general indexing categories than $\Delta$. In this subsection, we extend the framework of the previous sections to encompass categories of functors with source $\Delta_{/[1]}$. We call such a functor $M_{\bullet}:\Delta^{\op}_{/[1]} \xrightarrow{} \c{C}$ a \emph{birelative simplicial object}.  The corresponding birelative $2$-Segal conditions are a multi-coloured version\footnote{In the sense of opeards} of the $2$-Segal condition from Definition~\ref{def:2-Seg}. Similar conditions have already been independently introduced in \cite{carlier2020incidence} using a different presentation of the category $\Delta_{/[1]}$ for applications in combibatorics.

\begin{definition}\label{def:birel}
    Let $M_{\bullet}:\Delta^{\op}_{/[1]} \xrightarrow{} \c{C}$ be a birelative simplicial object. $M_{\bullet}$ is called \textit{birelative $2$-Segal}, if for every $n\geq 3$, $f:[n] \xrightarrow{} [1]$ and $0 \leq i < j\leq n$ the diagram
    \[
       \begin{xy}
           \xymatrix{M_{f} \ar[r] \ar[d] & M_{f|_{i,...,j}} \ar[d] \\
        M_{f|_{0,...,i,j,...,n}} \ar[r] & M_{f|_{i,j}}}
       \end{xy}
   \]
is Cartesian. We denote the full subcategory of $\Fun(\Delta_{/[1]}^{\op},\c{C})$ generated by the birelative $2$-Segal objects by $\operatorname{Bi2Seg_{\Delta}}(\c{C})$.
\end{definition}

  As $2$-Segal objects encode algebra objects, we will see in Theorem~\ref{thm:birseg} that birelative $2$-Segal objects encode bimodule objects. For completeness, we also introduce a $2$-Segal type condition describing left and right modules. These have already been studied in the case of $1$-categories in \cite{walde2016hall} and \cite{young2018relative}. We denote by $\Delta_{\leq}$ (resp. $\Delta_{\geq}$) the full subcategory of $\Delta_{/[1]}$ generated by the objects $f:[n] \xrightarrow{} [1]$ that take the value $0$ (resp. $1$) at least and the value $1$ (resp. $0$) at most ones. We call a functor with source $\Delta_{\leq}$ (resp. $\Delta_{\geq}$) a \emph{left (resp. right) relative simplicial object}. The corresponding $2$-Segal condition reads as:

\begin{definition}
     Let $M_{\bullet}:\Delta^{\op}_{\leq} \xrightarrow{} \c{C}$ be a left relative simplicial object. $M_{\bullet}$ is called \textit{left relative $2$-Segal}, if for every $n\geq 3$, $f:[n] \xrightarrow{} [1]\in \Delta_{\leq}$ and $0 \leq i < j\leq n$ the diagram
    \[
       \begin{xy}
           \xymatrix{M_{f} \ar[r] \ar[d] & M_{f|_{i,...,j}} \ar[d] \\
        M_{f|_{0,...,i,j,...,n}} \ar[r] & M_{f|_{i,j}}}
       \end{xy}
   \]
is Cartesian. We denote the full subcategory of $\Fun(\Delta_{\leq}^{\op},\c{C})$ generated by the left relative $2$-Segal objects by $\operatorname{L2Seg_{\Delta}}(\c{C})$. Similarly, we define \textit{right relative $2$-Segal objects} and the $\infty$-category of right relative $2$-Segal objects $\operatorname{R2Seg_{\Delta}}(\c{C})$.
\end{definition}
 
 \begin{remark}
The categories $\Delta_{\leq}$ and $\Delta_{\geq}$ defined above are equivalent to the categories, denoted with the same symbols, defined in \cite[Def.3.2.5]{walde2016hall}. Further, our notions of left and right relative $2$-Segal objects coincide with the notion of a relative $2$-Segal object in \cite[Def.3.5.1.]{walde20212}. 
\end{remark}

 To handle the cases of bi-, left, and right relative $2$-Segal conditions at once, we introduce the following notation:

\begin{notation}
    For $\#\in \{/[1],\leq, \geq\}$ we call a functor $M_{\bullet}:\Delta_{\#}^{\op}\rightarrow \c{C}$ a $\#$-relative simplicial object. Further, we call a $\#$-relative simplicial object a $\#$-relative $2$-Segal object if it satisfies the corresponding $2$-Segal conditions. We denote the full subcategory generated by $\#$-relative $2$-Segal objects by $\operatorname{\#-2Seg_{\Delta}}(\c{C})$.
\end{notation}

\begin{example}\label{ex:regularmod}
    Let $X_{\bullet}:\Delta^{\op}\rightarrow \c{C}$ be $2$-Segal. We can associate to $X_{\bullet}$ a birelative, left and right relative $2$-Segal object
    \begin{itemize}
        \item[(1)] $X_{\bullet}^{b}:\Delta^{\op}_{/[1]} \rightarrow \Delta^{\op} \rightarrow \c{C}$
        \item[(2)] $X_{\bullet}^{l}:\Delta^{\op}_{\leq} \rightarrow \Delta^{\op} \rightarrow \c{C}$
        \item[(3)] $X_{\bullet}^{r}:\Delta^{\op}_{\geq} \rightarrow \Delta^{\op} \rightarrow \c{C}$
    \end{itemize}
    via precomposition with the corresponding forgetful functor. It follows from the $2$-Segal conditions for $X_{\bullet}$ that these objects fulfill the respective (bi)relative Segal conditions. Under the equivalence of Corollary~\ref{cor:lmod} and Theorem~\ref{thm:birseg}, these correspond to the regular bimodule (resp. left, right module) associated to the algebra object corresponding to $X_{\bullet}$.
\end{example}

 We also introduce, in analogy with Definition~\ref{def:2-segspan}, a notion of morphism between $\#$-relative $2$-Segal objects:

\begin{definition}\label{def:birsegspan}
    Let $X_{\bullet},Y_{\bullet}$ and $M_{\bullet}$ be $\#$-relative $2$-Segal. A \textit{$\#$-relative $2$-Segal span} from $A_{\bullet}$ to $B_{\bullet}$ is given by a span
   \[
           \begin{xy}
               \xymatrix{& M_{\bullet} \ar[dr]^{t} \ar[dl]_{s}&  \\
               X_{\bullet} & & Y_{\bullet}}
           \end{xy}
    \]
    s.t. for every $n\geq 0$ and $f:[n] \xrightarrow{}[1]\in \Delta_{\#}$ the diagrams
    \[
           \begin{xy}
               \xymatrix{X_{f} \ar[d]& M_{f} \ar[d] \ar[l] \\
               X_{f|_{0,n}} & M_{f|_{0,n}} \ar[l]
               }
           \end{xy}
      \]
    and 
    \[
            \begin{xy}
                \xymatrix{Y_{f} \ar[d]& M_{f} \ar[d] \ar[l] \\
                 \prod_{i=0}^{n-1} Y_{f|_{i,i+1}} & \prod_{i=0}^{n-1} M_{f|_{i,i+1}} \ar[l] }
            \end{xy}
      \]
    are pullback diagrams. We call $s$ an \textit{active equifibered $\Delta^{\op}_{\#}$-morphism} and $t$ a \textit{relative Segal $\Delta^{\op}_{\#}$-morphism}. 
\end{definition}
 Analogously to Lemma~\ref{lem:2segcomp}, one can prove the following:
\begin{lemma}
    Let $\sigma:\Lambda_{1}^{2} \xrightarrow{} \Span(\operatorname{\#-2Seg_{\Delta}}(\c{C}))$ be given by a composable pair of spans
    \[
        \begin{xy}
            \xymatrix{& M_{\bullet} \ar[dr] \ar[dl] & & N_{\bullet} \ar[dr] \ar[dl] & \\
            X_{\bullet} & & Y_{\bullet} & & Z_{\bullet}}
        \end{xy}
    \]
    s.t. the individual spans are $\#$-relative $2$-Segal spans. Then also the composite span
    \[
            \begin{xy}
                \xymatrix{ & M_{\bullet} \times_{Y_{\bullet}} N_{\bullet}\ar[dr] \ar[dl]&\\ 
                X_{\bullet} & & Z_{\bullet}}
            \end{xy}
      \]
    is a $\#$-relative $2$-Segal span.
\end{lemma}
\begin{proof}
    Similar to Lemma~\ref{lem:2segcomp}.
\end{proof}
\begin{definition}
     We define the \emph{$\infty$-category of birelative $2$-Segal objects} $\operatorname{Bi2Seg^{\leftrightarrow}_{\Delta}}(\c{C})$ to be the wide subcategory of $\Span(\operatorname{Bi2Seg_{\Delta}}(\c{C}))$ with morphisms birelative $2$-Segal spans.

    Similarly, we define the \emph{$\infty$-category of left (resp. right) relative $2$-Segal objects} $\operatorname{L2Seg^{\leftrightarrow}_{\Delta}}(\c{C})$ (resp. $\operatorname{R2Seg^{\leftrightarrow}_{\Delta}}(\c{C}))$) as the wide subcategory of $\Span(\operatorname{L2Seg_{\Delta}}(\c{C}))$ (resp. $\Span(
    \operatorname{R2Seg_{\Delta}}(\c{C}))$) with morphisms given by left (resp. right) relative $2$-Segal spans.
\end{definition}

 It has been shown in \cite[Prop.3.5.10]{walde2016hall} for $1$-categories that left and right relative $2$-Segal objects admit a different description in terms of morphisms of simplicial objects, called relative $2$-Segal objects. This description is useful in the study of examples. Let us, therefore, lift this description into the realm of $\infty$-categories. We do this for left relative $2$-Segal objects. The case of right relative $2$-Segal objects is analogous.

\begin{definition}\cite[Def.2.2,Prop.3.5.10]{young2018relative,walde2016hall}
    A morphism $\pi:X_{\bullet} \xrightarrow{} Y_{\bullet}$ of simplicial objects is called \textit{relative $2$-Segal}, if 
    \begin{itemize}
        \item[1)] the source object $X_{\bullet}$ is $1$-Segal,
        \item[2)] the target object $Y_{\bullet}$ is $2$-Segal,
        \item[3)] and for every $0 \leq i <j \leq n$ the following square is Cartesian
        \begin{equation}\label{eq:active}
                \begin{tikzcd}
                X_n \arrow[r] \arrow[d] & Y_{\{i,..j\}} \arrow[d] \\
                    X_{\{0,...,i,j,...,n\}} \arrow[r] & Y_{\{i,j\}}
             \end{tikzcd}
        \end{equation}
    \end{itemize}
    We denote the full subcategory of $\Fun(\Delta^{1},\Fun(\Delta^{\op},\c{C}))$ generated by relative $2$-Segal morphisms by $\operatorname{Rel2Seg_{\Delta}}(\c{C})$.
\end{definition}
\begin{remark}
    We can extend the square in Equation~\eqref{eq:active} to the following diagram
    \[
    \begin{tikzcd}
        X_{n} \arrow[r] \arrow[d]& Y_{n} \arrow[r]\arrow[d] & Y_{\{i,\dots,j\}} \arrow[d]\\
        X_{\{0,\dots i,j,\dots,n\}} \arrow[r] & Y_{\{0,\dots,i,j,\dots,n\}}\arrow[r] & Y_{\{i,j\}}
    \end{tikzcd}
    \]
    Since $Y_{\bullet}$ is $2$-Segal, the right square is a pullback square. By the pasting law, it follows that the outer square is a pullback square if and only if the right square is a pullback square. Unraveling the definitions, this means that a morphism $X_{\bullet}\rightarrow Y_{\bullet}$ is relative $2$-Segal if and only if it is an active equifibered $\Delta^{\op}$-morphism from a $1$-Segal to a $2$-Segal object. 
\end{remark}
\begin{remark}
    Let us unpack the definition of a relative $2$-Segal object $\pi:X_{\bullet} \rightarrow Y_{\bullet}$. The $2$-Segal object $X_{\bullet}$ encodes an algebra object, and the object $Y_{0}$ encodes the underlying object of a left module. The remaining data encodes the module-action. For example, the module action of $X_{1}$ onto $Y_{0}$ is given by the span
    \[
    \begin{tikzcd}
        & X_{1} \arrow[dr, "\del_{0}"] \arrow[dl, swap, "(\pi_{1}{,}\del_{1})"] & \\
        Y_{1} \times X_{0} & & X_{0}
    \end{tikzcd}
    \]
    The relative $2$-Segal conditions again encode the higher coherences of the left module action.
\end{remark}

 We also introduce a notion of morphism between relative $2$-Segal objects:

\begin{definition}\label{def:relativesegalspan}
    Let $\pi^{i}_{\bullet}:X^{i}_{\bullet}\rightarrow Y^{i}_{\bullet}$ for $0\leq i \leq 2$ be relative $2$-Segal objects. A \textit{relative $2$-Segal span} from $\pi^{0}_{\bullet}$ to $\pi^{2}_{\bullet}$ is given by a span 
    \[
    \begin{tikzcd}
    X_{\bullet}^{0} \arrow[d, swap, "\pi^{0}_{\bullet}"] & X_{\bullet}^{1} \arrow[d, "\pi^{1}_{\bullet}"] \arrow[l, swap, "s^{X}_{\bullet}"] \arrow[r, "t^{X}_{\bullet}"] & X_{\bullet}^{2} \arrow[d,"\pi^{2}_{\bullet}"] \\
        Y_{\bullet}^{0} & Y_{\bullet}^{1} \arrow[l, "s_{\bullet}^{Y}"] \arrow[r,swap, "t_{\bullet}^{Y}"] & Y_{\bullet}^{2} 
    \end{tikzcd}
    \]
    of relative $2$-Segal objects s.t. 
    \begin{itemize}
        \item[(1)] the span $Y_{\bullet}^{0} \xleftarrow{s_{\bullet}^{Y}} Y_{\bullet}^{1} \xrightarrow{t_{\bullet}^{Y}}Y_{\bullet}^{2}$ is $2$-Segal,
        \item[(2)] for every $n\geq 1$ the square
        \[
        \begin{tikzcd}
            X^{1}_{n} \arrow[r, "s_{n}^{X}"] \arrow[d]& X^{0}_{n} \arrow[d] \\
            X^{1}_{\{0\}} \arrow[r, "s_{1}^{X}"] & X^{0}_{\{0\}}
        \end{tikzcd}
        \]
        is Cartesian,
        \item[(3)] and for every $n\geq 1$ the diagram
        \[
        \begin{tikzcd}
            X^{1}_{n} \arrow[r, "t_{n}^{X}"] \arrow[d]& X^{2}_{n} \arrow[d] \\
            Y^{1}_{\{0,1\}}\times \dots \times Y_{\{n-2,n-1\}}^{1} \times X^{1}_{\{n\}} \arrow[r] & Y^{2}_{\{0,1\}}\times \dots \times Y^{2}_{\{n-2,n-1\}} \times X^{2}_{\{n\}}
        \end{tikzcd}
        \]
        is Cartesian.
    \end{itemize}
    We call the pair $(s^{X}_{\bullet},s^{Y}_{\bullet})$ \textit{active equifibered} and the pair $(t^{X}_{\bullet},t^{Y}_{\bullet})$ \textit{relative Segal}. 
\end{definition}
\begin{definition}
    We define then \emph{$\infty$-category of relative $2$-Segal objects} as the wide subcategory $\operatorname{Rel2Seg_{\Delta}^{\leftrightarrow}}(\c{C})$ of the $\infty$-category $\Span(\operatorname{Rel2Seg_{\Delta}}(\c{C}))$ with morphisms given by relative $2$-Segal spans.
\end{definition}

We now turn to the construction of an equivalence $\Theta_{L}^{\leftrightarrow}: \operatorname{L2Seg_{\Delta}^{\leftrightarrow}(\c{C})} \xrightarrow{\simeq} \operatorname{Rel2Seg_{\Delta}^{\leftrightarrow}}$. To do so, we construct an equivalence 
\[
\Theta_{L}:\operatorname{L2Seg_{\Delta}}(\c{C}) \xrightarrow{\simeq} \operatorname{Rel2Seg_{\Delta}}(\c{C})
\]
that induces $\Theta_{L}^{\leftrightarrow}$ on $\infty$-categories of spans. 

For the construction of $\Theta_{L}$, we use observation \cite[Rem.3.2.8]{walde2016hall}. First, note that the category $\Delta_{\leq}$ contains two copies of the simplex category as full subcategories. We denote the corresponding fully faithful inclusion by $i^{\leq}_{0}:\Delta\xrightarrow{}\Delta_{\leq}$ and $i^{\leq}_{1}:\Delta\xrightarrow{} \Delta_{\leq}$. These are given on objects by

\begin{align*}
&i^{\leq}_{0}([n])(k) = 0 \quad \forall k\in [n],    \\
&i^{\leq}_{1}([n])(k) = \begin{cases}
    0 \quad \forall k \neq n+1 \\
    1 \quad k=n+1
\end{cases} .
\end{align*}
Further, we observe that the morphisms $d_{n+1}:i^{\leq}_{0}([n]) \rightarrow i^{\leq}_{1}([n])$ assemble into a natural transformation $d_{\bullet+1}:i^{\leq}_{0}(-)\Rightarrow i^{\leq}_{1}(-)$. This datum is equivalent to the datum of a lax cocone
\begin{equation}\label{diag:lax}
\begin{tikzcd}
    \Delta \arrow[dr,"i_{0}^{\leq}", ""{name=U, below}] \arrow[dd,swap, "id_{\Delta}"] & \\
    & \Delta_{\leq} \\
    \Delta \arrow[ur, swap,"i_{1}^{\leq}"] \arrow[Rightarrow,swap, from=U, shorten >=1.5ex]& 
\end{tikzcd}
\end{equation}
and therefore induces a unique map from the lax colimit. It follows from \cite[Thm.7.4]{gepner2015lax} that the lax colimit is given by the total space of the cocartesian Grothendieck construction $\int_{\Delta^{1}}id_{\Delta}\rightarrow \Delta$, where we identify the functor $id_{\Delta}:\Delta \rightarrow \Delta$ with a morphism in $\Cat$. The morphism 
\[
\theta_{L}: \int_{\Delta^{1}}id_{\Delta} \rightarrow \Delta_{\leq}
\]
induced by Diagram~\ref{diag:lax} is given on the fiber over $0$ by $i_{0}^{\leq}:\Delta \rightarrow \Delta_{\leq}$, on the fiber over $1$ by $i_{1}^{\leq}:\Delta \rightarrow \Delta_{\leq}$ and on morphisms $f:([n],0) \rightarrow ([m],1)$ by 
\[
i_{0}^{\leq}([n]) \xrightarrow{i_{0}^{\leq}(f)} i_{0}^{\leq}(m) \xrightarrow{d_{m+1}} i_{1}^{\leq}(m).
\]
It follows from the explicit description that this functor is an equivalence. 

\begin{prop}\label{prop:relseg}
    The functor $\Theta_{L}:=\theta_{L}^{\ast}:\Fun(\Delta_{\leq}^{\op},\c{C})\rightarrow \Fun(\int_{\Delta^{1}}\id_{\Delta}^{\op},\c{C})$ restricts to an equivalence of $\infty$-categories
    \[
    \Theta_{L}:\operatorname{L2Seg_{\Delta}}(\c{C}) \rightarrow \operatorname{Rel2Seg_{\Delta}}(\c{C}). 
    \]
\end{prop}
\begin{proof}
  It follows from the universal property of the lax colimit that we have an equivalence of $\infty$-categories
    \[
    \Fun(\int_{\Delta^{1}}id_{\Delta}^{\op},\c{C}) \simeq \Fun(\Delta^{1}\times \Delta^{\op},\c{C}) \simeq \Fun(\Delta^{1},\Fun(\Delta^{\op},\c{C})).
    \]
    Under this equivalence, the functor $\Theta_{L}$ maps a left relative simplicial object $X_{\bullet}:\Delta^{\op}_{\leq} \rightarrow\c{C}$ to the natural transformation
    \[
    (i_{1}^{\leq})^{\ast}X_{\bullet} \xrightarrow{\del_{n+1}} (i_{0}^{\leq})^{\ast} X_{\bullet}
    \]
    between simplicial objects. The claim follows from the observation that the functor $\Theta_{L}$ maps left relative $2$-Segal objects to relative $2$-Segal objects.
\end{proof}
 It follows from the discussion after Construction~\ref{constr:span} that the construction of the $\infty$-category of spans induces a functor:
\[
\Span(-):\sf{Cat}_{\infty}^{\lex} \rightarrow \sf{Cat}_{\infty} \
\]
with source the $\infty$-category of small $\infty$-categories with finite limits and finite limit preserving functors. 
We use this functor to conclude our comparison:
\begin{cor}\label{cor:lrelandrel}
    The functor $\Theta_{L}^{\leftrightarrow}:=\Span(\Theta_{L}): \Span(\operatorname{L2Seg_{\Delta}}(\c{C})) \rightarrow \Span(\operatorname{Rel2Seg_{\Delta}}(\c{C})) $ restricts to an equivalence of $\infty$-categories
    \[
    \Theta_{L}^{\leftrightarrow}: \operatorname{L2Seg_{\Delta}^{\leftrightarrow}}(\c{C})\rightarrow \operatorname{Rel2Seg_{\Delta}^{\leftrightarrow}}(\c{C})
    \]
\end{cor}
\begin{proof}
    It suffices to show that $\Theta_{L}^{\leftrightarrow}$ maps left relative $2$-Segal spans to relative $2$-Segal spans. But this follows from unraveling the definitions.
\end{proof}

\begin{remark}
    Analogously to the proof of Corollary~\ref{cor:lrelandrel} one can construct an equivalence 
   \[
\Theta^{\leftrightarrow}_{R}:\operatorname{R2Seg^{\leftrightarrow}_{\Delta}}(\c{C}) \xrightarrow{\simeq} \operatorname{Rel2Seg^{\leftrightarrow}_{\Delta}}(\c{C})
    \]
between the $\infty$-category of right relative $2$-Segal objects and relative $2$-Segal objects.
\end{remark}
   
 We use Corollary~\ref{cor:lrelandrel} to construct examples of left relative $2$-Segal objects and spans. This is particularly fruitful for the example of the hermitian Waldhausen construction that we discuss in Section~\ref{subsect:relbiseg}.

\begin{example}\label{ex:regularmodulerelativ}
    Let $X_{\bullet}:\Delta^{\op}\rightarrow\c{C}$ be a 2-Segal object. We denote by 
    \[
    -\ast[0]:\Delta \rightarrow\Delta
    \]
    the endofunctor of $\Delta$ that adds a maximal element. The simplical object $P^{\triangleright}(X)_{\bullet}:=(-\ast[0])^{\ast}(X_{\bullet})$ is called the final path space \cite[Section 6.2]{dyckerhoff2019higher}. It comes equipped with a morphism of simplicial sets
    \[
    p_{\triangleright}:P^{\triangleright}(X)_{\bullet} \rightarrow X_{\bullet}
    \]
    This morphism is a relative $2$-Segal object. Indeed, it follows from \cite[Thm.6.3.2]{dyckerhoff2019higher} that $P^{\triangleright}(X)_{\bullet}$ is $1$-Segal and it is easy to see that the relative $2$-Segal conditions for $p_{\triangleright}$ reduce to the $2$-Segal conditions on $X_{\bullet}$. This relative $2$-Segal object describes the regular left action of the $2$-Segal object $X_{\bullet}$ on itself. Similarly, one defines the initial path space $P^{\triangleleft}(X)_{\bullet}$ that describes the regular right action of $X_{\bullet}$ on itself.
\end{example}

\begin{example}\label{example:twisted}
    Let $X_{\bullet}$ be a 2-Segal object in $\c{C}$.  Recall, that the edgewise subdivision functor $e:\Delta \rightarrow \Delta$ is defined on objects as $[n]\mapsto [n] \ast [n]^{\op}$. We denote the functor given by precomposition with $e$ by 
    \[
    \Tw(-):= e^{\ast}: \Fun(\Delta^{\op},\c{C}) \rightarrow \Fun(\Delta^{\op},\c{C})
    \]
    and call it the twisted arrow construction. Further, we call for every simplicial object $X_{\bullet}$ the associated simplicial object $\Tw(X)_{\bullet}$ the twisted arrow simplicial object. It has been shown in \cite[Thm.2.9]{bergner2020edgewise}, that for every $2$-Segal object $X_{\bullet}$, the simplicial object $\Tw(X)_{\bullet}$ is $1$-Segal. Moreover, it is easy to check that the morphism
    \[
        \Tw(X)_{\bullet} \rightarrow X_{\bullet} \times X_{\bullet}^{\op}
    \]
    defines a relative $2$-Segal object. This encodes the regular bimodule action of $X_{\bullet}$ on itself.
\end{example}

\subsection{Birelative Decomposition Spaces}\label{subsect:decomposition}

An alternative, but equivalent, way to formalize the $2$-Segal conditions are the \emph{decomposition space conditions} of \cite{galvez2018decomposition}. For the definition of these recall that the simplex category $\Delta$ admits a factorization system $(\Delta^{\act},\Delta^{\inert})$ generated by the active (depicted $\twoheadrightarrow$) and inert morphisms (depicted $\rightarrowtail$) \cite{lurie2017higher}. This factorization system has the special property that $\Delta$ admits active-inert-pushouts, i.e. every cospan
\[
\begin{tikzcd}
    n \arrow[r,tail, "f"]  \arrow[d,two heads, "g"]& m \\
    l & 
\end{tikzcd}
\]
with $f$ inert and $g$ active admits an extension to a pushout diagram:
\[
\begin{tikzcd}
    n \arrow[r,tail]  \arrow[d,two heads]& m  \arrow[d,two heads]\\
    l    \arrow[r,tail] & k 
\end{tikzcd}
\]
in $\Delta$.
A simplicial space $X_{\bullet}:\Delta^{\op}\rightarrow \cS$ is then called a decomposition space if it maps active-inert pullbacks in $\Delta^{\op}$ to pullback diagrams of spaces.
The goal of this section is to extend the definition of a decomposition space to the category $\Delta_{/[1]}$. Therefore we first recall the definition of the active-inert factorization system on $\Delta_{/[1]}$
\begin{notation}
    Given an object $f:[n] \xrightarrow{} [1] \in \Delta_{/[1]}$, we will frequently denote it by $f_{n}$ to indicate its dependence on the source. Given $i<j$ in $[n]$ we denote by $[i,j]:=\{i,...,j\} \subset [n]$ the subinterval from $i$ to $j$ and by $f_{[i,j]}$ the restriction of $f_{n}$ to this subinterval.
\end{notation}
\begin{definition}
    A morphism $f:g_{n}^{0} \rightarrow g_{m}^{1}$ in $\Delta_{/[1]}$ is called 
    \begin{itemize}
        \item[$(1)$] \textit{active} if the underlying morphism of linearly ordered sets $f:[n] \rightarrow [m]$ is endpoint preserving. We depict active morphisms by $\twoheadrightarrow$.
        \item[$(2)$] \textit{inert} if the underlying morphism of linearly ordered sets $f:[n] \rightarrow [m]$ is a subinterval inclusion. We depict inert morphisms by $\rightarrowtail$
    \end{itemize}
    We denote the wide subcategory of $\Delta_{/[1]}$ spanned by the active (resp. inert) morphisms by $\Delta^{\act}_{/[1]}$ (resp. $\Delta^{\inert}_{/[1]}$). Similarly, we call a morphism $f$ in $\Delta^{\op}_{/[1]}$ active (resp. inert) if it is active (resp. inert) as a morphism in $\Delta_{/[1]}$.
\end{definition}
\begin{prop}
    Every morphism $f:g^{0}_{n} \rightarrow g^{1}_{m}$ in $\Delta_{/[1]}$ admits an unique factorization $f=f''\circ f'$ with $f''$ inert and $f'$ active. In particular the subcategories $(\Delta^{\act}_{/[1]},\Delta^{\inert}_{/[1]})$ form a factorization system on $\Delta_{/[1]}$ \cite[Def.5.2.8.8]{lurie2006higher}.
\end{prop}
\begin{proof}
    This is a consequence of the dual of \cite[Prop.2.1.2.5]{lurie2017higher} applied to the Cartesian fibration $\Delta_{/[1]} \rightarrow \Delta$. Alternatively, one can see this more directly as follows. The underlying morphism of linearly ordered sets $f:[n]\rightarrow [m]$ admits an active-inert factorization given by the active map $f':[n] \twoheadrightarrow [f(0),f(n)]$ and  the inert map $f'':[f(0),f(n)] \hookrightarrow [m]$. These morphisms admit unique extensions to morphisms in $\Delta_{/[1]}$.
\end{proof}
 To characterize the active-inert pullback diagrams in $\Delta^{\op}_{/[1]}$, we need the following definition:
\begin{definition}\label{def:composable}
    Let $g:[n]\xrightarrow{}[1]$ and $f:[m] \xrightarrow{} [1]$ be two objects in $\Delta_{/[1]}$. We will call $g$ and $f$ composable, if $g(n)=f(0)$. The concatenation of $g$ with $f$ denoted $g\ast f: [n+m] \xrightarrow{} [1]$ is defined as the map:
    \[
        (g\ast f)(k) = \begin{cases}
            g(k) \quad k\leq n \\
            f(k-n) \quad k \geq n
        \end{cases}
    \]
     Similarly, we denote for two linearly ordered sets $S, T$ by $S\ast T$ the partially ordered set $S\coprod T/(s_{max}\sim t_{min})$, where $s_{max}$ denotes the maximal element of $S$ and $t_{min}$ denotes the minimal element of $T$.
\end{definition}

 Given an active morphism $f^{2}:h^{0}_{k} \twoheadrightarrow h^{1}_{l}$ and an inert morphism $e^{1}:g^{0}_{n} \rightarrowtail h^{0}_{k}$ in $\Delta_{/[1]}$ we can consider an active-inert factorization of its composite $f^{2}\circ e^{1}$. This factorization can be organized into a commutative square
\[
\begin{tikzcd}
    g^{0}_{n} \arrow[d,rightarrowtail, "e^{1}"] \arrow[r, twoheadrightarrow, "f^{1}"] & g^{1}_{m} \arrow[d, rightarrowtail, "e^{2}"] \\
    h^{0}_{k} \arrow[r, twoheadrightarrow, "f^{2}"] & h^{1}_{l}
\end{tikzcd}
\]
It follows from the definition of active and inert morphisms that this square is equivalent to the square
\[
    \begin{tikzpicture}[auto]
    \node (A) at (0,0){$g^{0}_{n}$};
    \node (B) at (4,0){$g^{1}_{m}$};
    \node (C) at (0,-2){$h^{0}_{a}\ast g^{0}_{n} \ast h^{0}_{b}$};
    \node (D) at (4,-2){$h^{1}_{a}\ast g^{1}_{m} \ast h^{1}_{b}$};

    \draw[->>] (A) to node[scale=0.7] {$f^{1}$} (B);
    \draw[>->] (A) to (C);
    \draw[>->] (B) to (D);
    \draw[->>] (C) to node[scale=0.7,swap] {$f^{2}_{a} \ast f^{1}\ast f^{2}_{b} $} (D);
\end{tikzpicture}
\]
 In analogy with \cite[Sect.2.6]{galvez2018decomposition}, we call the squares of the form: 
\[
    \begin{tikzpicture}[auto]
    \node (A) at (0,0){$g^{0}_{n}$};
    \node (B) at (4,0){$g^{1}_{m}$};
    \node (C) at (0,-2){$h^{0}_{a}\ast g^{0}_{n} \ast h^{0}_{b}$};
    \node (D) at (4,-2){$h^{0}_{a}\ast g^{1}_{m} \ast h^{0}_{b}$};

    \draw[->>] (A) to node[scale=0.7] {$f^{1}$} (B);
    \draw[>->] (A) to (C);
    \draw[>->] (B) to (D);
    \draw[->>] (C) to node[scale=0.7,swap] {$\id \ast f^{1}\ast \id $} (D);
\end{tikzpicture}
\] 
\textit{identity extension squares}. These are precisely the active-inert pullbacks in $\Delta^{\op}_{/[1]}$. 
\begin{prop}
    Let $\sigma: \Delta^{1}\times \Delta^{1} \rightarrow \Delta_{/[1]}$ be an identity extension square:
    \[
    \begin{tikzpicture}[auto]
    \node (A) at (0,0){$g^{0}_{n}$};
    \node (B) at (4,0){$g^{1}_{m}$};
    \node (C) at (0,-2){$h^{0}_{a}\ast g^{0}_{n} \ast h^{0}_{b}$};
    \node (D) at (4,-2){$h^{0}_{a}\ast g^{1}_{m} \ast h^{0}_{b}$};

    \draw[->>] (A) to node[scale=0.7] {$f^{1}$} (B);
    \draw[>->] (A) to (C);
    \draw[>->] (B) to (D);
    \draw[->>] (C) to node[scale=0.7,swap] {$\id \ast f^{1}\ast \id $} (D);
\end{tikzpicture}
\] 
    Then $\sigma$ is a pushout square in $\Delta_{/[1]}$.
\end{prop}
 In analogy with \cite{galvez2018decomposition}, we introduce the following definition:
\begin{definition}\label{def:bireldecomp}
    Let $X_{\bullet}:\Delta^{\op}_{/[1]}\rightarrow \c{C}$ be a birelative simplicial object. $X_{\bullet}$ is called a \textit{birelative decomposition space} if it sends every active-inert pullback in $\Delta_{/[1]}^{\op}$ to a pullback square in $\c{C}$, i.e 
    \[
    \begin{tikzpicture}[auto]
    
    \node (A) at (0,0){$g^{0}_{n}$};
    \node (B) at (2,0){$g^{1}_{m}$};
    \node (C) at (0,-2){$h^{0}_{k}$};
    \node (D) at (2,-2){$h^{1}_{l}$};
    \node (E) at (6,0){$X_{g^{0}_{n}}$};
    \node (F) at (8,0){$X_{g^{1}_{m}}$};
    \node (G) at (6,-2){$X_{h^{0}_{k}}$};
    \node (H) at (8,-2){$X_{h^{1}_{l}}$};
    \node (I) at (3,-1) {};
    \node (J) at (5,-1) {};
    
    \draw (1.6,-1.8) -- (1.6,-1.6) -- (1.8,-1.6);
     \draw (7.6,-1.8) -- (7.6,-1.6) -- (7.8,-1.6);
    \draw[|->] (I) to node {$X$} (J);
    \draw[->>] (A) to node[scale=0.7] {$f^{1}$} (B);
    \draw[>->] (A) to node[scale=0.7,swap] {$e^{1}$} (C);
    \draw[>->] (B) to node[scale=0.7] {$e^{2}$} (D);
    \draw[->>] (C) to node[scale=0.7,swap] {$f^{2} $} (D);
     \draw[->] (F) to node[scale=0.7,swap] {$X_{f^{1}}$} (E);
    \draw[->] (G) to node[scale=0.7] {$X_{e^{1}}$} (E);
    \draw[->] (H) to node[scale=0.7,swap] {$X_{e^{2}}$} (F);
    \draw[->] (H) to node[scale=0.7] {$X_{f^{2}}$} (G);
    
\end{tikzpicture}
\] 
\end{definition}
 As in the case of $2$-Segal spaces \cite{galvez2018decomposition}, we show that the birelative decomposition space condition is equivalent to the birelative $2$-Segal condition:
\begin{prop}
    Let $X_{\bullet}:\Delta_{/[1]}^{\op} \rightarrow \c{C}  $ be a birelative simplicial object. The following are equivalent:
    \begin{itemize}
        \item[$(1)$] $X_{\bullet}$ is birelative $2$-Segal.
        \item[$(2)$] $X_{\bullet}$ is a \dspace.
    \end{itemize}
\end{prop}
\begin{proof}
    Assume first that $X_{\bullet}$ is a \dspace. Let $f:[n]\rightarrow [1]$ with $n\geq 3$ be an object of $\Delta_{/[1]}$ and $0\leq i< j\leq n$. Consider the square
    \[
\begin{tikzcd}
    f|_{i,j} \arrow[d,rightarrowtail] \arrow[r, twoheadrightarrow] & f_{i,...,j} \arrow[d, rightarrowtail, ] \\
    f_{0,...,i,j,...,n} \arrow[r, twoheadrightarrow] & f
\end{tikzcd}
\]
    in $\Delta_{/[1]}$. Note that $X_{\bullet}$ satisfies the birelative $2$-Segal conditions if and only if it maps each such square to a pullback square. But this square is an identity extension square in $\Delta_{/[1]}$ and hence an active-inert pullback. 

    Conversely assume that $X_{\bullet}$ is birelative $2$-Segal. We need to show that the image under $X_{\bullet}$ of every identity extension square 
     \[
    \begin{tikzpicture}[auto]
    \node (A) at (0,0){$g^{0}_{n}$};
    \node (B) at (4,0){$g^{1}_{m}$};
    \node (C) at (0,-2){$h^{0}_{a}\ast g^{0}_{n} \ast h^{0}_{b}$};
    \node (D) at (4,-2){$h^{0}_{a}\ast g^{1}_{m} \ast h^{0}_{b}$};

    \draw[->>] (A) to node[scale=0.7] {$f^{1}$} (B);
    \draw[>->] (A) to (C);
    \draw[>->] (B) to (D);
    \draw[->>] (C) to node[scale=0.7,swap] {$\id \ast f^{1}\ast \id $} (D);
\end{tikzpicture}
\] 
is a pullback diagram. We can extend every such to a rectangle:
   \[
    \begin{tikzpicture}[auto]
    \node (A) at (0,0){$g^{0}_{n}$};
    \node (B) at (4,0){$g^{1}_{m}$};
    \node (C) at (0,-2){$h^{0}_{a}\ast g^{0}_{n} \ast h^{0}_{b}$};
    \node (D) at (4,-2){$h^{0}_{a}\ast g^{1}_{m} \ast h^{0}_{b}$};
    \node (F) at (-4,0) {$g^{0}_{\{0,n\}}=g^{1}_{\{0,m\}}$};
    \node (E) at (-4,-2) {$h^{0}_{a} \ast g^{1}_{\{0,m\}} \ast h^{0}_{b}$};

    \draw[->>] (A) to node[scale=0.7] {$f^{1}$} (B);
    \draw[>->] (A) to (C);
    \draw[>->] (B) to (D);
    \draw[->>] (C) to node[scale=0.7,swap] {$\id \ast f^{1}\ast \id $} (D);
    \draw[->>] (F) to (A);
    \draw[->>] (E) to (C);
    \draw[>->] (F) to (E);
\end{tikzpicture}
\] 
Since $X_{\bullet}$ is birelative $2$-Segal it maps the left square and the outer rectangle to pullback diagrams. The claim follows from the pasting law for pullbacks.
\end{proof}

\begin{remark}
    The active-inert factorization system on $\Delta_{/[1]}$ restricts to an active inert factorization system on $\Delta_{\leq}$ and $\Delta_{\geq}$. One can analogously to Definition~\ref{def:bireldecomp} define a notion of left and right relative decomposition space and show that it is equivalent to the left and right relative $2$-Segal conditions.
\end{remark}

\subsection{Relation to bicomodule configurations}\label{subsect:carlier}

In \cite{carlier2020incidence}, Carlier introduced a class of birelative simplicial spaces called \emph{bicomodule configurations} and showed that these induce bicomodules in the homotopy category of spans of spaces \cite[Thm.2.4.1]{carlier2020incidence}. Our first goal of this section is to compare our notion of birelative $2$-Segal objects to the notion of bicomodule configurations:

\begin{prop}\label{prop:comparison}
    A birelative simplicial object $M_{\bullet}\in \Fun(\Delta_{/[1]}^{\op},\cC)$ is birelative $2$-Segal if and only if it is a bicomodule configuration.
\end{prop}

The definition of bicomodule configurations in \cite{carlier2020incidence} uses a different perspective on birelative simplicial objects in terms of grids. To this end, observe that the data of a birelative simplicial object $M_{\bullet}$ can be organized into tge following shape

\[
\begin{tikzcd}
    & M_{\{1\}} \arrow[r]& M_{\{1,1\}}\arrow[l,shift right=1.5ex] \arrow[l,shift left =1.5 ex]  &[-6ex] \dots \\
    M_{\{0\}} \arrow[d] & M_{\{0,1\}} \arrow[d] \arrow[r] \arrow[l]\arrow[u] & M_{\{0,1,1\}} \arrow[l,shift right=1.5ex] \arrow[l,shift left =1.5 ex] \arrow[u] \arrow[d]&[-6ex] \dots \\
    M_{\{0,0\}} \arrow[u,shift right=1.5ex] \arrow[u,shift left =1.5 ex] & M_{\{0,0,1\}} \arrow[l] \arrow[u,shift right=1.5ex] \arrow[u,shift left =1.5 ex] \arrow[r] & M_{\{0,0,1,1\}} \arrow[u,shift right=1.5ex] \arrow[u,shift left =1.5 ex] \arrow[l,shift right=1.5ex] \arrow[l,shift left =1.5 ex] &[-6ex] \dots \\[-5ex]
    \vdots & \vdots & \vdots & 
\end{tikzcd}
\]

where each row and column encodes an augmented simplicial object. More formally, this follows from the existence of a fully faithful functor $j:\Delta_{/[1]}\rightarrow \Delta_{+}\times \Delta_{+}$ that maps a morphism $f;[n]\rightarrow [1]$ that hits $0$ $i$-times and $1$ $j$-times to the object $([i-1],[j-1])$
Here, $\Delta_{+}$ denotes the augmented simplex category. Note that the essential image contains all objects in $\Delta_{+}\times \Delta_{+}$ except $(-1,-1)$. 

Since the definition of a bicomodule configuration refers to the simplicial objects appearing as rows and columns of that grid, we need to introduce a way to extract those in our set up.

\begin{notation}
Let $[n]\in \Delta$ a linearly ordered set. We denote by $\delta^{n}_{0},\delta^{n}_{1}:[n]\rightarrow [1]$ the morphism that is constant on $0$ (resp. $1$). We further denote $\epsilon^{n}_{0}\coloneq \delta^{n}_{0}\ast \id_{[1]}$ and $\epsilon^{n}_{1}\coloneq \id_{[1]}\ast \delta^{n}_{1}$. We extend this to the case $n=-1$ by defining $\epsilon_{1}^{-1}\coloneq \delta_{0}$, $\epsilon_{1}^{-1}\coloneq \delta_{0}$.
\end{notation}

\begin{construction}\label{const:pullback}
For every map $[n]\in \Delta$ there exist a canonical inclusion $c_{n}:\Delta\hookrightarrow \Delta_{/[1]}$ that map an object $[m]\in \Delta$ to the morphism $\epsilon_{0}^{m}\ast\delta^{n}_{1}:[n+m]\to [1]$. Dually there exists a functor $r_{n}:\Delta\rightarrow \Delta_{/[1]}$ that maps $[m]\in \Delta$ to the map $\delta^{n}_{0}\ast\epsilon_{1}^{m}:[n+m]\rightarrow [1]$. We further denote by $c_{-1}:\Delta\rightarrow \Delta_{/[1]}$ the functor that maps $[n]$ to $\delta_{0}^{n}$ and by $r_{-1}:\Delta\rightarrow \Delta_{/[1]}$ the functor that maps $[n]$ to $\delta_{1}^{n}$.

We denote the corresponding restriction functors by
\[
C_{n},R_{n}: \Fun(\Delta_{/[1]}^{\op},\cC) \to \Fun(\Delta_{+}^{\op},\cC)
\]
 Note that under the fully faithful functor $j$ the functor $C_{n}$ and $R_{n}$ precisely correspond to the projection onto the $n$-th row and column.
\end{construction}

\begin{proof}[Proof of Propostion~\ref{prop:comparison}]
    Let $M_{\bullet}$ be birelative simplicial object. We compare the conditions for different choices of $f\in\Delta_{/[1]}$. If $f$ is constant at $0$ or $1$ the birelative $2$-Segal conditions for $f$ are equivalent to the requirement that simplicial objects $C_{-1}(M_{\bullet})$ and $R_{-1}(M_{\bullet})$ denoted $X_{\bullet}$ and $Y_{\bullet}$ in \cite[Thm.2.4.1]{carlier2020incidence} are $2$-Segal.
    
    For $f:[n]\to [1]\in \Delta_{/[1]}$ that is not constant, we need to distinguish different cases for the parameters $0\leq i <j \leq n$ appearing in the definition of the birelative $2$-Segal conditions:
    \begin{itemize}
        \item if $f(i)=f(j)=0$ then the birelative $2$-Segal conditions are equivalent to active equifiberedness of the map $C_{0}(M_{\bullet})\rightarrow C_{-1}(M_{\bullet})$ if $f(i)=1$ only for $i=n$ and to the $1$-Segal conditions for $C_{n}(M_{\bullet})$ else. 
         
        \item if $f(i)=f(j)=1$ then the birelative $2$-Segal conditions are equivalent to active equifiberedness of the map $R_{0}(M_{\bullet})\rightarrow R_{-1}(M_{\bullet})$ if $f(i)=0$ only for $i=0$ and to the Segal conditions for $R_{n}(M_{\bullet})$ else. 
        
        \item if $f(i)=0$ and $f(j)=1$ then it follows from a pasting argument that these are equivalent to being stable in the sense of \cite[Sect.2.2]{carlier2020incidence}
    \end{itemize}
    Since this compares all possible birelative $2$-Segal conditions to the conditions of a bicomodule configuration, this finishes the proof.
\end{proof}

After the first version of this text appeared, Kock and Mikhail have described a notion of morphism between bicomodule configurations in \cite[Rem.3.2.4]{kock2025abacus}. Our last goal in this section is to show that their notion can be recovered as an example of our notion of birelative $2$-Segal span. 

Let $(s_{\bullet},t_{\bullet}):M_{\bullet}\rightarrow X_{\bullet}\times Y_{\bullet}$ be a birelative $2$-Segal span s.t. the morphisms $C_{-1}s_{\bullet},R_{-1}s_{\bullet}$ and $C_{-1}t_{\bullet}, R_{-1}s_{\bullet}$ are equivalences. Under the equivalence of birelative $2$-Segal spans with bimodule morphisms established in Theorem~\ref{thm:birseg}, these correspond to bimodule morphisms that restrict to the identity map between algebras. 

\begin{definition}
    A map of simplicial objects $\pi_{\bullet}:X_{\bullet}\rightarrow Y_{\bullet}$ is called a \emph{right} (resp. \emph{left}) \emph{fibration}, if for all $n\geq 1$ the diagram 
    \[
    \begin{tikzcd}
        X_{n} \arrow[r] \arrow[d,swap,"\pi_{n}"]& X_{\{n\}} \arrow[d,"\pi_{n-1}"] \\
        Y_{n} \arrow[r,swap]& Y_{\{n\}}    
    \end{tikzcd} \quad \text{resp.} \quad \begin{tikzcd}
        X_{n} \arrow[r] \arrow[d,swap,"\pi_{n}"]& X_{\{0\}} \arrow[d,"\pi_{n-1}"] \\
        Y_{n} \arrow[r,swap]& Y_{\{0\}} 
    \end{tikzcd}
    \]
    is Cartesian.
\end{definition}
\begin{prop}
     Let $(s_{\bullet},t_{\bullet}):M_{\bullet}\rightarrow X_{\bullet}\times Y_{\bullet}$ be a span between birelative $2$-Segal objects, s.t. $C_{-1}s_{\bullet},R_{-1}s_{\bullet}$ and $C_{-1}t_{\bullet}, R_{-1}s_{\bullet}$ are equivalences. Then the following are equivalent:
    \begin{itemize}
        \item[(1)] $(s_{\bullet},t_{\bullet})$ is a birelative $2$-Segal span.
        \item[(2)] for every $n\geq 0$ the morphisms $R_{n}s_{\bullet}$, $C_{n}t_{\bullet}$ are right and the morphisms $R_{n}t_{\bullet}$, $C_{n}s_{\bullet}$ are left fibrations.
    \end{itemize}
\end{prop}
\begin{proof}
    It is easy to see that $(1)$ implies $(2)$. For the other direction, let $f:[n]\rightarrow [1]$ be a morphism with $f(0)=0$ and $f(n)=n$ and denote by $i\in [n]$ the unique element, with $f(i)=0$ and $f(i+1)=1$. Note that in case that $C_{-1}s_{\bullet},R_{-1}s_{\bullet}$ and $C_{-1}t_{\bullet}, R_{-1}s_{\bullet}$ are equivalences, the birelative $2$-Segal conditions on $s_{\bullet}$ and $t_{\bullet}$ reduce to the claim that for every $f:[n]\rightarrow [1]$ with $f(0)=0$ and $f(n)=1$ the diagrams
\[
\begin{tikzcd}
    M_{f} \arrow[r,"s_{f}"] \arrow[d] & X_{f}\arrow[d]  \\
    M_{f\vert_{0,n}} \arrow[r] & X_{f\vert_{0,n}}
\end{tikzcd} 
\quad \text{and} \quad
\begin{tikzcd}
      M_{f} \arrow[r,"t_{f}"] \arrow[d] & Y_{f}\arrow[d]  \\
    M_{f\vert_{i,i+1}} \arrow[r] & Y_{f\vert_{i,i+1}}
\end{tikzcd}
\]
are pullback squares. We analyze the first one, because the second one is analogous. To this end, we decompose it into a rectangle
\[
\begin{tikzcd}
    M_{f} \arrow[r,"s_{f}"] \arrow[d] & X_{f}\arrow[d]  \\
    M_{f\vert_{0,\dots,i,n}} \arrow[r] \arrow[d] & X_{f\vert_{0,\dots,i,n}} \arrow[d] \\
    M_{f\vert_{0,n}} \arrow[r] & X_{f\vert_{0,n}}
\end{tikzcd} 
\]
Note that by Definition of the functors $C_{n}$ and $R_{n}$ the top and bottom rectangles are given by 
\[
\begin{tikzcd}
    R_{i}(M_{\bullet}) (\{i+1,\dots,n\}) \arrow[r] \arrow[d] &  R_{i}(X_{\bullet})(\{i+1,\dots,n\}) \arrow[d] \\
     R_{i}(M_{\bullet})(\{n\})\arrow[r] &  R_{i}(X_{\bullet})(\{n\})  
\end{tikzcd}
\]
and 
\[
\begin{tikzcd}
    C_{0}(M_{\bullet})(\{0,\dots,i\}) \arrow[r] \arrow[d] &  C_{0}(X_{\bullet})(\{0,\dots,i\}) \arrow[d] \\
     C_{0}(M_{\bullet})(\{0\})\arrow[r] &  C_{0}(X_{\bullet})(\{0\})
\end{tikzcd}
\]
which are pullback by assumption $(2)$. Hence, the claim follows from the pasting law for pullbacks. 
\end{proof}

\section{Segal objects from K-theory}\label{sect:examples}

Our main goal in this section is to provide examples of the various types of $2$-Segal conditions introduced in the last section. The fundamental example of a $2$-Segal space is the Waldhausen $S_{\bullet}$-construction\footnote{See Definition~\ref{def:Waldhausen}} of an exact $\infty$-category, whose space of $n$-simplices $S_{n}(\c{C})$ can be described as the space of length $n$-flags in $\c{C}$. To show that this simplicial space satisfies the $2$-Segal conditions one has to compare different pasting of flags. The reason that these are equivalent is a consequence of the third isomorphism theorem that holds in any exact $\infty$-category $\cC$. It is therefore natural to expect examples of other types of $2$-Segal objects to arise from similar constructions with exact $\infty$-categories. 

Indeed, as our first example, we consider exact functors $F:\c{C}\rightarrow \c{D}$ between exact $\infty$-categories. Every such induces a morphism of simplicial spaces $S_{\bullet}(F):S_{\bullet}(\c{C}) \rightarrow S_{\bullet}(\c{D})$. We show that under the assumptions of Proposition~\ref{prop:exactfunctor}, this morphism gives rise to examples of active equifibered and relative Segal morphisms. 

Moreover, we also present a different construction of an active equifibered and relative Segal morphism from an exact functor $F$. Instead of a morphism between $2$-Segal spaces, we construct from $F:\c{C}\rightarrow \c{D}$ a $2$-Segal space $S^{\rel}_{\bullet}(F)$ itself, called the relative $S_{\bullet}$-construction \cite{waldhausen}. This $2$-Segal space fits into a sequence of simplicial spaces
\begin{equation}\label{equ:recollement}
\c{D}_{\bullet} \xrightarrow{\iota_{\bullet}} S^{\rel}_{\bullet}(F) \xrightarrow{\pi_{\bullet}} S_{\bullet}(\c{C})
\end{equation}
that is important in algebraic K-theory since it induces the long exact sequence of relative K-theory \cite{waldhausen}.
After recalling the definition of the Waldhausen and relative Waldhausen construction in the setting of exact $\infty$-categories \cite{barwick2015exact}, we show that the morphisms $\iota_{\bullet}$ and $\pi_{\bullet}$ are examples of a relative Segal and an active equifibered morphism respectively. 

These constructions admit applications to the theory of Hall algebras. Indeed, every such morphism induces a (co)algebra morphism between the corresponding Hall algebras. In particular, the morphism induced by the relative $S_{\bullet}$-construction has been applied for the construction of derived Hecke actions \cite[Sect.5]{kapranov2019cohomological}. This is especially advantageous in the realm of  $\infty$-categories, where constructing homotopy coherent algebra morphisms can be challenging.   

This discussion admits an analog for birelative simplicial spaces in the setting of hermitian K-theory. For exact $1$-categories with exact duality $(\c{C},\sf{D}_{\c{C}})$ it has been proven by Young \cite[Thm.3.6]{young2018relative} that the hermitian $\c{R}_{\bullet}$-construction of $(\c{C},\sf{D}_{\c{C}})$ describes a relative $2$-Segal space over $S_{\bullet}(\c{C})$. Young uses this observation to construct representations of Hall algebras and to apply it in orientifold Donaldson\textendash Thomas-theory \cite{young2020representations}. 

In Section~\ref{subsect:relbiseg}, we extend Young's result to the setting of exact $\infty$-categories with duality. Our proof is inspired by the construction of the real Waldhausen $S_{\bullet}$-construction as presented in \cite{heine2019real}. For an exact $\infty$-category $\c{C}$ with duality functor $\sf{D}_{\c{C}}$ the relative $2$-Segal space from Example~\ref{example:twisted}
\begin{equation}\label{eq:twisted}
\Tw(S_{\bullet}(\c{C})) \rightarrow S_{\bullet}(\c{C}) \times S_{\bullet}(\c{C})^{\rev}
\end{equation}
admits an extension to a morphism of simplicial spaces with $\sf{C}_{2}$-action. For $n=1$, we can explicitly describe the $\sf{C}_{2}$-action by the commutative diagram:
\[
\begin{tikzpicture}
    
    \node (A) at (0,0) {\begin{tikzcd}
        C_{0,1} \arrow[r] \arrow[d] & C_{0,2} \arrow[d] \\
        0 \arrow[r] & C_{1,2} 
    \end{tikzcd}};
    \node (B) at (6,0) {\begin{tikzcd}
        \sf{D}_{\c{C}}(C_{1,2}) \arrow[r] \arrow[d] & \sf{D}_{\c{C}}(C_{0,2}) \arrow[d] \\
        0 \arrow[r] & \sf{D}_{\c{C}}(C_{0,1}) 
    \end{tikzcd}};
    \node (C) at (0,-2) {$(C_{0,1},C_{1,2})$};
    \node (D) at (6,-2) {$(\sf{D}_{\c{C}}(C_{1,2}), \sf{D}_{\c{C}}(C_{0,1}))$};
    \draw[|->] (A) to (B);
    \draw[|->] (A) to (C);
    \draw[|->] (B) to (D);
    \draw[|->] (1.5,-2) to (4,-2);
\end{tikzpicture}
\]
The hermitian $\c{R}_{\bullet}$-construction then arises as the simplicial space of homotopy fixed points. Furthermore, the induced morphism $\c{R}_{\bullet}(\c{C})\rightarrow S_{\bullet}(\c{C})$ is a relative $2$-Segal space. Using this method we finally construct examples of active equifibered and relative Segal maps between relative $2$-Segal objects from exact duality preserving functors $F:\c{C}\rightarrow \c{D}$.

These results can be applied to the construction of $\infty$-categorical Hall algebra representations and their morphisms. In particular, these constructions are essential for the construction of representations of categorified Hall algebras \cite{porta,diaconescu2022cohomological}.

\subsection{2-Segal Spans from Exact \texorpdfstring{$\infty$}{inf}-Functors}\label{subsect:relseg}

 In this subsection, we construct examples of active equifibered and relative Segal morphisms between Waldhausen $S_{\bullet}$-construction. For the reader's convenience, we first recall the construction of the Waldhausen $S_{\bullet}$-construction of an exact $\infty$-categories.

An \emph{exact $\infty$-category} consists of a triple $(\c{C},\c{C}^{\ingr},\c{C}^{\egr})$ of an additive $\infty$-category $\c{C}$ and two wide Waldhausen subcategories $\c{C}^{\ingr}$ and $\c{C}^{\egr}$ satisfying a list of compatibility conditions \cite[Def.3.1]{barwick2015exact}. We call the morphism in $\c{C}^{\ingr}$ \emph{ingressive} denoted $\rightarrowtail$ and the morphism in $\c{C}^{\egr}$ \emph{egressive} denoted $\twoheadrightarrow$. For an exact $\infty$-category $\c{C}$, a bicartesian square
\[
\begin{tikzcd}
    C_{1} \arrow[r, rightarrowtail ]\arrow[d, twoheadrightarrow]& C_{2} \arrow[d, twoheadrightarrow] \\
    0 \arrow[r, rightarrowtail] & C_{3}
\end{tikzcd}
\]
in $\c{C}$ is called an \emph{extension square}. A functor 
\[
F:(\c{C},\c{C}^{\ingr},\c{C}^{\egr}) \rightarrow \exact{D}
\]
between exact $\infty$-categories is called an \emph{exact $\infty$-functor} \cite[Def.4.1]{barwick2015exact} if it preserves ingressive and egressive morphisms, zero objects, and extension squares. We denote by $\sf{Exact}_{\infty}$ the subcategory of $\Fun(\Lambda_{2}^{2},\Cat)$ spanned by exact $\infty$-categories and exact functors. 
Consider for every $n\geq 0$ the poset $[n]$ as an $\infty$-category and let $\mathsf{Ar}([n]):=\Fun([1],[n])$ be the associated $\infty$-category of arrows. These assemble into a cosimplicial $\infty$-category $\mathsf{Ar}([-]): \Delta \rightarrow \Cat$.

\begin{definition}\label{def:Waldhausen}
     Let $(\c{C},\c{C}^{\ingr},\c{C}^{\egr})$ be an exact $\infty$-category. We denote by $S_{n}(\c{C},\c{C}^{\ingr},\c{C}^{\egr})\subset \sf{Map}(\mathsf{Ar}([n]),\c{C})$ the full subspace spanned by those functors $F:\mathsf{Ar}([n])\rightarrow \c{C}$ that satisfy the following conditions:
     \begin{itemize}
         \item[(1)] for every $0 \leq i \leq n$ we have $F(i,i)=0$.
         \item[(2)] for every $0 \leq i \leq k \leq j \leq n$ the morphism $F(i,j) \rightarrowtail F(k,j)$ is ingressive.
         \item[(3)] for every $0 \leq i \leq j \leq l \leq n$ the morphism $F(i,j) \twoheadrightarrow F(i,l)$ is egressive
         \item[(4)] for every $0 \leq i \leq k \leq j \leq l \leq n$ the square 
         \[
        \begin{tikzcd}
        F(i,j) \arrow[r, rightarrowtail ]\arrow[d, twoheadrightarrow]& F(k,j) \arrow[d, twoheadrightarrow] \\
        F(i,l) \arrow[r, rightarrowtail] & F(k,l)
        \end{tikzcd}
        \]
         is bicartesian.
     \end{itemize}
     When $n$ varies, the spaces $S_{n}(\c{C},\c{C}^{\ingr},\c{C}^{\egr})$ assemble into a simplicial space called the \textit{Waldhausen $S_{\bullet}$-construction}. When the classes of ingressive and egressive morphisms are clear from the context, we abuse notation and denote the Waldhausen $S_{\bullet}$-construction by $S_{\bullet}(\c{C})$. For completeness, we include a proof of the following statement. The original proof in the more general context of proto-exact $\infty$-categories is given in \cite[Thm.7.3.3]{dyckerhoff2019higher}. 
\end{definition}
\begin{prop}\cite[Thm.7.3.3]{dyckerhoff2019higher}
    Let $(\c{C},\c{C}^{\ingr},\c{C}^{\egr})$ be an exact $\infty$-category. The simplicial space $S_{\bullet}(\c{C})$ is $2$-Segal.
\end{prop}
\begin{proof} 
    By the path space criterion \cite[Thm.6.3.2]{dyckerhoff2019higher} it suffices to show that the two path spaces $(P^{\tri}S)_{\bullet}$ and $(P^{\tle}S)_{\bullet}$ are $1$-Segal.
    It follows from \cite[Lem.1.2.2.4.]{lurie2017higher} that for every $n\geq 0$ the projection maps 
    \[
   \Map_{\Cat}([n-1],\c{C}^{\egr}) \xleftarrow{p^{v}} S_{n}(\c{C})\xrightarrow{p_{h}} \Map_{\Cat}([n-1],\c{C}^{\ingr})
    \]
    that are induced by restricting to the subposet
    \[
    \{(0,n) < (1,n) < ... < (n-1,n)\} \hookrightarrow \mathsf{Ar}([n]) \hookleftarrow \{(0,1) < (0,2) < ... < (0,n)\}
    \]
    are equivalences. Through these equivalences, the $1$-Segal conditions for the two path spaces translate into the conditions that for every $n\geq 2$, the morphisms
    \[
       \Map_{\Cat}([n],\c{C}^{\egr}) \rightarrow \Map_{\Cat}([1]\amalg_{[0]}\dots \amalg_{[0]}[1],\c{C}^{\egr}) 
    \]
    and 
     \[
       \Map_{\Cat}([n],\c{C}^{\ingr}) \rightarrow \Map_{\Cat}([1]\amalg_{[0]}\dots \amalg_{[0]}[1],\c{C}^{\ingr}) 
    \]
    are equivalences. But this condition is satisfied, since the functor $[1]\amalg_{[0]}\dots\amalg_{[0]}[1]\rightarrow [n]$ is an equivalence of $\infty$-    categories.
\end{proof}

Let $(\c{C},\c{C}^{\ingr}, \c{C}^{\egr})$ and $(\c{D},\cD^{\ingr},\cD^{\egr})$ be exact $\infty$-categories and $F:(\c{C},\cC^{\ingr}, \cC^{\egr}) \rightarrow \exact{D}$ be an exact $\infty$-functor. It follows from the construction of the Waldhausen construction that $F$ induces a functor $S_{\bullet}(F):S_{\bullet}(\c{C}) \rightarrow S_{\bullet}(\c{D})$ between the corresponding Waldhausen constructions. We analyze the conditions under which this morphism is active equifibered (resp. relative Segal). To do so, we need to introduce some terminology:
\begin{definition}
    Let $\exact{C}$ be an exact $\infty$-category and let $\c{C}_{0} \subset \c{C}$ be a full exact\footnote{A subcategory $\cC_{0}\subset \cC$ is called exact, if the triple $(\cC_{0},\cC^{\ingr}\cap \cC_{0},\cC^{\egr}\cap\cC_{0})$ is an exact $\infty$-category.} subcategory. We call $\c{C}_{0}$ 
    \begin{itemize}
        \item[(1)] \textit{extension closed} if every exact-sequence $\sigma:\Delta^{1}\times \Delta^{1}\rightarrow \c{C}$
    \[
    \begin{tikzcd}
    C_{1} \arrow[r, rightarrowtail ]\arrow[d, twoheadrightarrow]& C_{2} \arrow[d, twoheadrightarrow] \\
    0 \arrow[r, rightarrowtail] & C_{3}
    \end{tikzcd}
    \]
    in $\c{C}$ with $C_{1},C_{3} \in \c{C}_{0}$, factors through $\c{C}_{0}$.
    \item[(2)] \textit{closed under quotients (resp. subobjects)} if for every egressive morphism $C_{2}\twoheadrightarrow C_{3}$ (resp. ingressive morphism $C_{1}\rightarrowtail C_{2}$) in $\c{C}$ with $C_{2} \in \c{C}_{0}$ also $C_{3}\in \c{C}_{0}$ (resp. $C_{1}\in \c{C}_{0}$)
    \end{itemize}
\end{definition}

We can then prove the following:
\begin{prop}\label{prop:exactfunctor}
Let $\exact{C}$ and $\exact{D}$ be exact $\infty$-categories and $F:\exact{C} \rightarrow \exact{D}$ be a fully faithful exact $\infty$-functor. The induced morphism $S_{\bullet}(F)$ of simplicial spaces is 
\begin{itemize}
    \item[(1)] active equifibered if and only if the essential image of $F$ is closed under quotients and subobjects.
    \item[(2)] relative Segal if and only if the essential image of $F$ is extension closed.
\end{itemize}
\end{prop}
\begin{proof}
    To show $(1)$, we need to show that for every $n\geq 1$ the morphism  
    \[
    \eta_{1}:\Map_{\Cat}([n-1],\cC^{\ingr}) \rightarrow \Map_{\Cat}([n-1],\c{D}^{\ingr})\times_{\c{D}^{\simeq}} \c{C}^{\simeq}
    \]
    is an equivalence of spaces. It follows from the $2$-out-of-$3$ property for fully faithful functors that $\eta_{1}$ is fully faithful. To prove the claim, it suffices to show that $\eta_{1}$ is essentially surjective. But this follows since $\c{C}\subset \c{D}$ is closed under subobjects.

    The proof of $(2)$ is similar. We need to show that for every $n\geq 1$ the functor
    \[
    \eta_{2}:\Map_{\Cat}([n-1],\cC^{\ingr}) \rightarrow \Map_{\ingr}([n-1],\c{D}^{\ingr})\times_{\c{D}^{\simeq}\times \dots \times \c{D}^{\simeq}} (\c{C}^{\simeq}\times \dots \times \c{C}^{\simeq})
    \]
    is an equivalence. It follows again from the $2$-out-of-$3$ property for fully faithful functors that this functor is fully faithful. Moreover, since the essential image is closed under extensions, it is easy to see that the functor is also essentially surjective.
\end{proof}
\begin{remark}\label{rem:Hallalgebras}
    Let $(\cC,\cC^{\ingr},\cC^{\egr})$ be an exact $\infty$-category, and $\cC_{0}\subset \cC$ be full exact subcategories. The inclusion functor
    \[
    i:\cC_{0}\rightarrow \cC
    \]
    induces a $\bK$-linear map between the corresponding vector spaces of groupoid functions 
    \[
    i_{\ast}:\Hom_{\Set}(\pi_{0}(\cC^{\simeq}_{0}),\bK) \rightarrow \Hom_{\Set}(\pi_{0}(\cC^{\simeq}),\bK)
    \]
    that maps the constant function $[C]_{\cC_{0}}$ for $C\in \cC_{0}$ to the corresponding constant function $[C]_{\cC}$ in $\cC$. If $\cC_{0}$ is closed under extensions (resp. under quotients and subobjects), then, according to the definition, this map naturally extends to a morphism between the corresponding Hall algebras \cite{ringel1990hall} (resp. Hall coalgebras). These classes of morphisms are precisely the ones described by the above class of active equifibered and relative Segal morphisms.
\end{remark}

As explained in the introduction of this section, we can also construct different examples of active equifibered and relative Segal maps, using the Sequence~\eqref{equ:recollement}. We therefore first recall all the necessary ingredients for this construction. Let $F:\exact{C} \rightarrow \exact{D}$ be an exact functor between exact $\infty$-categories. Following \cite{dyckerhoff2024spherical}, we define the \textit{relative $S_{\bullet}$-construction} of $F$ as the simplicial space denoted $S_{\bullet}^{\rel}(F)$ whose $\infty$-category of $n$-simplices is defined as the pullback
\[
\begin{tikzcd}
    S_{n}^{\rel}(F)\arrow[r] \arrow[d] & S_{n+1}(\c{D}) \arrow[d,"\del_{n+1}"]\\
    S_{n}(\c{C}) \arrow[r, "F"] & S_{n}(\c{D})
\end{tikzcd}
\]
Alternatively, we can also describe $S_{n}^{\rel}(F)$ as the full subspace of the space of sections of the Grothendieck construction of the functor
\[
\Fun([n-1],\cC^{\ingr}) \xrightarrow{\ev_{n-1}} \c{C} \xrightarrow{F} \c{D}
\]
generated by the ingressive sections. Since the full subcategory of $\Fun(\Delta^{\op},\c{C})$ generated by $2$-Segal objects is closed under limits, it follows that $S_{\bullet}^{\rel}(F)$ is itself $2$-Segal. The simplicial space $S_{\bullet}^{\rel}(F)$ sits in a sequence of simplicial spaces
\[
\c{D}_{\bullet} \xrightarrow{\iota_{\bullet}} S^{\rel}_{\bullet}(F) \xrightarrow{\pi_{\bullet}} S_{\bullet}(\c{C}),
\]
where the maps $\iota_{n}:\c{D} \rightarrow S_{n}(F)$ and $\pi_{n}:S_{n}(F) \rightarrow S_{n}(\c{C})$ are given by
\[
d \mapsto (0 \rightarrowtail ... \rightarrowtail 0 \rightarrowtail b)
\]
and
\[
(a_{0} \rightarrowtail ... \rightarrowtail ... \rightarrowtail a_{n-1} \rightarrowtail b) \mapsto (a_{0} \rightarrowtail ... \rightarrowtail a_{n-1})
\]
respectively. These maps have the following properties:
\begin{prop}\label{prop:relativeK}
    Let $F:\c{C} \rightarrow \c{D}$ be an exact functor between exact $\infty$-categories. Then:
    \begin{itemize}
        \item[(1)] the morphism $\pi_{\bullet}$ is active equifibered.
       \item[(2)] the morphism $\iota_{\bullet}$ is relative Segal.
    \end{itemize}
\end{prop}
\begin{proof}
Note that $S^{\sf{rel}}_{n}(F)$ can be written as the pullback
\[
    \begin{tikzpicture}[auto]
    \node (A) at (0,-3) {$\Map_{\Cat}([1],\cD^{\ingr})$} ;
    \node (B) at (3,-3) {$\c{D}^{\simeq}$} ;
    \node (E) at (0,0){$S_{n}^{\rel}(F)$};
    \node (F) at (3,0){$\Map_{\Cat}([n-1],\cC^{\ingr})$};
    \node (G) at (0,-1.5){$S_{1}^{rel}(F)$};
    \node (H) at (3,-1.5){$\c{C}^{\simeq}$};

    \draw (0.3,-0.5) -- (0.5,-0.5) -- (0.5,-0.3);
     \draw (0.3,-2) -- (0.5,-2) -- (0.5,-1.8);
     \draw[->] (E) to  (F);
    \draw[->] (E)  to (G);
    \draw[->] (F) to node[scale=0.7] {$\ev_{0}$} (H);
    \draw[->] (G) to  (H);
    \draw[->] (G) to (A);
    \draw[->] (A) to (B);
    \draw[->] (H) to (B);
    
\end{tikzpicture}
\]
of the outer rectangle. To show the claim for $\pi_{\bullet}$, we need to show that the upper diagram is a pullback as well. But this follows from the pasting law for pullbacks. For $(2)$ note that the pullback of spaces
\[
\c{D}^{\simeq} \times_{S_{1}^{\rel}(F)\times_{S_{0}^{\rel}(F)}\dots \times_{S_{0}^{\rel}(F)}S_{1}^{\rel}(F)} S_{n}^{\rel}(F) 
\]
can be identified with the full subspace of $S_{n}^{\rel}(F) \simeq S_{n}(\c{C})\times_{S_{n}(\c{D})} S_{n+1}(\c{D})$ generated by those objects $(F_{\c{C}}.G_{\c{D}},\alpha)$, s.t. 
\begin{itemize}
    \item[(1)] the restriction of $G_{\c{D}}$ to the subposet $\{(0,n)< ... < (n-1,n)\}$ is constant 
    \item[(2)] the restriction of $F_{\c{C}}$ to the subposet $\{(0,1)<(1,2)<...<(n-1,n)\}$ is constant at $0$
\end{itemize}
It then follows from an iterated Kan extension argument similar to \cite[Lem.1.2.2.4]{lurie2017higher} that this subspace is equivalent to $\c{D}^{\simeq}$.
\end{proof}

As an application, we use the above results for the construction of the so-called derived Hecke actions from \cite{kapranov2019cohomological}.
\begin{prop}\label{prop:rel2segobje}
    Let $\c{C}$ be an $\infty$-category with finite limits, and let $p:Y_{\bullet} \rightarrow X_{\bullet}$ be a relative $2$-Segal object. Let further, $f:X_{\bullet} \rightarrow Z_{\bullet}$ be an active equifibered, and $g:W_{\bullet} \rightarrow X_{\bullet}$ be a relative Segal morphism. Then
    \begin{itemize}
        \item[(1)] the composite $ f\circ p:Y_{\bullet} \rightarrow Z_{\bullet}$ is a relative $2$-Segal object.
        \item[(2)] the pullback $\pi:Y_{\bullet} \times_{X_{\bullet}} W_{\bullet}\rightarrow W_{\bullet}$ is a relative $2$-Segal object.       
    \end{itemize}
\end{prop}
\begin{proof}
    Since $Y_{\bullet}$ is $1$-Segal by definition, for the first case we only need to show that the composite $f\circ p$ is active equifibered. But this follows since active equifibered morphisms are closed under composition.
    
    For the second case, note that it follows from the proof of Lemma~\ref{lem:2segcomp} that the pullback of an active equifibered morphism along a relative Segal morphism is active equifibered. We therefore only need to show that $Y_{\bullet} \times_{X_{\bullet}} W_{\bullet}$ is $1$-Segal. Consider the morphisms:
    \[
    Y_{\bullet} \times_{X_{\bullet}} W_{\bullet} \rightarrow Y_{\bullet} \rightarrow \ast_{\bullet}.
    \]
    Since both of them are relative Segal, their composite is so. In particular, it follows that $Y_{\bullet} \times_{X_{\bullet}} W_{\bullet}$ is $1$-Segal.
\end{proof}

As a consequence, we obtain:
\begin{cor}\label{cor:1-SegWald}
    Let $F:\exact{C} \rightarrow \exact{D}$ be an exact fully faithful functor between exact $\infty$-categories whose essential image is extension closed. Then the simplicial space $S_{\bullet}(F)$ is $1$-Segal.
\end{cor}
\begin{proof}
    This follows from an application of Proposition~\ref{prop:rel2segobje} to the pullback diagram that defines $S_{\bullet}(F)$.
\end{proof}
\begin{example}
    Let $F:\exact{C}\rightarrow \exact{D}$ be a fully faithful exact functor between exact $\infty$-categories whose essential image is closed under extensions. Then the morphism
    \[
    S_{\bullet}(F) \xrightarrow{} S_{\bullet}(\c{C})
    \]
    defines a relative $2$-Segal object by Corollary~\ref{cor:1-SegWald}. Unraveling the definitions, this equips $\c{D}$ with the structure of an $S_{\bullet}(\c{C})$-module with action given by
    \[
    \begin{tikzcd}
        & S_{1}(F) \arrow[dr] \arrow[dl] & \\
        \c{C} \times \c{D} & & \c{D}
    \end{tikzcd}
    \]
    This module structure is responsible for the so-called \emph{derived Heck actions} as defined in \cite[Sect.5.2]{kapranov2019cohomological}. Indeed, the analysis of \cite[Sect.8]{dyckerhoff2018triangulated} for the theory with transfer given by Borel\textendash Moore homology applied to the above module structure recovers their derived Hecke actions. Furthermore, it can be used to extend their derived Hecke action to the categorified Hall algebra of \cite{porta,diaconescu2022cohomological}.
\end{example}

\subsection{Relative 2-Segal Spans from Duality preserving \texorpdfstring{$\infty$}{inf}-Functors}\label{subsect:relbiseg}

Our goal in this section is to construct examples of relative $2$-Segal spaces and relative $2$-Segal spans. to this end, we present a general construction of relative $2$-Segal objects from $2$-Segal spaces with duality. For this construction, we first have to recall some facts about $\infty$-categories with duality. Since we only work with simplicial objects in $\cS$, we drop it from the notation.
\begin{definition}
    Let $C\in \cC$ be an object in an $\infty$-category. A \emph{$\sf{C}_{2}$-action} on $C$ is a functor $\sf{BC}_{2}\rightarrow \cC$ that maps the unique point to $C$.
\end{definition}
 Recall that the category $\Delta$ admits a canonical $\sf{C}_{2}$-action. This action maps an object $[n]$ to $[n]$ and a morphism $f:[n]\rightarrow [m]$ to the composite morphism
\[
[n]\simeq [n]^{op} \xrightarrow{f^{op}} [m]^{op}\simeq [m]
\]
where the equivalence $[n] \simeq [n]^{op}$ is given by $i\mapsto n-i$. By functoriality of taking presheaves, the $\sf{C}_{2}$-action on $\Delta$ induces one on $\c{P}(\Delta)$ that maps a simplicial space $X_{\bullet}$ to $X_{\bullet}^{\rev}$. The restriction of this action to the full subcategory $\Cat\subset \c{P}(\Delta)$ maps an $\infty$-category to its opposite. 
\begin{definition}
    A \textit{simplicial space with duality} is a homotopy fixed point with respect to the above $\sf{C}_{2}$-action on $\c{P}(\Delta)$.  
    Similarly, an \textit{$\infty$-category with duality} is a homotopy fixed point with respect to the $\sf{C}_{2}$-action on $\Cat$ given by taking opposites.
\end{definition}

 More precisely an $\infty$-category with duality is a section of the cocartesian fibration $\widetilde{\sf{Cat}}_{\infty} \rightarrow \sf{BC}_{2}$ encoding the $\sf{C}_{2}$-action on $\Cat$. We define the $\infty$-category $\sf{Cat}_{\infty}^{\sf{C}_{2}}$ of $\infty$-categories with duality and the $\infty$-category of simplicial spaces with duality $\c{P}(\Delta)^{\sf{C}_{2}}$ as the respective $\infty$-categories of homotopy fixed points. 

\begin{remark}
    The $\infty$-categories $\Delta$, $\c{P}(\Delta)$ and $\Cat$ also admit a trivial $\sf{C}_{2}$-action. We denote the respective category of homotopy fixed points with respect to this trivial action by $\Delta[\sf{C}_{2}]$, $\c{P}(\Delta)[\sf{C}_{2}]$ and $\Cat[\sf{C}_{2}]$. Unraveling definitions, these are given by the $\infty$-categories of functors $\Cat[\sf{C}_{2}]\simeq \Fun(\sf{BC}_{2},\Cat)$. Hence, a homotopy fixed point with respect to this trivial action describes an object with $\sf{C}_{2}$-action.
\end{remark}

 Note that the non-trivial $\sf{C}_{2}$-action on $\c{P}(\Delta)$ restricts to a $\sf{C}_{2}$-action on $\operatorname{2-\Seg_{\Delta}}$. 
\begin{definition}
    A \textit{$2$-Segal space with duality} is a homotopy fixed point with respect to the non-trivial $\sf{C}_{2}$-action on $\operatorname{2-\Seg_{\Delta}}$. 
\end{definition}

We denote by $\operatorname{2-\Seg_{\Delta}^{\sf{C}_{2}}}$  the $\infty$-category of homotopy fixed points. By construction, this $\infty$-category comes equipped with a forgetful functor to $\operatorname{2-\Seg_{\Delta}}$. We call a morphism $f:X_{\bullet} \rightarrow Y_{\bullet}$ in $\operatorname{2-\Seg_{\Delta}^{\sf{C}_{2}}}$ active equifibered (resp. relative Segal) if its image in $\operatorname{2-\Seg_{\Delta}}$ is active equifibered (resp. relative Segal). Similarly, we call a span 
\[
\begin{tikzcd}
    & Y_{\bullet} \arrow[dr, "t_{\bullet}"] \arrow[dl, swap, "s_{\bullet}"] & \\
    X_{\bullet} & & Z_{\bullet}
\end{tikzcd}
\]
in $\operatorname{2-Seg_{\Delta}^{\sf{C}_{2}}}$ a \textit{$2$-Segal span with duality} if $s_{\bullet}$ is active equifibered and $t_{\bullet}$ is relative Segal. Since the $\infty$-category $\operatorname{2-\Seg_{\Delta}^{\sf{C}_{2}}}$ admits small limits, it makes sense to define:
\begin{definition}
    We define the $\infty$-category $\operatorname{2-Seg_{\Delta}^{\leftrightarrow,\sf{C}_{2}}}$ of \emph{$2$-Segal spaces with duality} as the subcategory of $\Span(\operatorname{2-Seg_{\Delta}^{\sf{C}_{2}}})$ with morphisms $2$-Segal spans with duality.
\end{definition}

 Similarly, note that the trivial $\sf{C}_{2}$-action on $\c{P}(\Delta)$ induces the trivial $\sf{C}_{2}$-action on $\operatorname{Rel2Seg_{\Delta}}$. 

\begin{definition}
    A \textit{relative $2$-Segal space with $\sf{C}_{2}$-action} is a homotopy fixed point with respect to the trivial $\sf{C}_{2}$-action on $\operatorname{Rel2Seg}$. We denote by $\operatorname{Rel2Seg_{\Delta}}[\sf{C}_{2}]$ the $\infty$-category of homotopy fixed points. 
\end{definition}
 As above the $\infty$-category $\operatorname{Rel2Seg}[\sf{C}_{2}]$ comes equipped with a forgetful functor to the $\infty$-category $\operatorname{Rel2Seg_{\Delta}}$. We define active equifibered and relative Segal morphisms as for $\operatorname{2-\Seg_{\Delta}^{\sf{C}_{2}}}$. The associated version of $2$-Segal span is called a \textit{relative $2$-Segal span with $\sf{C}_{2}$-action}.

\begin{definition}
    We define the \emph{$\infty$-category of relative $2$-Segal spaces with $\sf{C}_{2}$-action} $\operatorname{Rel2Seg_{\Delta}^{\leftrightarrow}}[\sf{C}_{2}]$ as the subcategory of $\Span(\operatorname{Rel2Seg_{\Delta}}[\sf{C}_{2}])$ with morphisms relative $2$-Segal spans with $\sf{C}_{2}$-action.
\end{definition}

 Our first goal in this section is the construction of a functor
 \[
 \operatorname{\Tw^{\leftrightarrow}(-)^{\sf{C}_{2}}}: \operatorname{2-Seg_{\Delta}^{\leftrightarrow,\sf{C}_{2}}}\rightarrow \operatorname{Rel2Seg_{\Delta}^{\leftrightarrow}},
 \]
 extending the functor $\Tw(-)$, that associates to every $2$-Segal space with duality a relative $2$-Segal space.
 Recall the definition of the edgewise subdivision functor $e:\Delta \rightarrow \Delta$ from Example~\ref{example:twisted}. We call a morphism in $\Cat[\sf{C}_{2}]$ a $\sf{C}_{2}$-equivariant functor. This functor becomes $\sf{C}_{2}$-equivariant if we equip the source with the trivial and the target with the non-trivial $\sf{C}_{2}$-action. By functoriality, this lifts to a $\sf{C}_{2}$-equivariant functor 
\[
\Tw := e^{\ast}:\c{P}(\Delta) \rightarrow \c{P}(\Delta)
\]
that induces a functor $\Tw: \c{P}(\Delta)^{\sf{C}_{2}}\rightarrow \c{P}(\Delta)[\sf{C}_{2}]$ on homotopy fixed points. The canonical inclusions $[n] \hookrightarrow [n]\ast [n]^{\op}$ and $[n]^{\op} \hookrightarrow [n]\ast [n]^{\op}$ induce a natural transformation 
\[
\Tw \Rightarrow id \times (-)^{op} : \c{P}(\Delta) \rightarrow \c{P}(\Delta)
\]
It follows from \cite[Lem.2.23]{heine2019real} that the natural transformation is $\sf{C}_{2}$-equivariant and hence induces a natural transformation 
\[
\Tw \Rightarrow id \times (-)^{\op}: \c{P}(\Delta)^{\sf{C}_{2}}\rightarrow \c{P}(\Delta)[\sf{C}_{2}]
\]
on homotopy fixed points. In total, we can interpret this construction as a functor
\[
\c{P}(\Delta)^{\sf{C}_{2}} \xrightarrow{\operatorname{\Tw^{\rightarrow}(-)}} \Fun(\Delta^{1},\c{P}(\Delta)[\sf{C}_{2}]) \xrightarrow{(-)^{\sf{C}_{2}}} \Fun(\Delta^{1},\c{P}(\Delta))
\]
that associates to a simplicial space with duality $X_{\bullet}$ the morphism
\[
\Tw(X)_{\bullet}^{\sf{C}_{2}} \rightarrow (X_{\bullet}\times X_{\bullet}^{\rev})^{\sf{C}_{2}}
\]
\begin{prop}
    Let $X_{\bullet} \in \c{P}(\Delta)^{\sf{C}_{2}}$ be a simplicial space with duality and let $X_{\bullet} \times X_{\bullet}^{\rev}$ be the induced simplicial space with $\sf{C}_{2}$-action. There exists an equivalence of simplicial spaces
    \[
    (X_{\bullet} \times X_{\bullet}^{\rev})^{\sf{C}_{2}} \simeq X_{\bullet}
    \]
\end{prop}
\begin{proof}
    For any simplicial space $X_{\bullet}$, we can construct a $\sf{C}_{2}$-action on $X_{\bullet}\times X_{\bullet}$ via right Kan extension
    \[
    \begin{tikzcd}
        \ast \arrow[r, "X_{\bullet}"] \arrow[d] & \c{P}(\Delta) \\
        \sf{BC}_{2} \arrow[ur,swap, "X_{\bullet}\times X_{\bullet}"] & 
    \end{tikzcd}
    \]
    By adjunction, the projection onto the first factor $X_{\bullet} \times X_{\bullet}^{\rev}\rightarrow X_{\bullet}$ induces a morphism of simplicial spaces with $\sf{C}_{2}$-action 
    \[
    \kappa: X_{\bullet} \times X_{\bullet}^{\rev} \rightarrow X_{\bullet} \times X_{\bullet}
    \]
    It follows from the construction that the functor underlying $\kappa$ is given by 
    \[
    \id_{X_{\bullet}} \times \mathsf{D}: X_{\bullet} \times X_{\bullet}^{\rev} \rightarrow X_{\bullet} \times X_{\bullet},
    \]
     where $\sf{D}$ denotes the duality on $X_{\bullet}$. Hence, it is an equivalence. We therefore obtain an equivalence on homotopy fixed points. The claim follows from the transitivity of right Kan extensions
\end{proof}

\begin{prop}\label{prop:twfunctor}
    The functor $\operatorname{\Tw^{\rightarrow}(-)}$ constructed above restricts to a functor
    \[
    \operatorname{\Tw^{\rightarrow}(-)}:\operatorname{2-Seg_{\Delta}^{\sf{C}_{2}}} \rightarrow \operatorname{Rel2Seg_{\Delta}}[\sf{C}_{2}]
    \]
    Furthermore, it preserves active equifibered and relative Segal morphisms.
\end{prop}
\begin{proof}
    It follows from Example~\ref{example:twisted} that the map 
    $\Tw(X)_{\bullet} \rightarrow X_{\bullet} \times X_{\bullet}^{\rev}$
    is a relative $2$-Segal object with $\sf{C}_{2}$-action. The second claim can be checked directly by looking at the corresponding pullback diagrams.
\end{proof}
\begin{prop}\label{prop:homotopyfunc}
    The functor $(-)^{\sf{C}_{2}}: \Fun(\Delta^{1},\c{P}(\Delta)[\sf{C}_{2}]) \rightarrow \Fun(\Delta^{1},\c{P}(\Delta))$
    restricts to a functor
    \[
   (-)^{\sf{C}_{2}}: \operatorname{Rel2Seg_{\Delta}}[\sf{C}_{2}] \rightarrow \operatorname{Rel2Seg_{\Delta}}
    \]
    Further, it preserves active equifibered and relative Segal morphisms.
\end{prop}
\begin{proof}
    The claim follows from the fact that the functor $\sf{lim}_{\sf{BC}_{2}}(-)$ preserves small limits.
\end{proof}

 Combining Proposition~\ref{prop:homotopyfunc} and \ref{prop:twfunctor} we can obtain our first goal of this section
\begin{prop}\label{prop:dualityfunct}
The composite $\infty$-functor
\[
\operatorname{2-Seg_{\Delta}^{\sf{C}_{2}}} \xrightarrow{\operatorname{\Tw^{\rightarrow}(-)}}\operatorname{Rel2Seg_{\Delta}}[\sf{C}_{2}] \xrightarrow{(-)^{\sf{C}_{2}}} \operatorname{Rel2Seg_{\Delta}}
\]
induces an $\infty$-functor on the level of spans
\[
\operatorname{\Tw^{\leftrightarrow}(-)^{\sf{C}_{2}}}: \operatorname{2-Seg_{\Delta}^{\leftrightarrow,\sf{C}_{2}}} \rightarrow\operatorname{Rel2Seg_{\Delta}^{\leftrightarrow}}[\sf{C}_{2}] \rightarrow\operatorname{Rel2Seg^{\leftrightarrow}_{\Delta}}
\]
\end{prop}

 We use this Proposition for the construction of examples of relative $2$-Segal spaces and spans. In the previous section, we have shown that algebraic $K$-theory is a rich source of examples of active equifibered and relative Segal $\Delta^{\op}$-morphisms. The analogue for relative $2$-Segal spans is hermitian $K$-theory. The ideas behind hermitian $K$-theory originate from the fundamental work of Hesselholt and Madsen \cite{hesselholt2015real} on real algebraic $K$-theory. 

An $\infty$-categorical formulation of these ideas is given in \cite{heine2019real}. Analogous to algebraic $K$-theory, hermitian $K$-theory is described by a hermitian analogue of the Waldhausen construction. The hermitian Waldhausen construction associates to every exact $\infty$-category with duality a simplicial space with duality. For the construction of the hermitian Waldhausen construction, we follow the presentation from \cite[Sect.8.2]{heine2019real}. To do so, we first recall some facts about exact $\infty$-categories with duality and $\cS[\sf{C}_{2}]$-enriched $\infty$-categories.

Recall from Subsection~\ref{subsect:relseg} that the $\infty$-category $\sf{Exact}_{\infty}$ of exact $\infty$-categories is defined as the subcategory of $\sf{Exact}_{\infty}\subset \Fun(\Lambda_{2}^{2},\Cat)$ with objects exact $\infty$-categories and morphisms exact functors. The $\infty$-category $\Lambda_{2}^{2}$ carries a natural $\sf{C}_{2}$-action. This action combines with the $\sf{C}_{2}$-action on $\Cat$ to an action on $\Fun(\Lambda_{2}^{2},\Cat)$ that is defined on objects as
\[
(\c{C}_{0} \rightarrow \c{C}_{2} \leftarrow \c{C}_{1}) \mapsto
(\c{C}_{1}^{op} \rightarrow \c{C}_{2}^{op} \leftarrow \c{C}_{0}^{op}).
\]
It is easy to see that this action restricts to a $\sf{C}_{2}$-action on the subcategory $\sf{Exact}_{\infty}$. 

\begin{definition}
     An \textit{exact $\infty$-category with duality} is a homotopy fixed point with respect to the above $\sf{C}_{2}$-action on $\sf{Exact}_{\infty}$. We call the $\infty$-category of homotopy fixed points $\sf{Exact}_{\infty}^{\sf{C}_{2}}$ the $\infty$-category of exact $\infty$-categories with duality. A morphism in this $\infty$-category is called an exact duality preserving functor.
\end{definition}

 Informally, an exact $\infty$-category with duality $\c{C}$ is a Waldhausen $\infty$-category such that the underlying $\infty$-category admits the structure of an $\infty$-category with duality and the class of ingressive morphisms $\c{C}_{in}$ together with the opposites of the ingressive morphisms form the structure of an exact $\infty$-category on $\c{C}$. These serve as an input for the hermitian Waldhausen construction.

 In \cite{heine2019real}, the authors work with real exact $\infty$-categories and equip the real Waldhausen $S_{\bullet}$-construction with the structures of a real simplicial space. Since the $\infty$-category of exact $\infty$-categories with duality forms a full subcategory of the $\infty$-category of real exact $\infty$-categories \cite[Rem.2.33]{heine2019real}, we can apply the construction of \cite[Sect 8.2]{heine2019real} to our situation. 

The key ingredient in the construction \cite[Sect.8.2]{heine2019real} is enriched $\infty$-category theory. In our situation, we are interested in $\infty$-categories enriched over the $\infty$-category $\cS[\sf{C}_{2}]$ of spaces with $\sf{C}_{2}$-action. The amount of enriched $\infty$-category theory necessary for this construction is described in \cite[App. A]{heine2019real}.\footnote{For a more general discussion of enriched $\infty$-category theory see \cite{gepner2015enriched}}
\begin{notation}
    In the following, we adopt the convention to denote $\cS[\sf{C}_{2}]$-enriched $\infty$-categories by $\eC$ to distinguish them from their underlying $\infty$-category $\c{C}$.
\end{notation}

 For an $\cS[\sf{C}_{2}]$-enriched $\infty$-category $\ec{C}$, the underlying $\infty$-category is obtained by taking homotopy fixed points on Hom-spaces \cite[Sect.3.1]{heine2019real}. It follows from \cite[Cor.2.8]{heine2019real} that the $\infty$-categories $\sf{Cat}_{\infty}^{\sf{C}_{2}}$ and $\c{P}(\Delta)^{\sf{C}_{2}}$ are Cartesian closed and are therefore enriched over $\cS[\sf{C}_{2}]$. Informally, the $\sf{C}_{2}$-action maps a duality preserving functor 
\[
F:(\cC,\sD_{\cC})\rightarrow (\cD,\sD_{\cD})
\]
to the duality preserving functor
\[
\begin{tikzcd}
    \cC \arrow[r, "\sD_{\cC}"] & \cC^{\op} \arrow[r,"F^{\op}"] & \cD^{\op} \arrow[r,"\sD_{\cD}^{\op}"] & \cD
\end{tikzcd}.
\]
This enrichment restricts along the inclusion $\Delta^{\sf{C}_{2}}\subset \sf{Cat}_{\infty}^{\sf{C}_{2}}$ to an enrichment of $\Delta$. We denote the $\cS[\sf{C}_{2}]$-enriched $\infty$-category $\Delta$ by $\un{\Delta}$. Further, the $\infty$-category $\cS[\sf{C}_{2}]$ is naturally enriched over itself.

It follows from \cite[Prop.2.12]{heine2019real} that there exists an equivalence $\ec{P}(\Delta)^{\sf{C}_{2}}\simeq \un{\Fun}_{\cS[\sf{C}_{2}]}(\underline{\Delta}^{op},\cS[\sf{C}_{2}])$ of $\cS[\sf{C}_{2}]$-enriched $\infty$-categories between the $\infty$-category of simplicial spaces with duality and the $\infty$-category of $\cS[\sf{C}_{2}]$-enriched functors. We will use this equivalence for the construction of the duality structures on the Waldhausen $S_{\bullet}$-construction. 

The $\infty$-category $\sf{Exact}^{\sf{C}_{2}}_{\infty}$ naturally admits the structure of a $\cS[\sf{C}_{2}]$-enriched $\infty$-category $\un{\sf{Exact}}_{\infty}^{\sf{C}_{2}}$ \cite[Sect.7.1]{heine2019real} that is cotensored over $\un{\sf{Cat}}_{\infty}^{\sf{C}_{2}}$ \cite[Sect.8.2]{heine2019real}. This cotensoring induces an $\cS[\sf{C}_{2}]$-enriched functor:
\[
(-)^{-}:\un{\sf{Exact}}_{\infty}^{\sf{C}_{2}} \times \un{\sf{Cat}}_{\infty}^{\sf{C}_{2},\op} \rightarrow \un{\sf{Exact}}_{\infty}^{\sf{C}_{2}}.
\]
For an $\infty$-category with duality $\c{I}$ and an exact $\infty$-category with duality $\c{C}$ the underlying exact $\infty$-category of the cotensor $\c{C}^{I}$ is given by the functor category $\Fun(I,\c{C})$. The exact structure is defined objectwise, and the induced duality structure maps a functor $F:I \rightarrow \c{C}$ to the composite functor
\[
\begin{tikzcd}
    I \arrow[r, "\sf{D}_{I}^{\op}"] & I^{\op} \arrow[r, "F^{\op}"] & \c{C}^{\op} \arrow[r, "\sf{D}_{\c{C}}"]& \c{C}. 
\end{tikzcd}
\]
As the first step of our construction of the $S_{\bullet}$-construction, we define an $\cS[\sf{C}_{2}]$-enriched version of the $\infty$-category of arrows. Analogously to the non-enriched case, we define the $\cS[\sf{C}_{2}]$-enriched functor
\[
\underline{\sf{Ar}}(-):\underline{\Delta} \subset \un{\sf{Cat}}_{\infty}^{\sf{C}_{2}} \xrightarrow{\un{\hom}([1],-)} \un{\sf{Cat}}_{\infty}^{\sf{C}_{2}}
\]
where $\un{\hom}(-,-)$ denotes the $\cS[\sf{C}_{2}]$-enriched internal Hom-functor of $\sf{Cat}_{\infty}^{\sf{C}_{2}}$. For every $[n]$ this induces a duality structure on the $\infty$-category of $\sf{Ar}([n])$.
Note that for every $[n]\in \un{\Delta}$ the induced duality structure on $\un{\c{C}}^{\un{\sf{Ar}([n])}}$ restricts to a duality structure on $S_{n}(\c{C})$. It follows that the $\cS[\sf{C}_{2}]$-enriched functor 
\[
(-)^{-}:\un{\sf{Exact}}_{\infty}^{\sf{C}_{2}} \times \un{\Delta}^{op} \rightarrow \un{\sf{Exact}}_{\infty}^{\sf{C}_{2}}
\]
restricts to an $\cS[\sf{C}_{2}]$-enriched functor
\[
\un{S}_{\bullet}(-): \un{\sf{Exact}}_{\infty}^{\sf{C}_{2}} \times \un{\Delta}^{op} \rightarrow \un{\sf{Exact}}_{\infty}^{\sf{C}_{2}}.
\]
Such a functor is transpose to a $\cS[\sf{C}_{2}]$-enriched functor 
\[
\un{S}^{\simeq}_{\bullet}(-): \un{\sf{Exact}}_{\infty}^{\sf{C}_{2}} \rightarrow \Fun_{\cS[\sf{C}_{2}]}(\un{\Delta}^{\op},\un{\sf{Exact}}_{\infty}^{\sf{C}_{2}}) \xrightarrow{(-)^{\simeq}} \Fun_{\cS[\sf{C}_{2}]}(\un{\Delta}^{\op},\un{\cS}[\sf{C}_{2}])\simeq \un{\c{P}}(\Delta)^{\sf{C}_{2}}.
\]
We can finally define:
\begin{definition}
 Let $\c{C}$ be an exact $\infty$-category with duality. The $2$-Segal space with duality $\un{S}_{\bullet}^{\simeq}(\c{C})$ is called the \textit{Waldhausen $S_{\bullet}$-construction with duality}.   
\end{definition}

The underlying simplicial space of $\un{S}_{\bullet}$ coincides with the Waldhausen $S_{\bullet}$-construction of an exact $\infty$-category as introduced in Definition~\ref{def:Waldhausen}. In particular, the Waldhausen $S_{\bullet}$-construction with duality is an example of a $2$-Segal space with duality. 

We can now apply Proposition~\ref{prop:dualityfunct} to the Waldhausen $S_{\bullet}$-construction with duality.

\begin{definition}\cite[Sect.1.8]{hornbostel2004localization}
    Let $\c{C}$ be an exact $\infty$-category with duality. We call the simplicial object $\Tw(\un{S}_{\bullet}^{\simeq}(\c{C}))^{\sf{C}_{2}}$ the \textit{hermitian $S_{\bullet}$-construction} and denote it by $\c{R}_{\bullet}(\c{C})$. 
\end{definition}

\begin{remark}
    It is known by the fundamental work of Waldhausen \cite{waldhausen} that for every exact $\infty$-category $\Tw(S_{\bullet}(\c{C}))$ is equivalent to Quillens $Q$-construction $Q(\c{C})$. The authors constructed a hermitian version of the $Q$-construction \cite{calmes2021hermitian}. We expect that for every exact $\infty$-category with duality $\c{C}$, the equivalence between the $\Tw(S_{\bullet}(\c{C}))$ and the $Q$-construction $Q(\c{C})$ extends to an equivalence of simplicial spaces with $\sf{C}_{2}$-action.
\end{remark}

 The following result extends \cite[Thm.3.6]{young2018relative} into the realm of $\infty$-categories:

\begin{cor}
    Let $\c{C}$ be an exact $\infty$-category with duality. The morphism
    \[
    \c{R}_{\bullet}(\c{C}) \rightarrow S_{\bullet}(\c{C})
    \]
    is a relative $2$-Segal object.
\end{cor}
\begin{proof}
    Apply Proposition~\ref{prop:dualityfunct} to $\un{S}^{\simeq}_{\bullet}(\c{C})$.
\end{proof}

 We can further use Proposition~\ref{prop:dualityfunct} for the construction of relative $2$-Segal spans. Let $F:\c{C} \rightarrow \c{D} \in \sf{Exact}_{\infty}^{\sf{C}_{2}}$ be a duality preserving exact functor. It follows from the construction of $\un{S}_{\bullet}(-)$ that it induces a morphism
\[
\un{S}_{\bullet}(F): \un{S}_{\bullet}(\c{C}) \rightarrow \un{S}_{\bullet}(\c{D})
\]
of simplicial spaces with duality. 

\begin{cor}
    Let $F:\c{C} \rightarrow \c{D} \in \sf{Exact}_{\infty}^{\sf{C}_{2}}$ be a fully faithful duality preserving exact functor and consider the induced diagram
    \[
  \begin{xy}
    \xymatrix{\c{R}_{\bullet}(\c{C}) \ar[r]^{\c{R}_{\bullet}(F)} \ar[d] & \c{R}_{\bullet}(\c{D}) \ar[d] \\
    S_{\bullet}(\c{C}) \ar[r]^{S_{\bullet}(F)} & S_{\bullet}(\c{D}) 
    }
\end{xy}
\]
Then the morphism $(\c{R}_{\bullet}(F),S_{\bullet}(F))$ is
\begin{itemize}
    \item[(1)] active equifibered if $F$ is closed under extensions and subobjects.
    \item[(2)] relative Segal if the essential image of $F$ is closed under extensions.
\end{itemize}
\end{cor}

\begin{remark}
    The author is not aware of a hermitian analog of the relative $S_{\bullet}$-construction constructed in Subsection~\ref{subsect:relseg}.  An analog of the construction given in Subsection~\ref{subsect:relseg} for exact $\infty$-categories does not exist for exact $\infty$-categories with duality. The problem is that the defining pullback diagram does not lift to a diagram of simplicial spaces with duality.
\end{remark}

\begin{remark}
    The relative $2$-Segal objects described in this section describe representations of Hall algebras, which have previously studied in the $1$-categorical context in \cite{young2018relative,young2020representations}. As in the case of Remark~\ref{rem:Hallalgebras}, the active equifibered and relative Segal morphisms constructed above, correspond on the level of representations of Hall algebras, to the inclusion of subrepresentations.
\end{remark}

\section{Modules in Span Categories}\label{sect:bimod}
Let $\c{C}$ be an $\infty$-category with finite limits. As sketched in the introduction of Section~\ref{sect:2-seg}, the Cartesian product on $\cC$ induces a symmetric monoidal structure on the $\infty$-category $\Span(\cC)$. The goal of this section is to provide a characterization of bimodule objects in the monoidal $\infty$-category $\Span(\c{C})^{\otimes}$ in terms of birelative $2$-Segal objects in $\c{C}$ (see Definition~\ref{def:birel}). 

The proof we present here is a multicolored version of the proof provided in \cite[Sect.2]{stern20212}, with many ideas drawn from there. For this, we use an explicit combinatorial model for the monoidal $\infty$-category $\Span(\cC)^{\otimes}$, denoted $\Span_{\Delta}(\cC^{\times})$, constructed in \cite{dyckerhoff2019higher} using quasi-categories. We recall this construction in Section~\ref{subsect:prelim}. After this preliminary Section, we prove the claimed equivalence in Section~\ref{sect:Charactofbimod}. 

Before we start, we sketch the strategy of the proof here. Since the $\infty$-category of spans is equivalent to its opposite, studying bicomodules instead of bimodules suffices. Let the functor 
\[
F:\Delta_{/[1]}\xrightarrow{} \Span_{\Delta}(\c{C}^{\times})
\]
over $\Delta$ represent a bicomodule object. Unraveling the definitions, we identify $F$ in Corollary~\ref{cor:condit} with a functor 
\[
F:\Tw(\Delta_{/[1]})\times_{\Delta} \Delta^{\amalg} \xrightarrow{} \c{C}
\]
satisfying a list of conditions. In particular, $F$ inverts a class of morphisms that we will denote by $E$. To deal with these conditions, we construct a localization functor 
\[
\c{L}:\Tw(\Delta_{/[1]})\times_{\Delta} \Delta^{\amalg} \rightarrow \Delta_{1}^{\ast}
\]
along the set of morphisms $E$. The category $\Delta_{1}^{\ast}$ contains $\Delta^{\op}_{/[1]}$ as a full subcategory and the remaining conditions translate under the restriction $\Fun(\Delta_{1}^{\ast},\c{C}) \xrightarrow{} \Fun(\Delta^{\op}_{/[1]},\c{C})$ to the birelative $2$-Segal conditions (see Definition~\ref{def:birel}). 

More explicitly we can describe this identification as follows. The category $\Tw(\Delta_{/[1]})\times_{\Delta}\Delta^{\amalg}$ has objects $(f,\{i,j\})$ represented by diagrams of the form:
\[
\triangle[{[i,j]\subset [n_{0}]}]{[n_{1}]}[{[1]}]{f}[{g^{0}}]{g^{1}}
\]
in $\Delta$. Similarly a $1$-morphism $(e^{0},e^{1}):(f_{0},[i,j]) \xrightarrow{} (f_{1},[l,m])$ in $\Tw(\Delta_{/[1]})\times_{\Delta}\Delta^{\amalg}$ can be represented by a commutative diagram in $\Delta$ of the form:
\[
\begin{tikzpicture}[auto]
                \node (A) at (0,0) {$[i,j]\subset [n_{0}]$};
                \node (B) at (3,0) {$[n_{1}]$};
                \node (C) at (3,-3) {$[m_{1}]$};
                \node (D) at (0,-3) {$[l,m] \subset [m_{0}]$};
                \node (E) at (1.5,-1.5) {$[1]$};
                \draw[->] (A) to node[scale=0.7] {$f^{0}$} (B);
                \draw[->] (C) to node[scale = 0.7,swap] {$e^{1}$} (B);
                \draw[->] (D) to node[scale = 0.7,swap] {$f^{1}$} (C);
                \draw[->] (A) to node[scale = 0.7,swap]{$e^{0}$} (D);
                \draw[->] (B) to node[scale=0.7] {$g^{1}$} (E);
                \draw[->] (A) to (E);
                \draw[->] (C) to node[scale=0.7,swap] {$h^{1}$} (E);
                \draw[->] (D) to (E);

        \end{tikzpicture}
\]
 s.t. $e_{0}(i)\leq l \leq m \leq e_{1}(j)$.  We denote by $X_{\bullet}:\Delta^{\op}_{/[1]}\xrightarrow{} \c{C}$ the birelative simplicial object associated to $F$ under the equivalence of Theorem~\ref{thm:bimodinspan}. In terms of $X_{\bullet}$ the value of $F$ on $(f,[i,j])$ admits an interpretation as:
\[
    X_{g^{1}_{[(f(i),f(i+1)]}} \times \dots \times X_{g^{1}_{[(f(j-1),f(j)]}}.
\]

 Similarly, in terms of $X_{\bullet}$ the value of $F$ on the morphism $(e^{0},e^{1})$ translates to the morphism:

\[
    X_{g^{1}_{[(f^{0}(i),f^{0}(i+1)]}} \times \dots \times X_{g^{1}_{[(f^{0}(j-1),f^{0}(j)]}} \xrightarrow{} X_{h^{1}_{[(f^{1}(l),f^{1}(l+1)]}} \times \dots \times X_{h^{1}_{[(f^{1}(m-1),f^{1}(m)]}},
\]
whose composition with the projection onto $X_{h^{1}_{[(f^{1}(l),f^{1}(l+1)]}}$ is given by:
\[
    X_{g^{1}_{[(f^{0}(i),f^{0}(i+1)]}} \times \dots \times X_{g^{1}_{[(f^{0}(j-1),f^{0}(j)]}} \xrightarrow{} X_{g^{1}_{[(f^{0}(i),f^{0}(i+1)]}} \xrightarrow{X_{e^{1}_{[f^{1}(l),f^{1}(l+1)]}}} X_{h^{1}_{[(f^{1}(l),f^{1}(l+1)]}}.
\]
We unravel this in a specific example.
\begin{example}
    Consider the objects $g:[3] \xrightarrow{0011} [1]$ and $\id_{[1]}:[1]\xrightarrow{} [1]$ in $\Delta_{/[1]}$ and denote by $X_{\bullet}$ the birelative simplicial object associated to the functor $F$. We denote by $f_{i,j}:[1] \xrightarrow{} [3]$ the unique morphism with image $\{i,j\}$ for $i\leq j$. The morphism $f_{0,3}$ extends to an object $(f_{0,3}:\id_{[1]}\xrightarrow{} g,\{0,1\})$ in $\Tw(\Delta_{/[1]})\times_{\Delta} \Delta^{\amalg}$.
Under the identification of the theorem the value of $F$ on $(f_{0,3},\{0,1\})$ identifies with 
\[
   X_{g_{[f(0),f(1)]}} = X_{g} \eqcolon X_{(0,0,1,1)}.
\]
Next, consider the morphisms
\[
\begin{tikzcd}
    {\{0,1\}\subset [1]} \arrow[rr,"f_{0,3}"]\arrow[dr] \arrow[dd,"\id_{[1]}"]& & {[3]} \arrow[dl]\\
    & {[1]} & \\
    {\{0,1\}\subset[1]} \arrow[rr, "\id_{[1]}"] \arrow[ur]& & {[1]}\arrow[uu, "f_{0,3}"] \arrow[ul]
\end{tikzcd}
\]
and 
\[
\begin{tikzcd}
    {\{0,1\}\subset [1]} \arrow[rr,"f_{0,3}"]\arrow[dr] \arrow[dd,"f_{0,3}"]& & {[3]} \arrow[dl]\\
    & {[1]} & \\
    {\{0,3\}\subset[3]} \arrow[rr, "\id_{[3]}"] \arrow[ur]& & {[3]}\arrow[uu, "\id_{[3]}"] \arrow[ul]
\end{tikzcd}
\]
in $\Tw(\Delta_{/[1]})\times_{\Delta} \Delta^{\amalg} $. Under the identification with $X_{\bullet}$, we can interpret the two morphisms $F(f_{0,3},\{0,1\}) \xrightarrow{} F(\id_{[3]},\{0,3\})$ and $F(f_{0,3},\{0,1\})\xrightarrow{} F(\id_{[1]},\{0,1\})$ in $\c{C}$ as the span
\[
\begin{tikzcd}
    &  X_{(0,0,1,1)} \arrow[dl ,swap, "X_{f_{0,1}}\times X_{f_{1,2}} \times X_{f_{2,3}}"] \arrow[dr,"X_{f}"]  & \\
    X_{(0,0)}\times X_{(0,1)} \times X_{(1,1)}& & X_{(0,1)}
\end{tikzcd}
\]
This span describes the simultaneous action of the algebra objects $X_{(0,0)}$ and $X_{(1,1)}$ on the bimodule $X_{(0,1)}$.
\end{example}

\subsection{Preliminaries}\label{subsect:prelim}

In this section, we introduce some notation and definitions essential for the combinatorics needed in the next section. Particularly, we recall an explicit construction of a monoidal structure on the $\infty$-category of spans in the model of quasi-categories. We will use this model in the following sections. First, we introduce some notation:
\begin{notation}
    Let $f:[n]\xrightarrow{} [1]$ be an object in $\Delta_{/[1]}$. We will say that $f$ is \emph{supported at $S \subset \{0,1\}$}, if $\mathsf{Im}(f) = S$. We can uniquely represent a morphism $f$ by a sequence $(0,0,...,0,1,...,1)$ with $n+1$-entries.
\end{notation}

\begin{definition}
    The \emph{interval category} $\nabla$ is the subcategory of $\Delta$ with objects given by those $[n]$ with $n\geq 1$, and morphisms are maps that preserve maximal and minimal elements, also called active maps. The \emph{augmented interval category} $\nabla^{+}$ is the wide subcategory of $\Delta$ with morphisms being the active maps.
\end{definition}

\begin{definition}\label{def:inner}
    Let $S$ be a linearly ordered set. An \emph{inner interstice} of $S$ is a pair $(n,n+1)\in S \times S$, where $n+1$ denotes the successor of $n$ in $S$. We denote the set of inner interstices of $S$ by $\mathbb{I}(S)$. This set comes equipped with a canonical linear order induced from $S$. This construction can be enhanced to a functor
    \[
\mathbb{I}:\nabla_{+}^{\op} \xrightarrow{} \Delta_{+}
    \]
    that associates to a morphism $f:S \xrightarrow{} T \in \nabla_{+}$ the morphism $\mathbb{I}(f):\mathbb{I}(T) \xrightarrow{} \mathbb{I}(S)$
    \[
        \mathbb{I}(f)(j,j+1) = (k,k+1) \quad ;f(k) \leq j < j+1 \leq f(k+1).
    \]
\end{definition}

\begin{definition}
    We define $\Delta^{\amalg}$ to be the category with objects $([n],i\leq j)$ consisting of a linearly ordered set $[n]$ and a pair of objects $i\leq j\in [n]$. A morphism $f:([n],i\leq j) \xrightarrow{} ([m],k\leq l)$ is given by a morphism $f$ in $\Delta$, s.t. $f(i)\leq k \leq l \leq f(j)$. We think of an object $([n],i\leq j)$ as a subinterval $\{i,...,j\} \subset [n]$. 

    We will frequently denote an object $([n],i\leq j)$ in $\Delta^{\amalg}$ by $([n],[i,j])$
    It is easy to check that the forgetful functor 
    \[
    \pi_{\amalg}:\Delta^{\amalg} \xrightarrow{} \Delta
    \]
    is a Cocartesian fibration. For every morphism $f:[n] \xrightarrow{} [m]$ a $\pi_{\amalg}$-cocartesian lift is given by the morphism $f:([n],i\leq j) \xrightarrow{} ([m],f(i)\leq f(j))$. 
\end{definition}

\begin{construction}\label{constr:decomp}
    Given a morphism $f:[m]\xrightarrow{} [n]$ over $[1]$ we can uniquely decompose it as follows. We can decompose the source into a composite of $[m] := \{0,1\}\ast  \{1,2\} \ast ... \ast \{m-1,m\}$. Restricting $f$ to each individual interval yields a morphism 
    \[
    f_{i}:= f|_{\{i-1,i\}}:\{i-1,i\}:= [1_{i}] \xrightarrow{} \{f(i-1)\leq f(i)\} := [n_{i}]
    \]
    in $\Delta_{/[1]}$ that preserves endpoints. We further denote $\{0\leq f(0)\}:= [n_{l}]$ and $\{f(m)\leq n\}:= [n_{r}]$. Using this decomposition, we can uniquely reconstruct $f$ as a morphism in $\Delta_{/[1]}$:
    \[
        f=f_{1} \ast...\ast f_{m} : [1_{1}]\ast... \ast [1_{m}] \xrightarrow{} [n_{1}]\ast ... \ast [n_{m}] \hookrightarrow [n_{l}] \ast [n_{1}] \ast ... \ast [n_{m}] \ast [n_{r}].
    \]
    We call this process \emph{decomposition}.
\end{construction}

 Next, we recall the construction of the Cartesian monoidal structure and the $\infty$-category of spans in the quasi-categorical model as presented in \cite[Chapt.10]{dyckerhoff2019higher}. In the following, we implicitly identify every $1$-category with its nerve.

\begin{construction}\cite[Constr.1.29]{stern20212}\label{constr:cart}
    Let $\c{C}$ be an $\infty$-category with finite products. We define a simplicial set over $\Delta$ via the adjunction formula
    \[
        \Map_{\Delta}(K,\Bar{\c{C}}^{\times}) \simeq \Map_{\mathsf{Set}_{\Delta}}(K\times_{\Delta}\Delta^{\amalg},\c{C}).
    \]
    This defines a Cartesian fibration $p:\Bar{\c{C}}^{\times} \xrightarrow{} \Delta$. We define $\c{C}^{\times}\subset \bar{\c{C}}^{\times}$ to be the full subcategory on those objects $F:\Delta^{0}\times_{\Delta}\Delta^{\amalg}\xrightarrow{} \c{C}$, that display $F(\{i\leq j\})$ as a product of $F(\{k \leq k+1\}) $ for $i \leq k < j$. The restricted functor $p:\c{C}^{\times} \xrightarrow{} \Delta$ is also a Cartesian fibration and exhibits the Cartesian monoidal structure on $\c{C}$. 

    A morphism $\Tilde{\Phi}:\Delta^{1}\xrightarrow{} \c{C}^{\times}$ represented by a map $\Phi: \{[n]\xrightarrow{} [m]\}\times_{\Delta}\Delta^{\amalg}\xrightarrow{} \c{C}$ is Cartesian if and only if $\Phi$ carries all $\pi_{\amalg}$-cocartesian  edges of  $\{[n]\xrightarrow{f} [m]\}\times_{\Delta}\Delta^{\amalg} \xrightarrow{} \Delta$ to equivalences \cite[Cor.3.2.2.12]{lurie2006higher}. This happens if and only if $\Phi$ maps all maps of the form $f:([i,j]\subset [n] )\xrightarrow{} ([f(i),f(j)]\subset [m])$ induced by a weakly monotone map $f:[n] \xrightarrow{} [m]$ to equivalences.
\end{construction}

\begin{construction}\cite[Constr.1.33]{stern20212}\label{constr:span}
    Given $X_{\bullet} \in \mathsf{Set}_{\Delta}$. There exists an adjunction:
    \[
        \Tw_{X}: (\mathsf{Set}_{\Delta})_{/X} \longleftrightarrow (\mathsf{Set}_{\Delta})_{/X}: \overline{\mathsf{Span}}_{X}
    \]
    by setting $\Tw_{X}(f:S\xrightarrow{} X) = \Tw(S) \xrightarrow{} \Tw(X) \xrightarrow{} X$ and $\overline{\mathsf{Span}}_{X}(S\xrightarrow{} X)$ to be the left vertical arrow of the pullback:
    \[
        \begin{xy}
            \xymatrix{\overline{\mathsf{Span}}_{X}(S) \ar[r] \ar[d] & \overline{\mathsf{Span}}(S) \ar[d] \\
            X \ar[r] & \overline{\mathsf{Span}}(X)
            }
        \end{xy}
    \]
    Let $p:S \xrightarrow{} X$ be a map of simplicial sets. An $n$-simplex in $\overline{\mathsf{Span}}_{X}(S)$ represented by a map $\phi:\Tw(\Delta^{n}) \xrightarrow{} S$ is called a \textit{Segal simplex} if, for every $\Delta^{k}\subset \Delta^{n}$ the composite diagram:
    \[
        \{0,k\} \ast \Tw(\mathsf{Sp}(\Delta^{k}))\subset \Tw(\Delta^{k}) \subset \Tw(\Delta^{n}) \xrightarrow{\phi} S
    \]
    is a $p$-limit diagram.  Here, $\mathsf{Sp}(\Delta^{k})$ denotes the spine of $\Delta^{k}$ and we call the join $\{0,k\} \ast \Tw(\mathsf{Sp}(\Delta^{k}))$ the \textit{Segal cone}. We denote by $\mathsf{Span}_{X}(S) \subset \overline{\mathsf{Span}}_{X}(S)$ the simplicial subset consisting of Segal simplices. In case $X\simeq \ast$ we adopt the notation $\Span(\c{C})\coloneqq \Span_{\ast}(\c{C})$. 
\end{construction}
 It follows from \cite[Thm.10.2.6]{dyckerhoff2019higher} that for every $\infty$-category $\c{C}$ the simplicial set $\Span(\c{C})$ is itself an $\infty$-category. 

\begin{prop}\cite[Prop.10.2.31]{dyckerhoff2019higher}
Let $\c{C}$ be an $\infty$-category with finite limits. Then the functor:
\[
\mathsf{Span}_{\Delta}(\c{C}^{\times}) \xrightarrow{} \Delta 
\]
is a monoidal $\infty$-category.
\end{prop}

 Note that the construction of the $\infty$-category $\Span(\c{C})$ is functorial. Indeed, every finite limit preserving functor $F:\c{C}\rightarrow \c{D}$ between $\infty$-categories $\c{C}$ and $\c{D}$ induces a functor $\Span(F): \Span(\c{C})\rightarrow \Span(\c{D})$ between $\infty$-categories of spans. In particular, after passing to homotopy coherent nerves the construction $\Span(-)$ induces an $\infty$-functor
\[
\Span(-):\sf{Cat}_{\infty}^{\lex} \rightarrow\Cat
\]
where $\sf{Cat}_{\infty}^{\lex}$ denotes the $\infty$-category of $\infty$-categories with finite limits and finite limit preserving functors.

\subsection{Characterization of Bimodules}\label{sect:Charactofbimod}
Let $\c{C}$ be an $\infty$-category with finite limits and denote by $\c{C}^{\times}$ the associated Cartesian monoidal structure as constructed in Construction~\ref{constr:cart}. Our main goal in this section is to prove the following:

\begin{theorem}\label{thm:bimodinspan}
    Let $\c{C}$ be an $\infty$-category with finite limits. There exists an equivalence of spaces:
    \[
        \BMod(\Span_{\Delta}(\c{C}^{\times}))^{\simeq} \simeq \mathsf{BiSeg}_{\Delta}(\c{C})^{\simeq}.
    \]
\end{theorem}
 To increase readability of the proof, we have included some technical lemmas in the Appendix~\ref{sect:Appendix}. As a start, let us unravel the datum of a cobimodule in $\mathsf{Span}_{\Delta}(\c{C}^{\times})$. Such an object is given by a commutative triangle:
\[
\begin{tikzcd}
    \Delta_{/[1]} \arrow[rr, "F'"]\arrow[dr] & & \Span_{\Delta}(\c{C}^{\times}) \arrow[dl] & \\
    & \Delta &
\end{tikzcd}
\]
s.t $F$ preserves inert morphisms and the adjoint morphism $\Tilde{F}: \Tw(\Delta_{/[1]}) \xrightarrow{} \c{C}^{\times} $ maps every $n$-simplex $\Delta^{n} \xrightarrow{} \Delta_{/[1]}$ to a Segal simplex. Unraveling the definitions, $F'$ corresponds by adjunction to a morphism
\[
F:\Theta_{1}\coloneqq \Tw(\Delta_{/[1]})\times_{\Delta}\Delta^{\amalg} 
\rightarrow
 \cC.
\]
We have included a precise discussion of these conditions in the Appendix~\ref{subsect:appendix1} and state here only the main result.

\begin{prop}\label{cor:condit}
    A functor $F:\Theta_{1} \xrightarrow{} \c{C}$ defines a bicomodule object, if and only if
    \begin{itemize}
        \item[(1)] $F$ sends degenerate\footnote{An interval $[i,j]\subset [n]$ is called degenerate if $i=j$.} intervals to terminal objects. 
        \item[(2)] $F$ sends every object $(\phi:g_{[n_{0}]}^{0} \xrightarrow{} g_{[n_{1}]}^{1}, [i,j])$ together with its projection to subintervals to a product diagram\footnote{See Construction~\ref{constr:cart}}.
        \item[(3)] F sends every morphism 
        \[
            \sigma \simeq 
                \begin{Bmatrix}
                    \begin{xy}
                \xymatrix{[i,j] \subset g_{n_{0}}^{0} \ar[r]^-{g}\ar[d]^-{f}&  g_{n_{1}}^{1} \\
                [\Tilde{i},\Tilde{j}] \subset g_{m_{0}}^{0} \ar[r]^-{\Tilde{g}} & g_{m_{1}}^{1} \ar[u]^-{\Tilde{f}} \\
                }
            \end{xy}
            \end{Bmatrix}
        \]
        s.t $f$ restricts to an isomorphism $\{i,...,j\} \xrightarrow{} \{\Tilde{i},...,\Tilde{j}\}$ and $\Tilde{f}$ to an isomorphism $\{g(i),...,g(j)\} \xrightarrow{} \{\Tilde{g}(\Tilde{i}),...,\Tilde{g}(\Tilde{j})\}$ to an equivalence. We denote by $E$ the set of morphisms of the above form.
        \item[(4)] $F$ maps all Segal cone diagrams from Definition~\ref{def:Segalcone} to limit diagrams.
    \end{itemize}
\end{prop}
\begin{definition}\label{def:bimodseg}
    We define several full subcategories of $\Fun(\Theta_{1},\cC)$. We denote by 
    \begin{itemize}
        \item $\BMod_{Sp}(\c{C})$ the full subcategory of $\Fun(\Theta_{1},\c{C})$ generated by those functors that satisfy the conditions of Proposition~\ref{cor:condit}.
        \item $\Fun^{\ast}(\Theta_{1},\c{C})$ the full subcategory generated by those functors that map degenerate intervals to terminal objects.
        \item $\BMod_{\pi}(\c{C})$ the full subcategory of $\Fun(\Theta_{1},\c{C})$ generated functors that satisfy conditions $(1)$ and $(2)$ of Proposition~\ref{cor:condit}.
    \end{itemize}
    Furthermore, we denote by $\Omega_{1}$ the full subcategory of $\Theta_{1}$ with objects those morphisms $([i,j]\subset g_{n_{0}}^{0} \xrightarrow{} g_{n_{1}}^{1})$, s.t. the interval $[i,j]$ is non-degenerate.
\end{definition}

 We first analyze condition $(1)$:

\begin{prop}
    The restriction functor induces an equivalence of $\infty$-categories 
    \[
        \Fun^{\ast}(\Theta_{1},\c{C}) \xrightarrow{\simeq} \Fun(\Omega_{1},\c{C})
    \]
\end{prop}
    \begin{proof}
        The proof is the same as \cite[Cor.2.9]{stern20212}.
    \end{proof}

 To handle condition $(3)$ of Proposition~\ref{cor:condit}, we explicitly construct a localization of $\Omega_{1}$ along the class of morphisms $E$.

\begin{definition}
    The category $\Delta_{1}^{\star}$ has objects consisting of pairs of a finite ordered set $[k]$ and a $[k]$-indexed sequence of composable\footnote{See Definition~\ref{constr:decomp}} objects $(f_{n_{0}},...,f_{n_{k}})$ in $\Delta_{/[1]}$. By definition, these define a morphism 
    \[
        f:= f_{n_{0}}\ast ...\ast f_{n_{k}}:[n_{0}]\ast[n_{1}] \ast ... \ast [n_{k}] \xrightarrow{} [1].
    \]
    A morphism $(g,\theta): ((f^{1}_{n_{0}},...f^{1}_{n_{k}}),[k]) \xrightarrow{} ((f^{2}_{m_{0}},...,f^{2}_{m_{l}}),[l])$ consists of a 
    \begin{itemize}
        \item[(1)] a morphism $\theta:[l]\xrightarrow{} [k]$.
        \item[(2)] a commutative diagram
        \[
        \begin{tikzcd}
            {[m_{0}]\ast ... \ast [m_{l}]} \arrow[rr, "g"] \arrow[dr,swap,"f^{2}"] & &  {[n_{0}] \ast ... \ast [n_{k}]} \arrow[dl, "f^{1}"] \\
            & {[1]} & 
        \end{tikzcd}
        \]
        s.t. for any $i \in [k]$ with $\theta^{-1}(i) = \{j_{1},...,j_{p}\}$, the restriction
        \[
            g_{i}:[m_{j_{1}}]\ast ... \ast [m_{j_{p}}] \xrightarrow{} [n_{0}] \ast ... \ast [n_{k}]
        \]
        has image contained in $[n_{i}]$.
    \end{itemize}
\end{definition}
\begin{construction}\label{constr:L}
    We define a functor $\mathcal{L}:\Omega_{1} \xrightarrow{} \Delta^{\ast}_{1}$ as follows:
    \begin{itemize}
        \item It maps an object $[i,j] \subset g_{n_{0}}^{0} \xrightarrow{f} g_{n_{1}}^{1}$ to the object $(g_{[f(i),f(i+1)]}^{1},...,g_{[f(j-1),f(j)]}^{1})$. The image is indexed by the inner interstices\footnote{See Definition~\ref{def:inner}} $\mathbb{I}([i,j])$ of the linearly ordered set $\{i,i+1,...,j\}$.
        \item A morphism in $\Omega_{1}$ is given by
           \[
             \sigma \simeq     
             \begin{Bmatrix}
               \begin{xy}
                \xymatrix{[i,j] \subset g_{n_{0}}^{0} \ar[r]^-{g}\ar[d]^-{f}&  g_{n_{1}}^{1} \\
                [\Tilde{i},\Tilde{j}] \subset g_{m_{0}}^{0} \ar[r]^-{\Tilde{g}} & g_{m_{1}}^{1} \ar[u]^-{\Tilde{f}} \\
                }
            \end{xy}
            \end{Bmatrix}.
        \]
        The functor $\mathcal{L}$ maps $\sigma$ to the map $\mathcal{L}(\sigma)=(h_{\sigma},\mathbb{I}(f))$, whose second component is given by \[
        \mathbb{I}(f)~:~\mathbb{I}([\Tilde{i},\Tilde{j}]) \xrightarrow{} \mathbb{I}([i,j]).
        \]
         We define $h_{\sigma}$ componentwise. Consider an inner interstice $\{p,p+1\}~\in~\mathbb{I}([i,j])$. The restriction of $h_{\sigma}$ to the component indexed by $\mathbb{I}(f)^{-1}(\{p,p+1\})$ is given by:
\[
\begin{tikzcd}
    {[\Tilde{g}(f(p)),\Tilde{g}(f(p)+1)] \ast ... \ast [\Tilde{g}(f(p+1)-1),\Tilde{g}(f(p+1))]} \arrow[rr,"\Tilde{f}"] \arrow[dr,swap, "g^{1}_{[\Tilde{g}(f(p)), \Tilde{g}(f(p+1))]}"]& & {[g(p),g(p+1)]} \arrow[dl, "g^{1}_{[g(p),g(p+1)]}"]\\
    & {[1]} &
\end{tikzcd}
\]
It follows from the commutativity of $\sigma$ that this yields a morphism in $\Delta_{1}^{\ast}$.
\end{itemize}
\end{construction}
 The technical proof of the following proposition is included in Appendix~\ref{subsect:appendix2}:
\begin{prop}\label{prop:localization}
    The functor $\mathcal{L}:\Omega_{1}\xrightarrow{} \Delta^{\star}_{1}$ is an $\infty$-categorical localization at the morphisms $E$ defined in Proposition~\ref{cor:condit} (3).
\end{prop}

\begin{definition}\label{def:bimod}
    Let $\c{C}$ be an $\infty$-category with finite limits. We define the category $\Fun^{\bim}(\Delta^{\star}_{1},\c{C}) \subset \Fun(\Delta^{\star}_{1},\c{C})$ as the full subcategory with objects those functors, that 
    \begin{itemize}
        \item[(I)] send diagrams of the form
        \[
            \begin{xy}
                \xymatrix{& & (g_{0},...,g_{k}) \ar[dr] \ar[drr] \ar[dl] \ar[dll] & & \\
                g_{0} & g_{1}  & ... &  g_{k-1} & g_{k}
                }
            \end{xy}
        \]
        to equivalences.
        \item[(II)] send diagrams of the form
        \[
            \begin{xy}
                \xymatrix{
                & (\ast_{j\in [n_{1}]}g_{j}^{2}, ..., \ast_{j\in [n_{l}]}g_{j}^{2}) \ar[dl] \ar[dr] & \\
                (g_{n_{1}}^{1},...,g_{n_{l}}^{1}) \ar[dr] & & (g_{1}^{2},...,g_{k}^{2}) \ar[dl] \\
                & (g_{1}^{1}|_{\{0,1\}},...,g_{1}^{l}|_{\{n_{l}-1,n_{l}\}}) & 
                }
            \end{xy}
        \]
        to equivalences.
    \end{itemize}
\end{definition}

\begin{prop}\label{prop:Lisanequiv}
    Restriction along the functor $\mathcal{L}:\Omega_{1} \rightarrow \Delta_{1}^{\ast}$ induces an equivalence of $\infty$-categories
    \[
        \BMod_{Sp}(\c{C}) \simeq \Fun^{\bim}(\Delta^{\star}_{1},\c{C}).
    \]
\end{prop}
\begin{proof}
    It follows from Proposition~\ref{prop:localization} that restriction along $\mathcal{L}$ induces an equivalence
   \[
        \Fun_{E^{-1}}(\Omega_{1},\c{C}) \simeq \Fun(\Delta_{1}^{\star},\c{C}),
   \]
    where we denote by $\Fun_{E^{-1}}(\Omega_{1},\c{C})$ the full subcategory of $\Fun(\Omega_{1},\c{C})$ on those functors, that map all morphisms in $E$ to equivalences. To conclude, we only have to match the remaining conditions of Proposition~\ref{cor:condit}. By construction, condition (2) on the left matches with condition (I) on the right. For condition (4) let $\sigma$ be a $3$-simplex 
    \[
        g^{0}_{n_{0}} \xrightarrow{\phi_{1}} g^{1}_{n_{1}} \xrightarrow{\phi_{2}} g^{2}_{n_{2}} \xrightarrow{\phi_{3}} g^{3}_{n_{3}} 
   \]
    in $\Omega_{1}$ and $F$ an object of $\BMod_{Sp}(\c{C})$. Further let $[i,j] \subset [n_{0}]$ be a subinterval. The corresponding limit diagram for condition (4) reads as
    \[
    \begin{tikzpicture}
        \node (A) at (0,0) {$F(\psi_{3},[i,j])$};
        \node (B) at (-4,-1.5) {$F(\phi_{1}, [i,j])$};
        \node (C) at (0,-1.5) {$F(\phi_{2}, [\psi_{1}(i),\psi_{1}(j)])$};
        \node (D) at (4,-1.5) {$F(\phi_{3}, [\psi_{2}(i),\psi_{2}(j)])$};
        \node (E) at (-2,-3) {$F(g^{1}_{n_{1}},[\psi_{1}(i),\psi_{1}(j)])$};
        \node (F) at (2,-3) {$F(g^{2}_{n_{2}},[\psi_{2}(i),\psi_{2}(j)])$};

        \draw[->] (A) to (B);
        \draw[->] (A) to (C);
        \draw[->] (A) to (D);
        \draw[->] (B) to (E);
        \draw[->] (C) to (E);
        \draw[->] (C) to (F);
        \draw[->] (D) to (F);
        
    \end{tikzpicture}
    \]
    It follows from the dual of \cite[Prop.4.4.2.2]{lurie2006higher} that this diagram is a limit diagram, if
     and only if the diagram
    \[
    \begin{tikzcd}
        & F(\psi_{3},[i,j]) \arrow[dr] \arrow[dl]& \\
      F(\psi_{2},[i,j])\arrow[dr]  & & F(\phi_{3} \circ \phi_{2}, [\phi(i),\phi(j)]) \arrow[dl] \\
       & F(\phi_{2},[\phi_{1}(i),\phi_{1}(j)]) &   
    \end{tikzcd}
    \]
    is a pullback diagram. Combining this with the previous diagram, we obtain the following diagram:
    \[
    \begin{tikzpicture}
        \node (A) at (0,0) {$F(\psi_{3},[i,j])$};
        \node (B) at (-4,-3) {$F(\phi_{1}, [i,j])$};
        \node (C) at (0,-3) {$F(\phi_{2}, [\psi_{1}(i),\psi_{1}(j)])$};
        \node (D) at (4,-3) {$F(\phi_{3}, [\psi_{2}(i),\psi_{2}(j)])$};
        \node (E) at (-2,-4.5) {$F(g^{1}_{n_{1}},[\psi_{1}(i),\psi_{1}(j)])$};
        \node (F) at (2,-4.5) {$F(g^{2}_{n_{2}},[\psi_{2}(i),\psi_{2}(j)])$};
        \node (G) at (-2,-1.5) {$F(\psi_{2},[i,j])$};
        \node (H) at (2,-1.5) {$F(\phi_{3} \circ \phi_{2}, [\phi(i),\phi(j)])$};

        \draw[->] (A) to (G);
        \draw[->] (A) to (H);
        \draw[->] (B) to (E);
        \draw[->] (C) to (E);
        \draw[->] (C) to (F);
        \draw[->] (D) to (F);
        \draw[->] (G) to (B);
        \draw[->] (G) to (C);
        \draw[->] (H) to (C);
        \draw[->] (H) to (D);
        
    \end{tikzpicture}
    \]
    It follows from the pasting property for pullbacks, that the diagram corresponding to any sub 2-simplex is a limit diagram. Iterating this argument, we see that condition (4) is satisfied if and only if it is satisfied on 2-simplices. The claim follows from the observation that condition (II) is precisely the image of condition (4) on $2$-simplices under $\mathcal{L}$.
\end{proof}
    Using the above proposition, the problem has shifted to analyzing the conditions of Definition~\ref{def:bimod}. To this end, consider the full subcategory $i:\Delta_{/[1]}^{\op} \xrightarrow{} \Delta^{\star}_{1}$ generated by objects $f:[n]\xrightarrow{} [1] \in \Delta_{/[1]}$ with $n\geq 0$. Restriction and right Kan extension induce an adjunction
    \[
        i^{\ast}:\Fun(\Delta^{\star}_{1},\c{C}) \rightleftarrows \Fun(\Delta^{\op}_{/[1]},\c{C}):i_{\ast}.
    \]
    We denote by $\Fun^{\times}(\Delta^{\star}_{1},\c{C})\subset \Fun(\Delta^{\star}_{1},\c{C})$ the full subcategory on those functors that satisfy condition (I) of Definition~\ref{def:bimod}.
\begin{prop}
    The functor $i_{\ast}:\Fun(\Delta_{/[1]}^{\op},\c{C}) \rightarrow \Fun(\Delta_{1}^{\star},\c{C})$ factors through the full subcategory $\Fun^{\times}(\Delta_{1}^{\star},\c{C})$. Moreover, the induced functor
    \[
    i_{\ast}:\Fun(\Delta_{/[1]}^{\op},\c{C}) \rightarrow \Fun(\Delta_{1}^{\star},\c{C})
    \]
    is an equivalence of $\infty$-categories.
\end{prop}
\begin{proof}
 Let $F\in \Fun(\Delta^{\op}_{/[1]},\c{C})$ be a functor and let $(f_{m_{1}}, f_{m_{2}},...,f_{m_{k}})$ be an object of $ \Delta_{1}^{\star}$. The value of $i_{\ast}F$ on $(f_{m_{1}}, f_{m_{2}},...,f_{m_{k}})$ can be computed as the limit of $F$ over the diagram indexed by the category 
 \[
 \c{D}:=(\Delta^{\op}_{/[1]})_{/(f_{m_{1}}, f_{m_{2}},...,f_{m_{k}})}.
 \]
 An object of the category $\c{D}$ consists of an element $i \in \{1,...,k\}$ together with a morphism $h_{l} \xrightarrow{} f_{m_{i}}$ in $\Delta_{/[1]}$. Note that a morphism from $(h_{l}\xrightarrow{} f_{m_{i}})$ to $(e_{p} \xrightarrow{} f_{m_{j}})$ in $\c{D}$ only exists if $i=j$. In this case, it is given by a commutative diagram
 \[
 \triangle[{e_{p}}]{h_{l}}[{f_{m_{i}}}]{}[{}]{}
 \]
 in $\Delta_{/[1]}$. It follows, that the value of $i_{\ast}F(f_{m_{1}},...,f_{m_{k}})$ is given by the limit
 \[
 \begin{gathered}
     \begin{xy}
             \xymatrix@C=0 cm{ & & i_{\ast}F(f_{m_{1}},...,f_{m_{k}}) \ar[dll] \ar[dl] \ar[dr] \ar[drr] & & \\
             F(f_{m_{1}}) & F(f_{m_{2}}) & ... & F(f_{m_{k-1}}) & F(f_{m_{k}}) }
         \end{xy}
    \end{gathered}
 \]
 This proves the first claim. The second claim follows from \cite[Prop.4.3.2.15]{lurie2006higher}.
\end{proof}

 We need the following auxiliary lemma:
\begin{lemma}\label{lem:Diagram}
    Let $F\in \Fun(\Delta_{1}^{\star},\c{C})$ be a functor. Then $F$ satisfies condition $(II)$, if and only if it satisfies condition $(II)$ where all but one of the $g_{2}^{i}$ have source equal to $[1]$. We will call this condition $(II')$.
\end{lemma}

\begin{proof}
    It follows from the assumptions that $(II)$ implies $(II')$. We assume that $F$ satisfy condition $(II')$. 
    We consider the diagram in $\Delta^{\star}_{1}$ displayed in Figure~\ref{fig:10}.
    \begin{figure}[ht]
    \centering
    \includegraphics[page=1, width= \textwidth]{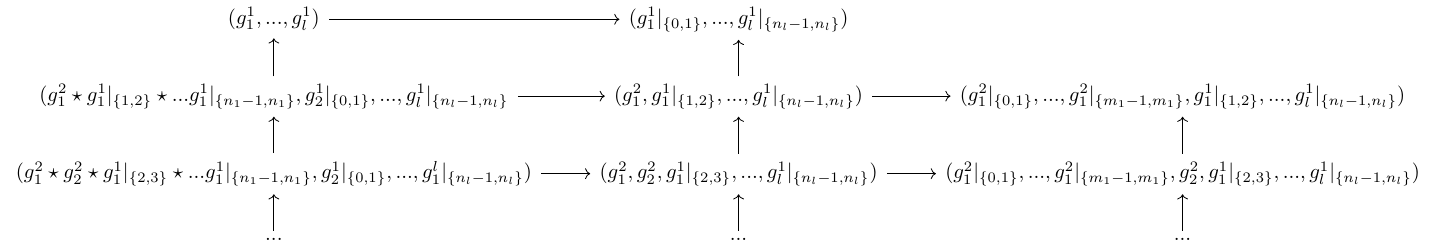}
    \caption{Proof of Lemma~\ref{lem:Diagram}}
    \label{fig:10}
    \end{figure}
    Since $F$ satisfies $(II')$, it follows that $F$ maps the bottom right square and the exterior right rectangle of the diagram to pullback diagrams. By the pasting law, it follows that the functor $F$ also sends the bottom left square to a pullback square. Hence, it follows from $(II')$ that the upper left square is also a pullback square. Another application of the pasting law implies that the vertical left square is a pullback square. Iterating this argument yields the claim.
\end{proof}
 After all these intermediate steps, we can finally prove:
\begin{prop}\label{prop:bimisbiseg}
    The restriction functor $i^{\ast}:\Fun(\Delta^{\ast}_{1},\c{C}) \xrightarrow{} \Fun(\Delta^{\op}_{/[1]},\c{C})$ descends to an equivalence
    \[
        \Fun^{\bim}(\Delta^{\star}_{1},\c{C}) \simeq \mathsf{BiSeg}_{\Delta}(\c{C})
    \]
    where $\sf{BiSeg}_{\Delta}(\c{C}) \subset \Fun(\Delta^{\op}_{/[1]},\c{C})$ denotes the full subcategory generated by birelative $2$-Segal objects. 
\end{prop}
\begin{proof}
    Let $F \in \Fun^{\times}(\Delta^{\star}_{1},\c{C})$ be a functor and let 
    \[
        \begin{xy}
            \xymatrix{g_{1} \ar[d] & \ast_{0\leq k < j-1} g_{1}|_{\{k,k+1\}} \ast g_{2} \ast (\ast_{j\leq k< n} g_{1}|_{k,k+1})\ar[d] \ar[l] \\
            (g_{1}|_{\{0,1\}},...,g_{1}|_{\{n-1,n\}})  & (g_{1}|_{\{0,1\}},...,g_{1}|_{\{j-2,j-1\}},g_{2},g_{1}|_{\{j,j+1\}},...,g_{1}|_{\{n-1,n\}}) \ar[l]
            }
        \end{xy}
    \]
    be a diagram of type $(II')$. We can expand this diagram  
    \[
        \begin{xy}
            \xymatrix{g_{1} \ar[d] & \ast_{0\leq k < j-1} g_{1}|_{\{k,k+1\}} \ast g_{2} \ast (\ast_{j\leq k< n} g_{1}|_{k,k+1})\ar[d] \ar[l] \\
            (g_{1}|_{\{0,1\}},...,g_{1}|_{\{n-1,n\}})  \ar[d]& (g_{1}|_{\{0,1\}},...,g_{1}|_{\{j-2,j-1\}},g_{2},g_{1}|_{\{j,j+1\}},...,g_{1}|_{\{n-1,n\}}) \ar[d] \ar[l] \\
            g_{1}|_{\{j-1,j\}} & g_{2} \ar[l] \\
            }
        \end{xy}
    \]
    Since $F$ satisfies condition $(I)$ of Definition~\ref{def:bimod}, it follows that $F$ maps the lower square to a pullback diagram. By the pasting lemma, the upper square is a pullback if and only if the outer rectangle is a pullback. But the latter can be identified with the opposite of the diagram
    \[
    \begin{tikzpicture}[auto]
        \node (A) at (-2,1.5) {$[n]$};
        \node (B) at (-2,-1.5) {$[1]$};
        \node (C) at (0,0) {$[1]$};
        \node (D) at (2,-1.5) {$[m]$};
        \node (E) at (2,1.5) {$[n+m-1]$};

        \draw[->](B) to (A);
        \draw[->](B) to (D);
        \draw[->](D) to (E);
        \draw[->](A) to (E);
        \draw[->](A) to node[scale=0.7] {$g_{1}|_{\{j-1,j\}}$}(C);
        \draw[->](B) to node[scale=0.7] {$g_{1}$} (C);
        \draw[->](D) to node[scale=0.7,swap] {$g_{2}$} (C);
        \draw[->](E) to (C);
    \end{tikzpicture}
    \]
    which precisely recovers the birelative Segal conditions.
\end{proof}

As a corollary, we obtain analogous results for left and right module objects in $\Span_{\Delta}(\c{C}^{\times})$:

\begin{cor}
    Let $\c{C}$ be an $\infty$-category with finite limits. There exist equivalences of spaces
    \[
        \LMod(\Span_{\Delta}(\c{C}^{\times}))^{\simeq} \simeq \mathsf{LSeg}_{\Delta}(\c{C})^{\simeq},
    \]
    and 
    \[
        \RMod(\Span_{\Delta}(\c{C}^{\times}))^{\simeq} \simeq \mathsf{RSeg}_{\Delta}(\c{C})^{\simeq}.
    \]
\end{cor}

\section{Module Morphisms in Span Categories}\label{sect:classificatin of module morphisms}
In the last section, we have constructed for every $\infty$-category with finite limits $\c{C}$ an equivalence between the space of bimodule objects in the monoidal $\infty$-category $\Span_{\Delta}(\c{C}^{\times})$ and the space of birelative 2-Segal objects. The goal of this section is to extend this to an equivalence of $\infty$-categories:

\begin{theorem}\label{thm:birseg}
    Let $\c{C}$ be an $\infty$-category with finite limits. There exists an equivalence of $\infty$-categories
    \[
        \mathsf{BiMod}(\Span_{\Delta}(\c{C}^{\times})) \simeq \operatorname{Bi2Seg_{\Delta}^{\leftrightarrow}}(\c{C}) 
    \]
    between the $\infty$-category of bimodule objects in $\Span_{\Delta}(\c{C}^{\times})$ and the $\infty$-category of birelative $2$-Segal objects and birelative $2$-Segal spans.
\end{theorem}

As a corollary of the above characterization, we obtain a similar characterization for the $\infty$-category of left (resp. right) module objects and algebra objects in $\Span_{\Delta}(\c{C}^{\times})$.

 \begin{cor}\label{cor:lmod}
    Let $\c{C}$ be an $\infty$-category with finite limits. There exists an equivalence 
    \[
        \LMod(\Span_{\Delta}(\c{C}^{\times})) \simeq \operatorname{L2Seg_{\Delta}^{\leftrightarrow}}(\c{C})
    \]
    between the $\infty$-category of left modules in $\Span_{\Delta}(\c{C}^{\times})$ and the $\infty$-category of left relative $2$-Segal objects and left relative $2$-Segal spans.
\end{cor}

\begin{cor}\label{cor:2seg}
    Let $\c{C}$ be an $\infty$-category with finite limits. There exists an equivalence 
    \[
        \Alg(\Span_{\Delta}(\c{C}^{\times})) \simeq \TSegSpan(\c{C})
    \]
    between the $\infty$-category of algebras in $\Span_{\Delta}(\c{C}^{\times})$ and the $\infty$-category of $2$-Segal objects and $2$-Segal spans.
\end{cor}

 Combining Corollary~\ref{cor:lmod} with Corollary~\ref{cor:lrelandrel}, we obtain

\begin{cor}\label{cor:rel2segismodule}
    Let $\c{C}$ be an $\infty$-category with finite limits. There exists an equivalence 
    \[
        \LMod(\Span_{\Delta}(\c{C}^{\times})) \simeq \operatorname{Rel2Seg_{\Delta}^{\leftrightarrow}}(\c{C})
    \]
    between the $\infty$-category of left modules in $\Span_{\Delta}(\c{C}^{\times})$ and the $\infty$-category of relative $2$-Segal objects.
\end{cor}

 Before we prove these statements, let us explicitly describe the relation between algebra morphisms and Segal spans at the level of lowest dimensional coherence. To this end, consider two algebra objects $(X_{1},\mu^{X})$ and $(Y_{1},\mu^{Y})$ in $\Span(\c{C}^{\times})$ with associated $2$-Segal objects $X_{\bullet}$ (resp. $Y_{\bullet}$). To construct the data of an algebra morphism $F:(X_{1},\mu^{X}) \xrightarrow{} (Y_{1},\mu^{Y})$ between those, we first have to provide a morphism on underlying objects. This is given by a span:
\[
\begin{tikzcd}
    & F_{1} \arrow[dr,"g_{1}"] \arrow[dl,swap, "f_{1}"] & \\
    X_{1} & & Y_{1}
\end{tikzcd}
\]
with tip denoted $F_{1}$.
For this span to be part of an algebra morphism, we have to provide higher coherence data. At the lowest level this is given by an invertible $2$-morphism 
\[
\alpha:F_{1}\circ \mu^{X} \simeq \mu^{Y} \circ (F_{1}\times F_{1})
\] in the $\infty$-category $\Span(\cC)$, i.e there has to exist an object $F_{2}\in \cC$ together with two invertible $1$-morphisms:
\begin{equation}\label{eq:1}
\begin{tikzcd}
& F_{2} \arrow[dr, "(\del_{1}^{F}{,}f_{2})"] \arrow[dl,swap, "(g_{2}{,}\del_{2}^{F}{,}\del_{0}^{F})"] & \\
Y_{2} \times_{Y_{1}\times Y_{1}} F_{1}\times F_{1} & & X_{2}\times_{X_{1}} F_{1}
\end{tikzcd}
\end{equation}  
fitting into a commutative diagram 
\[
\begin{tikzpicture}[auto]
    \node (A) at (0,0) {$X_{1}\times X_{1} $};
    \node (B) at (6,0) {$Y_{1} \times Y_{1}$};
    \node (C) at (12,0) {$Y_{1}$};
    \node (D) at (3.5,1.6) {$X_{2}$};
    \node (E) at (10,1.4) {$F_{1}$};
    \node (F) at (8,2) {$X_{2} \times_{X_{1}} F_{1}$};
    \node (G) at (3,1) {$F_{1}\times F_{1}$};
    \node (H) at (9,1) {$Y_{2}$};
    \node (I) at (6.5,1.6) {$Y_{2} \times_{Y_{1}\times Y_{1}} F_{1} \times F_{1}$};
    \node (J) at (7.3,2.8) {$F_{2}$};
    \node (K) at (6.25,0.75) {$X_{1}$};
    \node[scale=0.7] (L) at (7.3,2.4) {$\alpha$};
    \node[scale= 0.7] (M) at (7.6,1.2) {$f_{1}$};
    \draw[->] (D) to [bend right=10] node[swap,scale=0.7] {$(\del^{X}_{2},\del^{X}_{0})$} (A);
    \draw[->,dotted] (D) to [bend left= 15] node[swap, scale=0.7] {$\del^{X}_{1}$} (K);
    \draw[->,dotted] (E) to [bend right= 20]   (K);
    \draw[->] (E) [bend left= 10] to node[scale=0.7] {$g_{1}$} (C);
    \draw[dotted, ->] (F) to [bend right= 10] (D);
    \draw[->] (F) to [bend left= 10] node[scale=0.7] {$\pi_{2}$}(E);
    \draw[->] (G) to [bend right= 10] node[scale=0.7] {$f_{1}\times f_{1}$} (A);
    \draw[->] (G) to [bend left =10] node[swap,scale=0.7] {$g_{1}\times g_{1}$} (B);
    \draw[->] (H) to [bend right= 10] node[ scale=0.7] {$(\del^{Y}_{2},\del^{Y}_{0})$} (B);
    \draw[->] (H) to [bend left = 10] node[swap,scale=0.7] {$\del^{Y}_{1}$} (C);
    \draw[->] (I) [bend right = 10] to (G);
    \draw[->] (I) to [bend left =10] (H);
    \draw[->] (J) to [bend left = 10] node[scale= 0.7] {$(\del^{F}_{1},f_{2})$}  (F);
    \draw[->] (J) to [bend right= 20] node[swap, scale= 0.7] {$(g_{2},\del^{F}_{2},\del^{F}_{0})$}  (I);
    \draw[red, ->] (7.6,2.2) ..  controls (7.3,2.8)  and  (7.3,2.8) .. (6.75,1.9);
\end{tikzpicture}
\]
 As the notation suggests, these data define the $2$-simplices, and the $2$-dimensional face maps of the simplicial object $F_{\bullet}$, and the $2$-dimensional part of the maps $f_{\bullet}:F_{\bullet}\rightarrow X_{\bullet}$, and $g_{\bullet}:F_{\bullet}\rightarrow Y_{\bullet}$. The equivalences in \eqref{eq:1} form the lowest dimensional instances of the active equifibered and relative Segal conditions, respectively (see Definition~\ref{def:2-segspan}). Similarly, the higher simplices of the simplicial object $F_{\bullet}$, as well as the higher active equifibered and relative Segal conditions, imposed on morphisms $f_{\bullet}$ and $g_{\bullet}$ are encoded in the higher coherence data.

After these initial considerations, we now turn to the proof of Theorem~\ref{thm:birseg}. For this, we will need the following Lemma:
\begin{lemma}\label{lem:aux}
    Let $\c{C}$ be an $\infty$-category with finite limits and consider a morphism:
    \[
        F:\Tw(\Delta^{n}) \times (\Tw(\Delta_{/[1]})\times_{\Delta} \Delta^{\amalg})\xrightarrow{} \c{C}
    \]
    such that the associated morphism $\Tilde{F}:\Tw(\Delta^{n}) \xrightarrow{} \Fun((\Tw(\Delta_{/[1]})\times_{\Delta} \Delta^{\amalg}),\c{C})$ factors through $\mathsf{BMod}_{Sp}(\c{C})$. Then $F$ defines a morphism:
    \[
        F:\Delta_{/[1]} \times \Delta^{n} \xrightarrow{} \Span_{\Delta}(\c{C})
    \]
    if and only if for every $2$-simplex $\Delta^{2} \xrightarrow{} \Delta^{n}\times \Delta_{/[1]}$ of the form
    \begin{itemize}
        \item[$(1)$] $(f_{0}:[n_{0}]\xrightarrow{} [1],i) \xrightarrow{\phi_{1}} (f_{1}:[n_{1}]\xrightarrow{} [1],i) \xrightarrow{\phi_{2}} (f_{2}:[n_{2}]\xrightarrow{}[1],i) $
        \item[$(2)$] $(f_{0}:[n_{0}]\xrightarrow{} [1],i) \xrightarrow{i<j} (f_{0}:[n_{0}]\xrightarrow{} [1],j) \xrightarrow{j<k} (f_{0}:[n_{0}]\xrightarrow{}[1],k) $
        \item[$(3)$] $(f_{0}:[n_{0}]\xrightarrow{} [1],i) \xrightarrow{\phi_{1}} (f_{1}:[n_{1}]\xrightarrow{} [1],i) \xrightarrow{i<j} (f_{2}:[n_{2}]\xrightarrow{}[1],j)$
        \item[$(4)$] $(f_{0}:[n_{0}]\xrightarrow{} [1],i) \xrightarrow{i<j} (f_{0}:[n_{0}]\xrightarrow{} [1],j) \xrightarrow{\phi_{1}} (f_{1}:[n_{1}]\xrightarrow{}[1],j)$ 
    \end{itemize}
    the restriction $F|_{\Delta^{2}}$ is a Segal simplex.
\end{lemma}
\begin{proof}
    The only if condition follows from the assumption. For the other direction, we first observe that for every morphism $\phi:f_{n_{0}} \xrightarrow{} f_{n_{1}}$ in $\Delta_{/[1]}$, every $i<j<k$, and every fixed subinterval $[l,m] \subset [n_{0}]$ the diagram
    \[
    \squarevert[{F(\phi,i<k,[l,m])}]{F(\phi,j<k,[l,m])}[{F(\phi,i<j,[l,m])}]{F(\phi,j,[l,m])}
    \]
    is a limit diagram. Indeed, since $\Tilde{F}$ factors through $\mathsf{Bim}_{Sp}(\c{C})$, the diagram is equivalent to a product of diagrams of the form
    \[
    \squarevert[{F(f_{n_{1}}|_{\phi(p),\phi(p+1)},i<k,[\phi(p),\phi(p+1)])}]{F(f_{n_{1}}|_{\phi(p),\phi(p+1)},j<k,[\phi(p),\phi(p+1)])}[{F(f_{n_{1}}|_{\phi(p),\phi_(p+1)},i<j,[\phi(p),\phi(p+1)])}]{F(f_{n_{1}}|_{\phi(p),\phi_(p+1)},j,[\phi(p),\phi(p+1)])}
    \]
    These are pullback diagrams due to condition $(II)$. Let $\sigma:\Delta^{2} \xrightarrow{} \Delta_{/[1]}\times \Delta^{n}$ be a $2$-simplex represented by a composable pair of morphisms:
    \[
        (f_{0}:[n_{0}]\xrightarrow{} [1],i) \xrightarrow{\phi_{1}} (f_{1}:[n_{1}]\xrightarrow{} [1],j) \xrightarrow{\phi_{2}} (f_{2}:[n_{2}]\xrightarrow{}[1],k). 
    \]
    By Proposition~\ref{prop:Segalcone} we need to check that for every interval $[l,m]\subset [n_{0}] $ the following square is a pullback diagram: 
    \[
    \squarevert[{F(\psi_{2},i<k,[l,m])}]{F(\phi_{2},j<k,[\phi_{1}(l),\phi_{1}(m)])}[{F(\phi_{1},i<j,[l,m])}]{F(f_{1},j,[\phi_{1}(l),\phi_{1}(m)])}
    \]
   \begin{figure}[t]
    \centering
    \includegraphics[page=1, width= \textwidth]{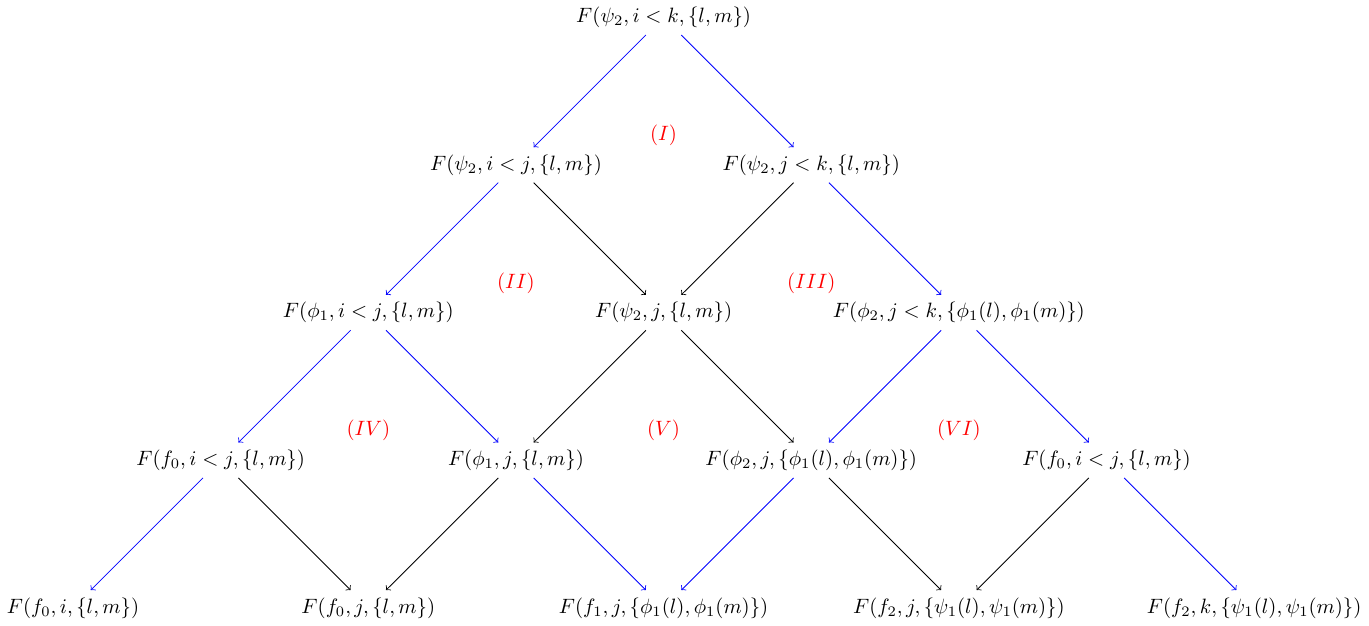}
    \caption{Image under $F$ of the decomposition of the $2$-simplex $\sigma$ into a $4$-simplex. The original $2$-simplex is colored blue. }
    \label{fig:compspan}
\end{figure}
   
  To check this, we include this diagram in a fourfold composite of spans shown in Figure~\ref{fig:compspan}. The square is depicted by the blue part in the diagram. We denote the composite rectangles, obtained by pasting two squares, by the sum of their labels. It follows from $(4)$, that the rectangles $(IV)$, $(IV)+(II)$, from (3) that the rectangles $(III)+(VI)$, $(VI)$, and from $(1)$ that the square $(V)$ are all pullback squares. Moreover, as shown above, the square $(I)$ is a pullback diagram. It follows from an iterated application of the pasting lemma that the blue square is a pullback square.
\end{proof}

\begin{proof}[Proof of Theorem~\ref{thm:birseg}]
We prove the equivalence $\mathsf{BicoMod}(\Span_{\Delta}(\c{C}^{\times})) \simeq  \operatorname{Bi2Seg_{\Delta}^{\leftrightarrow}}(\c{C})^{\op} $. The claimed result then follows from the equivalence $\mathsf{BicoMod}(\Span_{\Delta}(\c{C}^{\times}))\simeq \mathsf{BiMod}(\Span_{\Delta}(\c{C}^{\times}))^{\op}$.  

    Unraveling definitions and using Proposition~\ref{prop:Segalcone}, an $n$-simplex $\eta \in \mathsf{BicoMod}(\Span_{\Delta}(\c{C}^{\times}))_{n}$ is represented by a morphism
    \[
        \eta:\Tw(\Delta^{n}) \times (\Tw(\Delta_{/[1]})\times_{\Delta} \Delta^{\amalg}) \xrightarrow{} \c{C}
    \]
    such that 
    \begin{itemize}
        \item[(1)]  for every $0 \leq k \leq n$ the restriction $\eta_{\{k\}}:\{\id_{k}\}\times (\Tw(\Delta_{/[1]})\times_{\Delta} \Delta^{\amalg}) \xrightarrow{} \c{C} $ preserves inert morphisms, i.e. defines a bicomodule in $\Span_{\Delta}(\c{C}^{\times})$. 
        \item[(2)] for every 2-simplex $\Delta^{2} \xrightarrow{} \Delta_{/[1]} \times \Delta^{n}$ of the form of Lemma~\ref{lem:segalandlimits} depicted 
        \[
        (f_{0}:[n_{0}]\xrightarrow{} [1],i) \xrightarrow{\phi_{1}} (f_{1}:[n_{1}]\xrightarrow{} [1],j) \xrightarrow{\phi_{2}} (f_{2}:[n_{2}]\xrightarrow{}[1],k) 
        \]
        with $0\leq i \leq j\leq k \leq n$ and every subinterval $[p,q]\subset [n_{0}]$ the associated Segal cone diagram\footnote{See Definition~\ref{def:Segalcone}} is a limit diagram:
    \end{itemize}
    \begin{equation}\label{diag:proofpullback}
    \begin{gathered}
        \begin{tikzcd}
        & \eta(\psi_{2},[p,q],ik)\arrow[dr] \arrow[dl] & \\\eta(\phi_{2},[\phi_{1}(p),\phi_{1}(q)],jk) \arrow[dr] & & \eta(\phi_{1},[p,q],ij) \arrow[dl]\\
        & \eta(f_{1}, [\phi_{1}(p),\phi_{1}(q)], j) &
        \end{tikzcd}
    \end{gathered}
    \end{equation}
    The datum of an $n$-simplex in $\mathsf{BicoMod}(\Span_{\Delta}(\c{C}^{\times}))$ corresponds via adjunction to a functor 
    \[
    \eta~:~\Tw(\Delta^{n})~\xrightarrow{}~\Fun(\Tw(\Delta_{/[1]})\times_{\Delta} \Delta^{\amalg},\c{C}).
    \]
     We claim that $\eta$  factors through the full subcategory
    \[
        \BMod_{Sp}(\c{C}) \subset \Fun(\Tw(\Delta_{/[1]})\times_{\Delta} \Delta^{\amalg},\c{C})), 
    \]
    as defined in Definition~\ref{def:bimodseg}. For the evaluation of $\eta$ on objects of the form $ \id_{k}$, this follows from condition $(1)$ and $(2)$. Hence, let $i<j$ be an object of $\Tw(\Delta^{n})$. We need to check the conditions of Definition~\ref{def:bimodseg}. All conditions except condition $(3)$ of Definition~\ref{def:bimodseg} are satisfied by assumption. For this condition to hold it suffices to check that for every inert morphism $\phi:f^{0}_{n_{0}} \xrightarrow{} f^{1}_{n_{1}} \in \Delta_{/[1]}$ and every subinterval $[k,l] \subset [n_{0}]$ the morphisms
    \begin{itemize}
        \item[(i)] $\eta(\phi,[k,l],i<j) \xrightarrow{} \eta(f^{1}_{n_{1}},[\phi(k),\phi(l)],i<j)$ 
        \item[(ii)] $\eta(\phi,[k,l],i<j) \xrightarrow{} \eta(f^{0}_{n_{0}},[k,l],i<j)$
    \end{itemize}
    are equivalences. To this end, consider the following $2$-simplices:
    \begin{itemize}
        \item[(i)] $(f^{0}_{n_{0}},i) \xrightarrow{\phi} (f^{1}_{n_{1}},i) \xrightarrow{i<j} (f^{1}_{n_{1}},j) $
        \item[(ii)] $(f^{0}_{n_{0}},i) \xrightarrow{i<j} (f^{0}_{n_{0}},j) \xrightarrow{\phi} (f^{1}_{n_{1}},j) $
    \end{itemize}
    in $\Delta_{/[1]} \times_{\Delta} \Delta^{n}$. Since $\eta(-,-,i)\in \mathsf{Bim}_{Sp}(\c{C})$ for all $i\in [n]$, the pullback Diagrams~\ref{diag:proofpullback} associated by $\eta$ to these $2$-simplices exhibit the morphisms in $(i)$ and $(ii)$ as pullbacks of equivalences. Hence, they are themselves equivalences.

    Consequently, we can apply Lemma~\ref{lem:aux}. It therefore, suffices to analyze condition $(I)-(IV)$ of Lemma~\ref{lem:aux}. Condition $(I)$ and $(II)$ imply that $\eta$ descends to a morphism 
    \[
        \eta:\Delta^{n} \xrightarrow{} \Span(\mathsf{BMod}_{Sp}(\c{C})).
    \]
    Under the equivalence $\sf{BMod}_{Sp}(\c{C})\simeq\operatorname{\Fun^{bim}}(\Delta^{\star}_{1},\c{C})$ of Proposition~\ref{prop:Lisanequiv} condition $(III)$ reads as
    \[
    \squarevert[{\eta_{\phi_{1}(p) \leq \phi_{1}(p+1)}(ij) \times ... \times \eta_{\phi_{1}(q-1) \leq \phi_{1}(q)}(ij)}]{\eta_{\phi_{1}(p) \leq \phi_{1}(p+1)}(ij) \times ... \times \eta_{\phi_{1}(q-1) \leq \phi_{1}(q)}(j)}[{\eta_{p,p+1}(ij) \times ... \times \eta_{q-1,q}(ij)}]{\eta_{p,p+1}(j) \times ... \times \eta_{q-1,q}(j)}
    \]
    where the maps are the image under $\eta$ of the diagram 
    \[
    \squarevert[{f^{1}_{\{\phi_{1}(p)\leq \phi_{1}(p+1)\}},...,f^{1}_{\{\phi_{1}(q-1)\leq \phi_{1}(q)\}},ij)}]{(f^{1}_{\{\phi_{1}(p)\leq \phi_{1}(p+1)\}},...,f^{1}_{\{\phi_{1}(q-1)\leq \phi_{1}(q)\}},j)}[{(f^{0}_{\{p,p+1\}},...,f^{0}_{\{q-1,q\}},ij)}]{(f^{0}_{\{p,p+1\}},...,f^{0}_{\{q-1,q\}},j)}
    \]
    in the category $\Delta^{\ast}_{1} \times \Tw(\Delta)$.
    Similarly, condition $(IV)$ is given by
    \[
    \squarevert[{\eta_{\phi_{1}(p) \leq \phi_{1}(p+1)}(ij) \times ... \times \eta_{\phi_{1}(q-1) \leq \phi_{1}(q)}(ij)}]{\eta_{\phi_{1}(p),\phi_{1}(p)+1}(ij)\times ... \times \eta_{\phi_{1}(q)-1,...,\phi_{1}(q)}(ij)}[{\eta_{\phi_{1}(p) \leq \phi_{1}(p+1)}(i) \times ... \times \eta_{\phi_{1}(q-1) \leq \phi_{1}(q)}(i)}]{\eta_{\phi_{1}(p),\phi_{1}(p)+1}(i)\times ... \times \eta_{\phi_{1}(q)-1,\phi_{1}(q)}(i)}
    \]
    Observe that the above diagrams are products of individual square-shaped diagrams. It, therefore, suffices to show that each such square is a pullback diagram. For condition $(III)$, we therefore need to check that for every $p\leq r \leq q-1$ the diagram
    \[
    \squarevert[{\eta_{\phi_{1}(r) \leq \phi_{1}(r+1)}(ij)}]{\eta_{\phi_{1}(r) \leq \phi_{1}(r+1)}(j)}[{\eta_{\{r,r+1\}}(ij)}]{\eta_{\{r,r+1\}}(j)}
    \]
    is a pullback square. Similarly, for condition $(IV)$, we need to check that the diagram:
   \[
    \squarevert[{\eta_{\phi_{1}(r)\leq \phi_{1}(r+1)}(ij)}]{\eta_{\{\phi_{1}(r),\phi_{1}(r)+1\}}(ij)\times... \eta_{\{\phi_{1}(r+1)-1,\phi_{1}(r+1)\}}(ij)}[{\eta_{\phi_{1}(r)\leq \phi_{1}(r+1)}(i)}]{\eta_{\{\phi_{1}(r),\phi_{1}(r)+1\}}(i) \times... \eta_{\{\phi_{1}(r+1)-1,\phi_{1}(r+1)\}}(i)}
    \]
    is a pullback diagram.
    Under the equivalence of Proposition~\ref{prop:bimisbiseg}, the map $\eta$ identifies with a functor $\eta:\Tw(\Delta^{n}) \xrightarrow{} \Fun(\Delta_{/[1]}^{\op},\c{C})$. We denote the value of $\eta$ on an object $i\leq j \in \Tw(\Delta^{n})$ by $M^{i,j}_{\bullet}$, if $i<j$ and by $X^{i}_{\bullet}$, if $i=j$. 

    Under this equivalence condition $(III)$ translates into the condition that for every $f_{l} \in \Delta^{\op}_{/[1]}$ and $0\leq i<j\leq n$ the diagram
    \[
    \squarevert[{M^{i,j}_{f_{l}}}]{X^{j}_{f_{l}}}[{M^{i,j}_{f_{\{0,l\}}}}]{ X^{j}_{f_{\{0,l\}}}}
    \]
    is a pullback diagram. Similarly, condition $(IV)$ translates into the condition that for every $f_{l} \in \Delta^{\op}_{/[1]}$ and $0\leq i<j\leq n$ the diagram:
    \[
    \squarevert[{M^{i,j}_{f_{l}}}]{X_{f_{l}}^{i}}[{M^{i,j}_{f_{\{0,1\}}} \times ... \times M^{i,j}_{f_{\{l-1,l\}}}}]{X_{f_{\{0,1\}}}^{i} \times ... \times X_{f_{\{l-1,l\}}}^{i}}
    \]
    is pullback diagrams. By Lemma~\ref{lemma:aux2} and Lemma~\ref{lem:aux3}, the conditions suffice to hold for $j=i+1$. But these are precisely the active equifibered and relative Segal conditions. This finishes the proof.
\end{proof}

\begin{lemma}\label{lemma:aux2}
    Let $\Tilde{\eta}:\Delta^{n} \xrightarrow{} \Span(\Fun(\Delta_{/[1]}^{\op},\c{C}))$ be a functor with adjoint $\eta:\Tw(\Delta^{n}) \xrightarrow{} \Fun(\Delta_{/[1]}^{\op},\c{C})$ and denote by $M^{i,j}_{\bullet}$ the value of $\eta$ on the object $i \rightarrow j \in \Tw(\Delta^{n})$ if $i<j$ and by $X^{i}_{\bullet}$, if $i=j$. Further, let $f_{l}: [l] \rightarrow [1]$ be an object of $\Delta_{/[1]}$. Then the following are equivalent:
    \begin{itemize}
    \item[(1)] the diagram
    \[
    \squarevert[{M^{i,j}_{  f_{l}}}]{X^{j}_{f_{l}}}[{M^{i,j}_{f_{\{0,l\}}}}]{ X^{j}_{f_{\{0,l\}}}}
    \]
    is a pullback diagram for all $0\leq i<j \leq n$.
    \item[(2)] the diagram
    \[
    \squarevert[{M^{i,i+1}_{  f_{l}}}]{X^{i+1}_{f_{l}}}[{M^{i,i+1}_{f_{\{0,l\}}}}]{ X^{i+1}_{f_{\{0,l\}}}}
    \]
    is a pullback diagram for all $0\leq i < n$.
    \end{itemize}
    
\end{lemma}
\begin{proof}
    (1) implies (2) by assumption. To prove the converse, we do induction on the difference $p=j-i$. The claim is true for $p=1$ by assumption. Assume that the result holds for $p-1$. Let $0\leq i<j\leq n$ be objects of $\Delta^{n}$ with $j-i=p$. We denote $j-1$ by $k$. Consider the diagram:
    \[
        \begin{xy}
            \xymatrix{X^{j}_{f} \ar[d] & M^{k,j}_{f} \ar[l] \ar[d] & M^{i,j}_{f} \ar[l] \ar[d] \\
            X_{f_{\{0,l\}}}^{j}  & M_{f_{\{0,l\}}}^{k,j} \ar[l] \ar[d] & M^{i,j}_{f_{\{0,l\}}} \ar[l] \ar[d]\\
            &  X_{f_{\{0,l\}}}^{k} &  M_{f_{\{0,l\}}}^{i,k}\ar[l] 
            }
        \end{xy}
    \]
    The upper left square is a pullback diagram by the induction hypothesis, and the lower square is a pullback diagram by the definition of $\eta$. To show that the big horizontal rectangle is a pullback diagram, it, therefore suffices to show that the big vertical rectangle is a pullback diagram. To show this, we factor the vertical rectangle as
    \[
        \begin{xy}
            \xymatrix{M_{f}^{k,j}  \ar[d] & M_{f}^{i,j}\ar[d] \ar[l] \\
            X_{f}^{k} \ar[d] & M_{f}^{i,k} \ar[l] \ar[d]\\
            X_{f_{\{0,l\}}}^{k} & M_{f_{\{0,l\}}}^{i,k} \ar[l] \\
            }
        \end{xy}
    \]
    The lower square is a pullback square by the induction hypothesis, and the upper square is a pullback square by the definition of $\eta$. It follows that the outer square is a pullback square. This finishes the proof.
\end{proof}    

    \begin{lemma}\label{lem:aux3}
    Let $\Tilde{\eta}:\Delta^{n} \xrightarrow{} \Span(\Fun(\Delta_{/[1]}^{\op},\c{C}))$ be a functor with adjoint $\eta:\Tw(\Delta^{n}) \xrightarrow{} \Fun(\Delta_{/[1]}^{\op},\c{C})$ and denote the value of $\eta$ on the object $i \rightarrow j \in \Tw(\Delta^{n})$  by $M^{i,j}_{\bullet}$  if $i<j$ and by $X^{i}_{\bullet}$ if $i=j$. Further, let $f_{l}: [l] \rightarrow [1]$ be an object of $\Delta_{/[1]}$. Then the following are equivalent:
    \begin{itemize}
    \item[(1)] the diagram
    \[
    \squarevert[{M^{i,j}_{f_{l}}}]{X_{f_{l}}^{i}}[{M^{i,j}_{f_{\{0,1\}}} \times ... \times M^{i,j}_{f_{\{l-1,l\}}}}]{X_{f_{\{0,1\}}}^{i} \times ... \times X_{f_{\{l-1,l\}}}^{i}}
    \]
    is a pullback diagram for all $0\leq i<j \leq n$.
    \item[(2)] the diagram
    \[
    \squarevert[{M^{i,i+1}_{f_{l}}}]{X_{f_{l}}^{i}}[{M^{i,i+1}_{f_{\{0,1\}}} \times ... \times M^{i,i+1}_{f_{\{l-1,l\}}}}]{X_{f_{\{0,1\}}}^{i} \times ... \times X_{f_{\{l-1,l\}}}^{i}}
    \]
    is a pullback diagram for all $0\leq i < n$.
    \end{itemize}
    
\end{lemma}
\begin{proof}
    The proof is analogous to the one of Lemma~\ref{lemma:aux2}. By assumption, $(1)$ implies $(2)$. For the converse, we do induction on the difference $p=j-i$. The case $p=1$ follows by assumption. Assume therefore the result holds for $p-1$. Let $0\leq i < j\leq n$ be objects of $\Delta^{n}$ and denote $j-1$ by $k$. Consider the diagram:
    \[
        \begin{xy}
            \xymatrix{\prod_{m=1}^{l}M^{k,j}_{f_{\{m-1,m\}}} \ar[d] & \prod_{m=1}^{l}M^{i,j}_{f_{\{m-1,m\}}} \ar[l] \ar[d] & M^{i,j}_{f} \ar[l] \ar[d] \\
            \prod_{m=1}^{l}X^{k}_{f_{\{m-1,m\}}} & \prod_{m=1}^{l}M^{i,k}_{f_{\{m-1,m\}}} \ar[l] \ar[d] & M^{i,k}_{f} \ar[l] \ar[d] \\
            & \prod_{m=1}^{l}X^{i}_{f_{\{m-1,m\}}} & X_{f}^{i} \ar[l] \\  
            }
        \end{xy}
    \]
    It follows from the inductive hypothesis that the bottom square is 
    a pullback diagram. Further, the upper left square is a pullback by definition of $\eta$. Therefore, to show that the big vertical rectangle is a pullback diagram, it suffices to show that the big horizontal rectangle is a pullback. The horizontal rectangle factors as
    \[
        \begin{xy}
            \xymatrix{\prod_{m=1}^{l}M^{k,j}_{f_{\{m-1,m\}}} \ar[d] & M^{k,j}_{f} \ar[l] \ar[d] & M^{i,j}_{f} \ar[d] \ar[l] \\
            \prod_{m=1}^{l}X^{k}_{f_{\{m-1,m\}}} & X_{f}^{k} \ar[l]& M^{i,k}_{f} \ar[l] \\
            }
        \end{xy}
        \]
    The left square is a pullback square by the induction hypothesis and the right square is a pullback square by definition of $\eta$. Consequently, also the big horizontal square is a pullback diagram.
\end{proof}

\appendix
\section{Auxiliary Statements}\label{sect:Appendix}
The goal of this section is to present the technical proofs of Proposition~\ref{cor:condit} and \ref{prop:localization}. We have extracted these technical proofs from Section~\ref{sect:bimod} to maintain the flow of the argument presented there.

\subsection{Proof of Proposition~\ref{cor:condit}}\label{subsect:appendix1}

The goal of this subsection is to prove Proposition~\ref{cor:condit}. As a start, we unravel the datum of a bicomodule in $\mathsf{Span}_{\Delta}(\c{C}^{\times})$.  Such an object is given by a commutative triangle:
\[
\begin{tikzcd}
    \Delta_{/[1]} \arrow[rr, "F"] \arrow[dr] & & \Span_{\Delta}(\c{C}^{\times}) \arrow[dl] \\
    & \Delta & 
\end{tikzcd}
\]
such that $F$ preserves inert morphisms and the adjoint morphism $\Tilde{F}: \Tw(\Delta_{/[1]}) \xrightarrow{} \c{C}^{\times} $ maps every $n$-simplex $\Delta^{n} \xrightarrow{} \Delta_{/[1]}$ to a Segal simplex.\footnote{See Construction~\ref{constr:span}}

We analyze the second condition first. By the universal property of the Cartesian monoidal structure, $\Tilde{F}$ is equivalent to a functor
\[
    \bar{F}:\Tw(\Delta_{/[1]})\times_{\Delta}\Delta^{\amalg} \xrightarrow{} \c{C}.
\]
Similarly, an $n$-morphism $\eta:\Delta^{n} \rightarrow \Fun_{\Delta}(\Delta_{/[1]},\Span_{\Delta}(\c{C}^\times))$ is equivalent to a functor
\[
    \bar{\eta}:\Tw(\Delta_{/[1]} \times \Delta^{n})\times_{\Delta}\Delta^{\amalg} \simeq  \Tw(\Delta^{n}) \times (\Tw(\Delta_{/[1]}) \times_{\Delta}\Delta^{\amalg}) \xrightarrow{} \c{C}.
\]
We put $\Theta_{n+1}:= \Tw(\Delta_{/[1]} \times \Delta^{n})\times_{\Delta}\Delta^{\amalg} \simeq \Tw(\Delta^{n})\times \Theta_{1}$. The objects of this category are given by
\[
    [i,j]\subset (g_{p}:[p]\xrightarrow{} [1],k) \xrightarrow{f} (g_{q}:[q]\xrightarrow{} [1],l),
\]
where $0 \leq k\leq l\leq n$, $f:g_{p} \xrightarrow{} g_{q}$ is a morphism in $\Delta_{/[n]}$ and $[i,j]$ is a subinterval of $[p]$. 
\begin{definition}\label{def:Segalcone}
    Let $\Bar{\eta}:\Tw(\Delta^{n}\times \Delta_{/[1]})\times_{\Delta}\Delta^{\amalg} \xrightarrow{}\c{C}$ be a morphism of simplicial sets, $\sigma:\Delta^{k} \xrightarrow{} \Delta^{n}\times \Delta_{/[1]}$ a $k$-simplex
     \[
        \begin{xy}
            \xymatrix{(g_{n_{0}},i_{0}) \ar[r]^-{\phi_{1}}& ... \ar[r]^-{\phi_{k}} & (g_{n_{k}},i_{k})  }
           
        \end{xy},
     \]
    and $[i,j]\subset [n_{0}]$ a subinterval. The \emph{Segal cone} of $\Bar{\eta}$ associated to $\sigma$ and $[i,j]$ is the diagram
    \begin{align*}\label{diag:Segalcone}
    \centering
\begin{tikzpicture}[scale=0.85]
\node[scale=0.9] (1) at (0.25,2)  {$\Bar{\eta}(\psi_{n},[i,j],i_{0}<i_{k})$};
\node[scale=0.9] (2) at (-4,1) {$\Bar{\eta}(\phi_{1},[i,j],i_{0}<i_{1})$};
\node[scale=0.9] (3) at (5,1) {$\Bar{\eta}(\phi_{k},[\psi_{k-1}(i),\psi_{k-1}(j)],i_{k-1}<i_{k})$};
\node[scale=0.9] (5) at (3,0) {$\Bar{\eta}(g_{n_{k-1}},[\psi_{k-1}(i),\psi_{k-1}(j)],i_{k-1})$} ;
\node[scale=0.9] (6) at (8,0) {$\Bar{\eta}(g_{n_{k}},[\psi_{k}(i),\psi_{k}(j)],i_{k})$} ;
\node[scale=0.9] (7) at (-2,0) {$\Bar{\eta}(g_{n_{1}},[\psi(i),\psi(j)],i_{1})$} ;
\node[scale=0.9] (8) at (-6,0)  {$\Bar{\eta}(g_{n_{0}},[i,j],i_{0})$};
\node[scale=0.9] (9) at (0.25,1) {$\cdots$};  
\draw[->] (1) to (2) ;
\draw[->] (1) to (3);
\draw[->] (3) to (5);
\draw[->] (3) to (6);
\draw[->] (2) to (7);
\draw[->] (2) to (8);
\draw[draw=none] (-.5*\textwidth,0) -- (.5*\textwidth,0);
\end{tikzpicture}
\end{align*}
    where we put $\psi_{i} := \phi_{i} \circ \phi_{i-1} \circ ... \circ \phi_{1}$.
\end{definition}

\begin{prop}\label{prop:Segalcone}
    Let $\Bar{\eta}: \Tw(\Delta^{n}\times \Delta_{/[1]})\times_{\Delta}\Delta^{\amalg} \xrightarrow{}\c{C}$ be a morphism of simplicial sets. Its adjoint $\eta$ defines an $n$-morphism in $\Fun_{\Delta}(\Delta_{/[1]},\Span_{\Delta}(\c{C}^{\times}))$ if and only if for every $k$-simplex
    \[
        \begin{xy}
            \xymatrix{([n_{0}],i_{0})\ar[dr] \ar[r]^-{\phi_{1}}& ... \ar[r]^-{\phi_{k}} & ([n_{k}],i_{k}) \ar[dl] \\
            &  [1] & }
        \end{xy}
     \]
    in $\Delta_{/[1]}\times \Delta^{n}$ and every interval $[i,j] \subset [n_{0}]$ the associated Segal cone of $\Bar{\eta}$ is a limit diagram in $\c{C}$.
\end{prop}
\begin{proof}
    By definition of $\Span_{\Delta}(\cC^{\times})$, we need to show that for every $k$-simplex $\Delta^{k} \xrightarrow{} \Delta^{n} \times \Delta_{/[1]}$  given by
    \[
        (g_{n_{0}},i_{0}) \xrightarrow{} (g_{n_{1}},i_{1}) \xrightarrow{} ... \xrightarrow{} (g_{n_{k}},i_{k}),y
    \]
    the restriction of $\Bar{\eta}$ to $\Tw(\Delta^{k}) \subset \Tw(\Delta_{/[1]}\times \Delta^{n})$ is a Segal simplex in $\c{C}^{\times}$.
    According to \cite[Lem.10.2.13]{dyckerhoff2019higher} there exists a functor
    \[
        \Tilde{H}_{t}:(\Delta^{1}\times \Tw(\Delta^{k})) \times_{\Delta} \Delta^{\amalg} \xrightarrow{} \c{C}
    \]
    adjoint to a homotopy
    \[
        H_{t}:(\Delta^{1}\times \Tw(\Delta^{k})) \xrightarrow{} \c{C}^{\times}
    \]
    such that $H_{1}= \Bar{\eta}$, the components of this homotopy are Cartesian morphisms and $H_{0}$ has image contained in $\c{C}^{\times}_{g_{0}}$. Consequently, the restriction of $\Bar{\eta}$ to the Segal cone\footnote{See Construction~\ref{constr:span}} is a relative limit in $\c{C}^{\times}$ if and only if the restriction of $H_{0}$ to the Segal cone is a limit in $\c{C}^{\times}_{g_{0}}$ \cite[Prop. 4.3.1.9]{lurie2006higher}. 

    We can check the latter for each interval $[i,j] \subset [n_{0}]$ individually. The restriction of $\Tilde{H}_{0}$ to the Segal cone is represented in $\c{C}$ by the diagram
\[
\begin{tikzpicture}[scale=0.85]
\node[scale=0.9] (1) at (0.25,2)  {$\Tilde{H}_{0}(\psi_{n},[i,j],i_{0}<i_{n})$};
\node[scale=0.9] (2) at (-4,1) {$\Tilde{H}_{0}(\phi_{1},[i,j],i_{0}<i_{1})$};
\node[scale=0.9] (3) at (5,1) {$\Tilde{H}_{0}(\phi_{k}, [i,j],i_{k-1}<i_{k})$};
\node[scale=0.9] (5) at (3.25,0) {$\Tilde{H}_{0}(g_{n_{k-1}},[\psi_{k-1}(i),\psi_{k-1}(j)],i_{k-1})  $} ;
\node[scale=0.9] (6) at (8,0) {$\Tilde{H}_{0}(g_{n_{k}},[\psi_{k}(i),\psi_{k}(j)], i_{k})$} ;
\node[scale=0.9] (7) at (-2,0) {$\Tilde{H}_{0}(g_{n_{1}},[\psi_{1}(i),\psi_{1}(j)],i_{1})$} ;
\node[scale=0.9] (8) at (-6,0)  {$\Tilde{H}_{0}(g_{n_{0}},[i,j],i_{0})$};
\node[scale=0.9] (9) at (0.25,1) {$\cdots$};  
\draw[->] (1) to (2) ;
\draw[->] (1) to (3);
\draw[->] (3) to (5);
\draw[->] (3) to (6);
\draw[->] (2) to (7);
\draw[->] (2) to (8);
\end{tikzpicture}
\]
    Since the components of the homotopy are Cartesian, it restricts to an equivalence between this diagram and the Segal cone diagram of $\Bar{\eta}$ associated to $\sigma$ and $[i,j]$. Therefore the simplex is a Segal simplex, if for all simplices $\sigma$ and all subintervals $[i,j]$, the Segal cone diagram of $\Bar{\eta}$ associated to $\sigma$ and $[i,j]$ are limit diagrams.
\end{proof}
\noindent In fact, it turns out to be sufficient to check this condition for $2$-simplices:
\begin{lemma}\label{lem:segalandlimits}
    Let $\c{C}$ be an $\infty$-category with finite limits. For an $n$-simplex 
    \[
    F:\Delta^{n} \xrightarrow{} \Fun_{\Delta}(\Delta_{/[1]},\overline{\Span_{\Delta}}(\c{C}^{\times}))     \]
    the following are equivalent:
    \begin{itemize}
        \item[(1)] $F$ restricts to a functor $F:\Delta^{n} \xrightarrow{} \Fun_{\Delta}(\Delta_{/[1]},\Span_{\Delta}(\c{C}^{\times}))$.
        \item[(2)] for every non-degenerate $2$-simplex $\Delta^{2}\xrightarrow{} \Delta_{/[1]} \times \Delta^{n}$ the restriction $F|_{\Delta^{2}}$ is a Segal simplex.
    \end{itemize}
\end{lemma}
\begin{proof}
    This follows from the iterated application of \cite[Prop.4.2.3.8]{lurie2006higher}.
\end{proof}
\noindent To study the first condition, recall that a morphism $\phi:\Delta^{1} \xrightarrow{} \Delta_{/[1]}$ is called inert if its image in $\Delta$ is inert. If we depict a morphism in $\Delta_{/[1]}$ by a commutative diagram
\[
\begin{tikzcd}
    {[n_{0}]} \arrow[rr,"f"] \arrow[dr] & & {[n_{1}]} \arrow[dl] \\
    & {[1]} &
\end{tikzcd}
\]
then the morphism is inert when $f$ is the inclusion of a subinterval.
Let $\phi:\Delta^{1} \xrightarrow{} \Delta_{/[1]}$ be an inert morphism and $F$ a bicomodule object in $\Span_{\Delta}(\c{C}^{\times})$. By definition $F:\Delta_{/[1]} \xrightarrow{} \Span_{\Delta}(\c{C}^{\times})$ has to map inert morphisms to Cartesian morphisms. Unraveling the definition of the Cartesian fibration $\Span_{\Delta}(\c{C}^{\times})$, it follows that the image of $\phi$ is Cartesian, if and only if the adjoint map $\widetilde{F\circ\phi}:\Tw(\Delta^{1})\xrightarrow{} \c{C}^{\times}$ maps all morphisms in $\Tw(\Delta^{1})$ to Cartesian morphisms. In other words for any morphism $\phi:g^{0}_{n_{0}} \xrightarrow{} g^{1}_{n_{1}}$ and any interval $[i,j] \subset [n_{0}]$ 
\begin{itemize}
    \item[(1)] the morphism $F(\phi,\{i,j\})\xrightarrow{} F(g^{0}_{n_{0}}, \{i,j\})$ induced by the source map $\phi \xrightarrow{} g_{n_{0}}^{0}$ in $\Tw(\Delta_{/[1]})$ is an equivalence.
    \item[(2)] the morphism $F(\phi,\{i,j\})\xrightarrow{} F(g^{1}_{n_{1}}, \{\phi(i),\phi(j)\})$ induced by the target map $\phi \xrightarrow{} g_{n_{1}}^{1}$ in $\Tw(\Delta_{/[1]})$ is an equivalence.
\end{itemize}

\noindent We derive some consequences of this:

\begin{lemma}
    Suppose $F:\Delta_{/[1]}\xrightarrow{} \Span_{\Delta}(\c{C}^{\times})$ represents a bicomodule. Let $f:g^{0}_{n_{0}}\xrightarrow{} g^{1}_{n_{1}}$ be a morphism in $\Delta_{/[1]}$ viewed as an object in $\Tw(\Delta_{/[1]})$
    \begin{itemize}
        \item[$(1)$] Denote for $[i,j]\subset [n_{0}]$ by $f|_{[i,j]}$ the restriction of $f$ to $g^{0}_{[i,j]}\subset g^{0}_{n_{0}}$. Then the induced morphism
        \[
            F(f|_{[i,j]},[i,j]) \xrightarrow{} F(f,[i,j])
        \]
        is an equivalence.
        \item[$(2)$] Let $\Tilde{f}:g_{n_{0}}^{0}\xrightarrow{}g_{[l,k]}^{1}$ be a morphism with $[l,...,k]\subset [n_{1}]$, such that the composite with the inert morphism $\phi:g_{[l,k]}^{1} \hookrightarrow g_{n_{1}}^{1}$ is $f$. Then the induced morphism
        \[
            F(f,[l,k])\xrightarrow{} F(\Tilde{f},[l,k])
        \]
        is an equivalence.
    \end{itemize}
\end{lemma}
\begin{proof}
    The proof is analogous to \cite[Prop.2.2]{stern20212}
\end{proof}

\begin{lemma}
    Suppose $F:\Delta_{/[1]}\xrightarrow{} \Span_{\Delta}(\c{C}^{\times})$ represents a bicomodule. Let $\sigma$ be a morphism in the $\infty$-category $\Tw(\Delta_{/[1]})\times_{\Delta}\Delta^{\amalg}$ of the form:
        \[
            \sigma =
                \begin{Bmatrix}
                    \begin{xy}
                \xymatrix{[i,j] \subset g_{n_{0}}^{0} \ar[r]^-{g}\ar[d]_-{f}&  g_{n_{1}}^{1} \\
                [\Tilde{i},\Tilde{j}] \subset g_{m_{0}}^{0} \ar[r]_-{\Tilde{g}} & g_{m_{1}}^{1} \ar[u]_-{\Tilde{f}} \\
                }
            \end{xy}
            \end{Bmatrix}
        \]
        such that f restricts to an isomorphism $[i,...,j] \xrightarrow{} [\Tilde{i},...,\Tilde{j}]$ and $\Tilde{f}$ restricts to an isomorphism 
        \[
        [g(i),...,g(j)] \xrightarrow{} [\Tilde{g}(\Tilde{i}),...,\Tilde{g}(\Tilde{j})].
        \]
        Then $F$ sends $\sigma$ to an equivalence.
\end{lemma}
\begin{proof}
    We decompose $\sigma$ into the diagram
    \[
    \begin{tikzpicture}[auto]
        \node (A) at (0,0) {$g^{0}_{[i,j]}$};
        \node (B) at (0,-1.5) {$g^{0}_{n_{0}}$};
        \node (C) at (0,-3) {$g^{0}_{m_{0}}$};
        \node (D) at (1.5,0) {$g^{1}_{n_{1}}$};
        \node (E) at (1.5,-1.5) {$g^{1}_{m_{1}}$};
        \node (F) at (1.5,-3) {$g^{1}_{m_{1}}$};

        \draw[->](A) to node[scale=0.7,swap]{$\phi$} (B);
        \draw[->](B) to node[scale=0.7,swap]{$f$} (C);
        \draw[->](E) to node[scale=0.7,swap]{$\id$} (D);
        \draw[->](F) to node[scale=0.7,swap]{$\bar{f}$} (E);
        \draw[->](A) to node[scale=0.7]{$g_{[i,j]}$} (D);
        \draw[->](B) to node[scale=0.7]{$g$} (E);
        \draw[->](C) to node[scale=0.7,swap]{$\bar{g}$} (F);  
    \end{tikzpicture}
    \]
    Note that the upper square gets mapped by $F$ to an equivalence and that the composite of the left vertical maps is inert. So by $2$-out-of-$3$ we can restrict to those $\sigma$, such that the morphism $f$ is inert. Such a diagram $\sigma$ can be further decomposed as:
    \[
    \begin{tikzpicture}[auto]
        \node (A) at (0,0) {$g^{0}_{[i,j]}$};
        \node (B) at (0,-1.5) {$g^{0}_{[i,j]}$};
        \node (C) at (0,-3) {$g^{0}_{m_{0}}$};
        \node (D) at (1.5,0) {$g^{1}_{n_{1}}$};
        \node (E) at (1.5,-1.5) {$g^{1}_{m_{1}}$};
        \node (F) at (1.5,-3) {$g^{1}_{m_{1}}$};

        \draw[->](A) to node[scale=0.7,swap]{$\id$} (B);
        \draw[->](B) to node[scale=0.7,swap]{$f$} (C);
        \draw[->](E) to node[scale=0.7,swap]{$\bar{f}$} (D);
        \draw[->](F) to node[scale=0.7,swap]{$\id$} (E);
        \draw[->](A) to node[scale=0.7]{$g_{[i,j]}$} (D);
        \draw[->](B) to node[scale=0.7]{$\bar{g}\circ f$} (E);
        \draw[->](C) to node[scale=0.7,swap]{$\bar{g}$} (F);
    \end{tikzpicture}
    \]
    Using the same argument as above applied to the lower square, we can reduce to the case that $f= \id$. So let $\sigma$ be a diagram of the form

\[
    \begin{tikzpicture}[auto]
        \node (A) at (0,0) {$g^{0}_{n_{0}}$};
        \node (B) at (0,-1.5) {$g^{0}_{n_{0}}$};
        \node (D) at (1.5,0) {$g^{1}_{n_{1}}$};
        \node (E) at (1.5,-1.5) {$g^{1}_{m_{1}}$};
    
        \draw[->](A) to node[scale=0.7,swap]{$\id$} (B);
        \draw[->](E) to node[scale=0.7,swap]{$\bar{f}$} (D);
        \draw[->](A) to node[scale=0.7,]{$g_{[i,j]}$} (D);
        \draw[->](B) to node[scale=0.7,swap]{$\bar{g}$} (E);
    \end{tikzpicture}
    \]
    such that $\bar{f}$ induces an isomorphism $[\bar{g}(0),\bar{g}(n_{0})] \xrightarrow{} [g(0),g(n_{0})]$. In particular, there exists a decomposition
    \[
    \begin{tikzpicture}[auto]
        \node (A) at (0,0) {$g^{0}_{n_{0}}$};
        \node (B) at (0,-1.5) {$g^{0}_{n_{0}}$};
        \node (C) at (0,-3) {$g^{0}_{n_{0}}$};
        \node (D) at (1.5,0) {$g^{1}_{n_{1}}$};
        \node (E) at (1.5,-1.5) {$g^{1}_{m_{1}}$};
        \node (F) at (1.5,-3) {$g^{1}_{[\bar{g}(0),\bar{g}(n_{0})]}$};

        \draw[->](A) to node[scale=0.7,swap]{$\id$} (B);
        \draw[->](B) to node[scale=0.7,swap]{$\id$} (C);
        \draw[->](E) to node[scale=0.7,swap]{$\bar{f}$} (D);
        \draw[->](F) to node[scale=0.7,swap]{$\psi$} (E);
        \draw[->](A) to node[scale=0.7]{$g$} (D);
        \draw[->](B) to node[scale=0.7]{$\bar{g}$} (E);
        \draw[->](C) to node[scale=0.7,swap]{$\bar{g}$} (F);
    \end{tikzpicture}
    \]
    with $\psi$ inert. It follows by assumption that $f\circ \psi$ is itself inert. Hence the claim follows from $2$-out-of-$3$ again.
\end{proof}

\noindent Putting these results together, we have proven Proposition~\ref{cor:condit}:
\begin{prop}
    A functor $F:\Theta_{1} \xrightarrow{} \c{C}$ defines a bicomodule object if and only if
    \begin{itemize}
        \item[(1)] $F$ sends degenerate intervals to terminal objects. 
        \item[(2)] $F$ sends every object $(\phi:g_{[n_{0}]}^{0} \xrightarrow{} g_{[n_{1}]}^{1}, [i,j])$ together with its projection to subintervals to a product diagram.\footnote{See Construction~\ref{constr:cart}}
        
        \item[(3)] F sends morphisms of the form
        \[
            \sigma \simeq 
                \begin{Bmatrix}
                    \begin{xy}
                \xymatrix{[i,j] \subset g_{n_{0}}^{0} \ar[r]^-{g}\ar[d]^-{f}&  g_{n_{1}}^{1} \\
                [\Tilde{i},\Tilde{j}] \subset g_{m_{0}}^{0} \ar[r]^-{\Tilde{g}} & g_{m_{1}}^{1} \ar[u]^-{\Tilde{f}} \\
                }
            \end{xy}
            \end{Bmatrix}
        \]
        s.t. the morphism f restricts to an isomorphism $\{i,...,j\} \xrightarrow{} \{\Tilde{i},...,\Tilde{j}\}$ and the morphism $\Tilde{f}$ restricts to an isomorphism $\{g(i),...,g(j)\} \xrightarrow{} \{\Tilde{g}(\Tilde{i}),...,\Tilde{g}(\Tilde{j})\}$ to equivalences. We denote by $E$ the wide subcategory of $\Theta_{1}$ with morphisms of the above form.
        \item[(4)] $F$ maps all Segal cone diagrams from Definition~\ref{def:Segalcone} to limit diagrams.
    \end{itemize}
\end{prop}

\subsection{Proof of Proposition~\ref{prop:localization}}\label{subsect:appendix2}
The goal of this subsection is to prove the following:
\begin{prop}
    The functor $\mathcal{L}:\Omega_{1}\xrightarrow{} \Delta^{\star}_{1}$ from Construction~\ref{constr:L} is an $\infty$-categorical localization at the morphisms $E$, defined in Corollary~\ref{cor:condit} (3).
\end{prop}
\noindent To prove this, we show that this functor satisfies the assumptions of \cite[Lem.3.1.1]{walde20212}. To this end, we need to consider for any object $N=(e_{m_{0}}^{0},...,e_{m_{k-1}}^{k-1})\in \Delta^{\star}_{1}$  the strict fiber of the functor $\mathcal{L}$ at $N$. We denote the strict fiber by $\Omega_{1,N}$ and we denote by $\Omega_{1,N}^{E}$ the subcategory of $\Omega_{1,N}$ with morphisms in $E$. To apply \cite[Lem.3.1.1]{walde20212}, we first construct an explicit initial object in $\Omega_{1,N}^{E}$:
\begin{construction}
    For any $N=(e_{m_{0}}^{0},...,e_{m_{k-1}}^{k-1})\in \Delta^{\star}_{1}$ as above we define the object $I_{N}\in \Omega_{1,N}$ as
\[
\triangle[{[0,k] \subset[k]}]{[m_{0}] \ast ... \ast [m_{k-1}]}[{[1]}]{f_{N}}[{e_{k}}]{e_{m_{0}}^{0} \ast ... \ast e_{m_{k-1}}^{k-1}}
\]
    where the morphism $f_{N}$ is defined by
    \[
       f_{N}(i)= \begin{cases}
           0 \in [m_{i}] \quad i\leq k \\
           m_{k-1}\in [m_{k-1}] \quad i=k
       \end{cases} ,
    \]
    and $e_{k}$ is the unique morphism that makes the triangle commute. If $e_{m_{0}}^{0} \ast ... \ast e_{m_{k-1}}^{k-1}$ is supported at $\{0\}$, we denote by $I_{N} \ast \id_{[1]}$ the object
    \[
    \begin{tikzcd}
        {[0,k] \subset[k]} \arrow[dr] \arrow[rr, "f_{N}"] & & {[m_{0}] \ast ... \ast [m_{k-1}]\ast [1]} \arrow[dl, "e_{m_{0}}^{0} \ast ... \ast e_{m_{k-1}}^{k-1}\ast \id_{[1]}"] \\
        & {[1]} &
    \end{tikzcd}
\]
Similarly, we denote by $\id_{[1]}\ast I_{N}$ the analogously defined object if $e_{m_{0}}^{0} \ast ... \ast e_{m_{k-1}}^{k-1}$ is supported at $\{1\}$.
\end{construction}

\begin{lemma}\label{lem:initial}
    For every $N$ the category $\Omega_{1,N}^{E}$ has an initial object given by
    \begin{itemize}
        \item[(1)] $I_{N}\ast \id_{[1]}$ if $e_{m_{0}}^{0} \ast ... \ast e_{m_{k-1}}^{k-1}$ is supported at $0$.
        \item[(2)] $I_{N}$ if $e_{m_{0}}^{0} \ast ... \ast e_{m_{k-1}}^{k-1}$ is surjective.
        \item[(3)] $\id_{[1]}\ast I_{N}$ if $e_{m_{0}}^{0} \ast ... \ast e_{m_{k-1}}^{k-1}$ is supported at $1$.
    \end{itemize}
\end{lemma}
\begin{proof}
     We will prove that the respective object is initial. Let
    \[
            [i,j] \subset [n] \xrightarrow{\alpha} [l] \xrightarrow{h} [1]
    \]
    be an object in $\Omega_{1,N}^{E}$. We will construct a unique morphism $\Phi$ in $\Omega_{1,N}^{E}$ from the claimed initial object. In any of the above cases, we define the morphism $\phi:[k]\xrightarrow{}[n]$ as the inert map that includes $[k]$ as the interval $\{i,...,j\}$. This is uniquely defined since every morphism in $E$ has to induce an isomorphism between the chosen subintervals. We can decompose\footnote{See Construction~\ref{constr:decomp}}  the object $[l]\xrightarrow{h} [1]$ into 
    \[
       [l_{left}]\ast \{\alpha(i),...,\alpha(j)\} \ast [l_{right}] \xrightarrow{h_{left}\ast h \ast h_{right}} [1].   
    \]
    Since any morphism in $E$ has to map $\{\alpha(i),...,\alpha(j)\} $ isomorphically onto $[f_{N}(0),f_{N}(k)] = [n]$, the restriction of $\Phi$ to $\{\alpha(i),...,\alpha(j)\}$ is uniquely determined. We show that our assumptions allow us to extend this arrow uniquely to the outer part. In case $(2)$, we can uniquely extend. We only discuss case $(1)$, since $(3)$ is analogous. In case $(1)$, we can uniquely extend $\Phi$ to $[l_{left}]$ by setting it constantly $0$. An extension to $[l_{right}]$ is equivalent to a morphism to $\id_{[1]}$ in $\Delta_{/[1]}$. This morphism exists uniquely since $\id_{[1]}$ is a final object of $\Delta_{/[1]}$.
\end{proof}
\noindent We furthermore need to show that the inclusion $\Omega^{E}_{1,N}\subset \Omega^{1}_{/N}$ of the strict into the lax fiber is cofinal. To show this, we make the following preliminary considerations. By Quillen's theorem A \cite[Thm.4.1.3.1]{lurie2006higher} this amounts to showing, that for every object 
\[
            \{i,j\} \subset [n] \xrightarrow{\alpha} [l] \xrightarrow{h} [1],
\]
whose image under $\mathcal{L}$ is given by $(h_{l_{i+1}},...,h_{l_{j}})$, and every morphism
\[
    g:(h_{l_{i+1}},...,h_{l_{j}}) \xrightarrow{} (e_{m_{0}},...,e_{m_{k-1}}) 
\]
in $\Delta^{\star}_{1}$, the category $\Omega_{N,1}^{E} \times_{(\Omega_{1})_{/N}} ((\Omega_{1})_{/N})_{g/}$ is contractible. We do so by constructing an explicit initial object. By definition, the morphism $g$ amounts to a pair consisting of a morphism $\gamma:[k-1] \xrightarrow{} \{i+1,...,j\}$ in $\Delta$ and a morphism 
\[
\Bar{g}:[m_{0}]\ast... \ast [m_{k-1}] \xrightarrow{} [l_{i+1}]\ast ...\ast[l_{j}] \text{ in } \Delta_{/[1]}.
\] 
We denote by $[n_{c}]:= \{p,...,q\}\subset \{i,j\} \subset [n]$ the unique linearly ordered set such that $\gamma$ factors as
\[
   \gamma: [k-1] \xrightarrow{} [n_{c}] \hookrightarrow [n_{l}] \ast [n_{c}] \ast [n_{r}]. 
\]
By construction $\Bar{g}$ has image contained in $[l_{c}] := [l_{p+1}]\ast... \ast [l_{q}]$. We introduce the following decompositions:
\begin{align*}
    [l] &= [l_{l}] \ast [l_{c}] \ast [l_{r}], \\
    [m] &= [m_{l}] \ast [m_{c}] \ast [m_{r}].
\end{align*}

\begin{lemma}\label{lem:contr}
    There exists a morphism in $\Omega_{1}$
    \[
        \begin{xy}
            \xymatrix{[p,q] \subset d_{[p,...,q]} \ar[r] \ar[d] & h_{l_{c}} \\
            [0,k] \subset h_{\{0,l^{1}\}} \ast e_{k} \ast h_{\{0,l^{2}\}} \ar[r] &  h_{l^{1}} \ast e_{m} \ast h_{l^{2}} \ar[u]
            }
        \end{xy}
    \]
    that extends to a morphism
    \[
        \mu_{Z,N} := \begin{Bmatrix}
          \begin{xy}
            \xymatrix{Z = [i,j] \subset d_{n_{l}} \ast d_{n_{c}} \ast d_{n_{r}} \ar[r] \ar[d] & h_{l_{l}} \ast h_{l_{c}} \ast h_{l_{r}} \\
            Z_{N} = [0,k] \subset d_{n_{l}} \ast h_{\{0,l^{1}\}} \ast e_{k} \ast h_{\{0,l^{2}\}} \ast d_{n_{r}} \ar[r] & h_{l_{l}}\ast h_{l^{1}} \ast e_{m} \ast h_{l^{2}} \ast h_{l_{r}} \ar[u]
            }
        \end{xy}
         \end{Bmatrix}
     \]
    covering $g$, such that $\mu_{Z,N}$ is an initial object in $\Omega_{N,1}^{E} \times_{(\Omega_{1})_{/N}} ((\Omega_{1})_{/N})_{g/}$.
\end{lemma}
\begin{proof}
    We define a morphism $\nu:\{p+1,...,q-1\} \xrightarrow{} [k]$ as $\mathbb{I}(\gamma)$ and extend it to a morphism
    \[
    \nu: \{p,...,q\} \xrightarrow{} [1]\ast[k] \ast [1]
    \]
    by mapping endpoints to endpoints. We decompose $[l_{c}] = [l^{1}]\ast [l_{c}^{m}] \ast [l^{2}]$, where we denote by $[l^{m}_{c}]$ the subinterval of $[l]$, that contains the image of $\Bar{g}$. By construction, the map $\Bar{g}:e_{m}\xrightarrow{} h_{[l_{c}^{m}]}$ hits both endpoints. Therefore, we can extend it to a morphism
    \[
        \Bar{g}':= \id_{h_{l^{1}}} \ast g \ast id_{h_{l^{2}}}: h_{l^{1}} \ast e_{m} \ast h_{l^{2}} \xrightarrow{} h_{l^{c}}.
    \]
    By construction $\Bar{g}'$ also hits both endpoints and extends uniquely to a morphism over $[1]$. Similarly, we define
    \[
        \Bar{f}_{N}:h_{\{0,l^{1}\}} \ast e_{k} \ast h_{\{0,l^{2}\}} \xrightarrow{} h_{l^{1}} \ast e_{m} \ast h_{l^{2}}
    \]
    to be $f_{N}$ on $e_{k}$ and to send endpoints to endpoints. This morphism also extends uniquely to a morphism over $[1]$.  By decomposing the morphisms $\nu$, $h$ and $\Bar{f}_{N} \circ \nu$, we obtain a not necessarily commuting diagram
    \[
        \begin{xy}
            \xymatrix{[p,q] \subset d_{[1_{p+1}]}\ast ... \ast d_{[1_{p}]} \ar[r]^-{\alpha} \ar[d]^-{\nu} & h_{l_{p+1}}\ast... \ast h_{l_{q}} \\
            [0,k] \subset h_{\{0,l^{1}\}} \ast e_{k_{p+1}} \ast... \ast e_{k_{q}}\ast h_{\{0,l^{2}\}} \ar[r]_-{\Bar{f}_{N}} & h_{l^{1}} \ast e_{m_{p+1}} \ast ... \ast e_{m_{q}} \ast h_{l^{2}} \ar[u]^-{\Bar{g}'}
            }
        \end{xy}
    \]
    This diagram commutes in $\Delta_{/[1]}$ if and only if it commutes restricted to each individual $[1_{r}]$ with $r \in \{p+1 \leq q\}$. Since all maps preserve endpoints, this is clear for $r=p+1$ and $r=q$. In the other cases, it suffices to show that $\Bar{g}'$ sends the endpoints of $[n_{r}]$ to the endpoints of $[m_{r}]$. This follows since $\Bar{g}'$ is induced by a morphism $g$ in $\Delta^{\star}_{1}$.

    Since all morphisms preserve endpoints, we can extend the diagram by taking star products with the morphisms $\id_{d_{n_{l}}}$, $\id_{d_{n_{r}}}$, $\id_{h_{l_{l}}}$, $\id_{h_{l_{r}}}$, $\alpha|_{d_{n_{l}}}:d_{n_{l}} \xrightarrow{} d_{l_{l}}$ and $\alpha|_{d_{n_{r}}}:d_{n_{r}} \xrightarrow{} d_{l_{r}}$
    \[
        \begin{xy}
            \xymatrix{Z = \{i,j\} \subset d_{n_{l}} \ast d_{n_{c}} \ast d_{n_{r}} \ar[r] \ar[d] & h_{l_{l}} \ast h_{l_{c}} \ast h_{l_{r}} \\
            Z_{N} = \{0,k\} \subset d_{n_{l}} \ast h_{\{0,l^{1}\}} \ast e_{k} \ast h_{\{0,l^{2}\}} \ast d_{n_{r}} \ar[r] & h_{l_{l}}\ast h_{l^{1}} \ast e_{m} \ast h_{l^{2}} \ast h_{l_{r}} \ar[u]
            }
        \end{xy}
    \]
    By construction, this diagram defines a morphism in $\Omega_{1}$ covering $g$. We call this morphism $\mu_{Z,N}$. 

   Let us show that $\mu_{Z,N}$ is initial. Suppose, we are given another morphism in $\Omega_{1}$
    \[
        \begin{xy}
            \xymatrix{Z = [i,j] \subset d_{n} \ar[r]^-{\alpha} \ar[d]^{\rho} & h_{l} \\
            X= [0,k] \subset x_{a} \ar[r]_-{\beta} &  y_{b} \ar[u]^-{w} 
            }
        \end{xy}
    \]
    covering $g$. We can decompose it as follows
    \begin{equation}\label{diag:square1}
    \begin{gathered}
        \begin{xy}
            \xymatrix{Z = [i,j] \subset d_{n_{l}} \ast d_{n_{c}} \ast d_{n_{r}} \ar[r]^-{\alpha} \ar[d]^{\rho} & h_{l_{l}} \ast h_{l_{c}} \ast h_{m_{r}} \\
            X = [0,k] \subset x_{a_{l}} \ast x_{a_{c}} \ast x_{a_{r}} \ar[r]_-{\beta} & y_{b_{l}} \ast y_{b_{c}} \ast y_{b_{r}} \ar[u]^{w} \\   
            }
        \end{xy}
    \end{gathered}
    \end{equation}
    It follows from \cite[Lem. A.3]{stern20212} that $\rho$ is uniquely determined except at the endpoints by $\gamma$. Therefore, the restriction of diagram Diagram~\eqref{diag:square1} to $d_{n_{c}}$ looks as follows:
    \[
        \begin{xy}
            \xymatrix{d_{n_{c}} \ar[r]^{h} \ar[d]^-{\rho} & h_{l_{c}} \\
            x_{a_{c}^{1}}\ast e_{k} \ast x_{a^{2}_{c}} \ar[r]^-{\beta_{c}^{1} \ast f_{N} \ast \beta_{c}^{2}} & y_{b_{c}^{1}}\ast e_{n} \ast y_{b_{c}^{2}} \ar[u]^-{w}\\
            }
        \end{xy}
    \]
    Every morphism $Z_{N} \xrightarrow{} X$ in $E$ commuting with the above morphism $Z \xrightarrow{} X$ and $\mu_{Z,M}$, must in particular restrict to a commutative diagram
    \begin{equation}\label{diag:2}
    \begin{gathered}
        \begin{xy}
            \xymatrix{d_{n_{c}}\ar[r] \ar[d] \ar[ddr] & h_{l_{c}} & \\
            h_{\{0,l^{1}\}} \ast e_{k} \ast h_{\{0,l^{2}\}} \ar[r]|\hole \ar[dr] & h_{l^{1}} \ast e_{m} \ast h_{l^{2}} \ar[u] & \\
            & x_{a_{c}^{1}} \ast e_{k} \ast x_{a_{2}^{c}} \ar[r] & y_{b_{c}^{1}} \ast e_{m} \ast y_{b_{c}^{2}} \ar[ul] \ar[uul] \\
            }
        \end{xy}
    \end{gathered}
    \end{equation}
    Moreover, since the bottom square is in $E$, it must restrict to a commutative diagram
    \[
        \begin{xy}
            \xymatrix{e_{k} \ar[r] \ar[d] & e_{m} \\
            e_{k} \ar[r] & e_{m} \ar[u]}
        \end{xy}
        \]
    It follows that the component $h_{\{0,l^{1}\}} \ast e_{k} \ast h_{\{0,l^{2}\}} \xrightarrow{} x_{a_{c}^{1}} \ast e_{k} \ast x_{a_{c}^{2}}$ is uniquely determined by the left hand triangle in Diagram~\ref{diag:2}. We can further decompose $w$, since it must restrict to $\Bar{g}$ on $e_{m}$, into
    \[
        w= w^{1}\ast g \ast w^{2}: y_{b_{c}^{1}} \ast e_{m} \ast y_{b_{c}^{2}} \xrightarrow{} h_{l^{1}} \ast h_{l_{c}^{m}} \ast h_{l^{2}}
    \]
    Hence, the component of the bottom morphism in the right hand triangle is uniquely determined and must be given by
    \[
        w^{1} \ast\id_{[m]} \ast w^{2}: y_{b_{c}^{1}} \ast e_{m} \ast y_{b^{2}_{c}} \xrightarrow{} h_{l^{1}} \ast e_{m} \ast h_{l^{2}}.
    \]
    We can now extend back to the full diagram
    \[
         \begin{xy}
            \xymatrix{d_{n_{l}}\ast d_{n_{c}} \ast d_{n_{r}} \ar[r] \ar[d] \ar[ddr] & h_{l_{l}} \ast h_{l_{c}} \ast h_{l_{r}} & \\
            d_{n_{l}} \ast h_{\{0,l^{1}\}} \ast e_{k} \ast h_{\{0,l^{1}\}} \ast d_{n_{r}} \ar[r]|\hole \ar[dr] & h_{l_{l}} \ast h_{l^{1}} \ast e_{m} \ast h_{l^{2}} \ast h_{l_{r}} \ar[u] & \\
            & x_{a_{l}} \ast x_{a_{c}^{1}} \ast e_{k} \ast x_{a_{2}^{c}} \ast x_{a_{r}} \ar[r] & y_{b_{l}} \ast y_{b_{c}^{1}} \ast e_{m} \ast y_{b_{c}^{2}} \ast y_{b_{r}} \ar[ul] \ar[uul] \\
            }
        \end{xy}
       \]
    Note that the vertical arrows in the back square are given by identities when restricted to $h_{l_{l}}$, $h_{l_{r}}$, $e_{m_{l}}$, and $e_{m_{r}}$. Hence, the bottom square is uniquely determined by the morphisms $h_{l_{l}} \xrightarrow{} x_{a_{l}}$, $h_{l_{r}} \xrightarrow{} x_{a_{r}}$, $y_{b_{l}} \xrightarrow{} e_{m_{l}}$, and $y_{b_{r}} \xrightarrow{} e_{m_{r}}$. So there is a unique morphism $Z_{N} \xrightarrow{} X$ with the desired properties. 
\end{proof}

\begin{proof}[Proof of Proposition~\ref{prop:localization}]
    The result follows from combining Lemma~\ref{lem:contr} and Lemma~\ref{lem:initial} with \cite[Lem.3.1.1.]{walde20212}.
\end{proof}

\bibliographystyle{alpha}
\bibliography{main.bib}

\end{document}